\documentclass[11pt,a4paper]{article}
\pdfoutput=1

\RequirePackage{amsmath}%
\RequirePackage{amsthm}%
\usepackage{amsfonts}%
\usepackage{amssymb}%
\usepackage{appendix}
\usepackage{lineno}
\usepackage{color}%
\usepackage{url}%
\usepackage[hidelinks]{hyperref}
\usepackage{floatflt,graphicx,graphics,subfig}
\usepackage{mathdots}
\usepackage{algpseudocode}
\usepackage{algorithm}

\theoremstyle{definition}
\newtheorem{example}{Example}
\newtheorem{remark}{Remark}

\theoremstyle{theorem}
\newtheorem{theorem}{Theorem}

\let\Re\relax
\DeclareMathOperator{\Re}{\mathrm{Re}}
\let\Im\relax
\DeclareMathOperator{\Im}{\mathrm{Im}}

\DeclareMathOperator{\spann}{\mathrm{span}}

\renewcommand{\i}{\mathrm{i}}
\newcommand{\bI}{{\bf I}}
\newcommand{\bM}{{\bf M}}
\newcommand{\bN}{{\bf N}}
\newcommand{\bJ}{{\bf J}}
\newcommand{\bK}{{\bf K}}
\newcommand{\bD}{{\bf D}}
\newcommand{\bE}{{\bf E}}

\newcommand{\DD}{{\mathbb D}}

\newcommand{\bt}{{\bf t}}

\newcommand{\bx}{{\bf x}}

\newcommand{\bvb}{{\bf b}}
\newcommand{\bc}{{\bf c}}
\newcommand{\bn}{{\bf n}}
\newcommand{\bT}{{\bf T}}

\newcommand{\Null} {{\rm Null}}

\begin{document}
\title{A boundary integral equation method for Steklov eigenvalue problems for smooth planar domains}

\author{Jamie Swan$^{\rm a}$, Mohamed M.S. Nasser$^{\rm a}$, Harri Hakula$^{\rm b}$ \&  Matti Vuorinen$^{\rm c}$}
	
\date{}
\maketitle
 	
\vskip-0.7cm %
\centerline{$^{\rm a}$Department of Mathematics, Statistics \& Physics, Wichita State University,} %
\centerline{Wichita, KS 67260-0033, USA}%
\centerline{\tt jmswan@shockers.wichita.edu, mms.nasser@wichita.edu}%

\vskip+0.1cm %
\centerline{$^{\rm b}$Aalto University, Department of Mathematics and System Analysis,}
\centerline{P.O. Box 11100, FI–00076 Aalto, Finland}
\centerline{\tt harri.hakula@aalto.fi}           %  \\
        
\vskip+0.1cm %				
\centerline{$^{\rm c}$Department of Mathematics and Statistics, University of Turku,} 
\centerline{FI-20014 Turku, Finland}
\centerline{\tt vuorinen@utu.fi}

\begin{abstract}
In this paper, we present a novel numerical method for computing the Steklov spectrum of bounded and unbounded simply connected planar domains with smooth boundaries. 
For the unit disk, the Dirichlet-to-Neumann operator can be written explicitly in terms of the classical conjugation operator.
We extend this framework to general simply connected domains by expressing the Dirichlet-to-Neumann operator in terms of the generalized conjugation operator, which is defined via the boundary integral equation with the generalized Neumann kernel. 
Combined with Fourier differentiation on an equidistant boundary grid, this approach leads to a dense algebraic eigenvalue problem for the boundary traces of the Steklov eigenfunctions. 
The resulting method is simple, highly accurate, and uses only boundary data. 
It treats interior and exterior problems in a unified manner and reconstructs eigenfunctions inside the domain using the Cauchy integral formula. 
Numerical experiments on benchmark domains and parameter-dependent smooth families, including ellipses and star-like curves, demonstrate high accuracy and illustrate how the Steklov spectrum changes with geometry.
\end{abstract}

\begin{center}
\begin{quotation}
\end{quotation}
\end{center}

\begin{center}
\begin{quotation}
\end{quotation}
\end{center}

%-------------------------------------------------------------------
\section{Introduction}

The Steklov eigenvalue problem for the Laplacian on a planar domain
$G \subset \mathbb{C}$ consists of finding nontrivial pairs $(u,\lambda)$
of $C^2$-functions $u:G \to \mathbb{R}$  and real numbers $\lambda$ such that
\begin{equation}\label{stekProb}
\Delta u = 0 \quad \text{in } G, \qquad
\frac{\partial u}{\partial \bn} = \lambda u \quad \text{on } \Gamma = \partial G,
\end{equation}
where $\frac{\partial u}{\partial\bn}$ is the exterior normal derivative. 
This paper treats simply connected planar domains with smooth boundaries.
The eigenvalues satisfy
\[
0 = \lambda_0 < \lambda_1 \le \lambda_2 \le \cdots \to \infty
\]
and define the Steklov spectrum of the domain $G$.
For both bounded and unbounded simply connected domains $G$, the Steklov eigenvalues have the following asymptotic behavior
\begin{equation}\label{eq:asym_b}
\lambda_{2k-1} = \frac{2\pi k}{|\Gamma|} + \epsilon_k, \quad \lambda_{2k} = \frac{2\pi k}{|\Gamma|} + \epsilon_k'
\end{equation}
where $|\Gamma|$ is the length of the boundary $\Gamma$ and $\epsilon_k,\epsilon_k'\to0$ as $k\to \infty$~\cite{Bun,Dit}.

An equivalent formulation to \eqref{stekProb} is obtained through the Dirichlet-to-Neumann (DtN)
operator, which maps boundary Dirichlet data to the corresponding Neumann
traces of its harmonic extension. The Steklov eigenvalues coincide with the
spectrum of this operator. The problem is therefore a boundary spectral problem
for a first-order pseudodifferential operator. 
The DtN map is a standard object in the theory of elliptic boundary value problems, playing a central role in spectral geometry and inverse boundary problems~\cite{HofmannEgert,kuz,LeeUhlmann,Lee,TaylorPDE}.
Physically, the Steklov problem models systems where the dynamics are
driven by interactions at the interface. In the theory of heat
conduction, it describes the cooling of a solid body immersed in a fluid
bath at zero temperature, where the heat flux across the boundary is
proportional to the temperature difference, and the cooling rate is the
eigenvalue. In fluid mechanics, specifically in the study of sloshing,
the Steklov eigenvalues correspond to the natural frequencies of the
free surface of an ideal fluid contained in a vessel. The boundary
condition in this context represents the kinematic and dynamic
constraints at the free surface under the influence of gravity.

Sharp analytical results are available for Steklov eigenvalues under geometric
constraints. For simply connected planar domains with fixed perimeter, the disk
maximizes the first nonzero Steklov eigenvalue \cite{Henrot,Weinstock}. More general
inequalities of Hersch--Payne--Schiffer type relate low-order Steklov eigenvalues
to boundary length and area \cite{PolterovichSurvey,HerschPayneSchiffer}. These results indicate a strong
dependence of the Steklov spectrum on boundary geometry.

The Steklov problem is well suited to
boundary-based discretizations, since both the governing equation and the
spectral condition are imposed on the boundary. For the unit disk, the DtN
operator admits an explicit representation in terms of the classical conjugation operator
(or Hilbert transform) $\bK$, leading to closed-form expressions for Steklov
eigenvalues and eigenfunctions~\cite{Duren}. Extending this structure to general
simply connected domains requires a more general operator formulation.

The central question addressed in this paper is the following: Can the DtN operator for simply connected domains be approximated directly on the boundary by a method analogous to that for the unit disk? In this work, we develop such a method for bounded and unbounded simply connected planar domains with smooth boundaries. The method relies on the generalized conjugation operator $\bE$, which extends the classical conjugation operator $\bK$ from the unit disk to general simply connected domains, providing explicit boundary relations for harmonic conjugates~\cite{Nas-E}. 
When combined with Fourier-based differentiation, this framework provides a concise, direct representation of the Dirichlet-to-Neumann (DtN) operator on the boundary.
The resulting formulation leads to a dense algebraic eigenvalue problem whose spectrum approximates the Steklov spectrum. In addition to the eigenvalues, this approach produces boundary traces of the Steklov eigenfunctions as well as the boundary values of analytic functions whose real parts are the eigenfunctions. 
The interior values of these analytic functions, and consequently those of the eigenfunctions, can then be computed via the Cauchy integral formula. The method applies to both interior and exterior domains while completely avoiding volumetric discretization. For analytic boundaries, the numerical experiments reported below demonstrate exponential convergence.

We first analyze the unit disk, where the resulting formulas can be checked
against the exact Steklov spectrum, and then extend the construction to general
simply connected domains through the operator $\bE$. This allows the main
scientific question to be answered in a clean progression: the disk provides the
model calculation, the generalized conjugation operator provides the analytical
extension, and the numerical examples show how the resulting method behaves for
interior and exterior smooth geometries.

We assume that the boundary $\Gamma=\partial G$ is a smooth Jordan curve
with positive orientation, i.e., $\Gamma$ is oriented so that $G$ is always on the left of $\Gamma$. Hence, $\Gamma$ has a counterclockwise orientation for bounded $G$ and clockwise
orientation for unbounded $G$.
The boundary curve is parametrized by a
$2\pi$-periodic complex function $\eta(t)$ with $0\leq t \leq 2\pi$. We
assume $\eta(t)$ is twice continuously differentiable and the first
derivative  $\eta'(t) \neq 0$~\cite{Weg-Nas}. 
Let  $H$ denote the space of all H\"older continuous $2\pi$-periodic real functions.
Due to the smoothness of the boundary parametrization $\eta(t)$, any H\"older continuous function $\hat{\gamma}$ on $\Gamma$ can be identified with a H\"older continuous function $\gamma$ of the parameter $t$ via $\gamma(t)=\hat{\gamma}(\eta(t))$ for $0\le t\le 2\pi$, and vice versa.

Given a function $\gamma\in H$, it is well known that there exists a unique harmonic function $u$ in $G$, bounded at infinity if $G$ is unbounded, satisfying the boundary condition $u(\eta(t))=\gamma(t)$ for $0\le t\le 2\pi$. The DtN operator $\bT$ maps the function $\gamma$ to its normal derivative via $\bT\gamma(t)=\left.\frac{\partial u}{\partial \bn}\right|_{\eta(t)}$. Consequently, the Steklov eigenvalue problem~\eqref{stekProb} can be written as 
\[
\bT\gamma=\lambda\gamma,
\]
meaning that the Steklov eigenvalues coincide with the spectrum of the operator $\bT$. 
While the DtN map can be computed via the boundary integral equation with the \textit{adjoint} generalized Neumann kernel~\cite{Nas-adj}, we instead employ the generalized conjugation operator in this paper. This approach utilizes the integral equation with the generalized Neumann kernel and leads to an elegant explicit formula for the DtN operator $\bT$, as shown in Equations~\eqref{eq:DtN-K} and~\eqref{eq:DtN-E} below.

Recall that the harmonic function $u(z)$, where $z=x+\i y\in G$, can be interpreted as the real part of an analytic function $f(z)$ in $G$. 
To derive the explicit formulas~\eqref{eq:DtN-K} and~\eqref{eq:DtN-E} for the DtN operator $\bT$, we express the outward normal derivative of $u(z)$ in terms of the derivative $f'(z)$. 
Note that the complex unit tangent vector to $\Gamma$ is given by $\eta'(t)/|\eta'(t)|$ for $0 \leq t \leq 2\pi$. 
Since $\Gamma$ has a positive orientation, the outward unit normal vector is represented by $\bn(\eta(t)) = -\i\eta'(t)/|\eta'(t)|= e^{\i \theta (\eta (t))}$, where $\theta (\eta (t))$ is the angle between the positive real axis and $\bn(\eta(t))$ for $0 \leq t \leq 2\pi$. The outward normal derivative of $u(z)$ can then be written as 
\[
\frac{\partial u}{\partial \bn} = \nabla u \cdot \bn = u_x \cos \theta + u_y \sin \theta = \Re [e^{\i \theta}(u_x-\i u_y)].
\]
By the Cauchy-Riemann equations, $f'(z) = u_x(z)-\i u_y(z)$, which yields
\begin{equation}\label{eq:par-u-n} 
\left.\frac{\partial u}{\partial \bn}\right|_{\eta(t)} =
\Re\left[\frac{-\i\eta'(t)}{|\eta'(t)|}f'(\eta(t))\right], \quad
0 \leq t \leq 2\pi.
\end{equation}

\subsection{Related Work}

The present work lies at the intersection of Steklov spectral geometry and
boundary integral equations. On the analytical
side, the Steklov problem has been studied extensively through its relation to
the DtN operator, isoperimetric inequalities, and inverse
boundary questions. For simply connected planar domains, classical results such
as the Weinstock inequality and the Hersch--Payne--Schiffer inequalities show
that the low-order Steklov spectrum is strongly constrained by boundary
geometry~\cite{PolterovichSurvey,Henrot,HerschPayneSchiffer,Weinstock}. These
results motivate the numerical study of how the spectrum changes under smooth
deformations of the boundary.

From the computational point of view, Steklov eigenvalues have been approximated by a range of finite element, boundary element, and conformal-mapping-based methods. 
The work of Alhejaili and Kao~\cite{Alhe} is particularly relevant to this paper, since it also exploits the interaction between harmonic functions and complex analysis. 
However, the method presented in~\cite{Alhe} requires knowing a conformal mapping from the unit disk $\DD$ onto the simply connected domain $G$ (see also~\cite{IN}).
For the examples presented in~\cite{Alhe,IN}, the simply connected domains are chosen as images of the unit disk under known conformal mappings. 
To use these conformal mapping methods for a given simply connected domain $G$, it is required to numerically compute the conformal mapping from $\DD$ onto $G$ as a first step, followed by computing the eigenvalues in the second step. 

Our approach is different in that it does not require knowing or computing conformal mappings. Instead, it uses the generalized conjugation operator $\bE$ directly on the physical boundary, leading to a boundary-only formulation that treats bounded and unbounded domains within the same operator framework. Therefore, the method developed here should be viewed as complementary to finite element and high-order boundary element approaches rather than as a universal replacement.
Finite element methods are flexible with respect to geometry,
coefficients, and lower-regularity boundaries, and they remain the natural tool
when one needs volumetric discretization or local mesh refinement. High-order
boundary element methods retain the boundary-only character of Laplace problems
while also accommodating a broader range of geometries and approximation spaces
\cite{BabuskaOsborn,SauterSchwab,SchwabBook}.

The operator-theoretic ingredients come from the theory of the boundary integral equation with the generalized
Neumann kernel and the generalized conjugation operator developed in
\cite{Nas-E,Nas-AMC,Weg-Nas}. Those works provide the analytic machinery needed
to recover harmonic conjugates from boundary data on smooth domains. The present
paper combines that machinery with the Steklov boundary condition and Fourier
differentiation to obtain a practical eigenvalue solver for smooth interior and
exterior Steklov problems.
The method uses
dense matrices in its current form, and is not intended in this paper to address
polygonal or corner singularities. Those issues require a different analytical
and numerical treatment and are therefore outside the scope of the present
smooth-boundary study.

\subsection{Organization}

Section~\ref{sec:disk} treats the unit disk and uses the classical conjugation
operator $\bK$ to derive the model eigenvalue formulation in a setting where the exact
Steklov spectrum is known. Section~\ref{sec:E} introduces the generalized
conjugation operator $\bE$ for smooth simply connected domains.
Section~\ref{sec:sim} combines $\bE$ with the Steklov boundary condition to
derive the algebraic eigenvalue problem used in computation. The numerical
examples in Section~\ref{sec:ex} first validate the method on benchmark
geometries, then examine how the spectrum varies in smooth
parameter-dependent families, discuss computational complexity, and conclude
with an independent comparison against an available $hp$-FEM solver.
Section~\ref{sec:con} concludes with the main computational message of the
paper, and the appendix records the differentiation matrix used in the
discretization.

\section{The conjugation operator \texorpdfstring{$\bK$}{K}}\label{sec:disk}

We begin with the unit disk, where the Steklov spectrum is known exactly, to establish the method in a setting that allows for direct error verification before extending it to general simply connected domains with smooth boundaries. We assume the unit circle is parametrized by $\eta(t)=e^{\i t}$ for $0\le t\le 2\pi$.

The conjugation operator $\bK$ is defined on the H\"older space $H$ by~\cite{Hen,Weg}  
\begin{equation}\label{eq:opK}
\bK \gamma(s)=  \frac{1}{2\pi}\int_0^{2\pi} \cot \left(\frac{s - t}{2}\right)\gamma(t)dt, \quad s\in[0,2\pi],
\end{equation}
where the integral is understood in the Cauchy principal value sense. The operator $\bK$ is also known as the Hilbert transform; see~\cite[p.~103]{Hen} and~\cite[p.~22]{Wegert}. Note that the Hilbert transform defined in~\cite[Eq.~(1.14)]{Wegert} has opposite sign to the conjugation operator $\bK$ defined in~\eqref{eq:opK}. We also define the operator $\bJ$ by
\begin{equation}\label{eq:opJ}
\bJ\gamma(s)=  \frac{1}{2\pi}\int_0^{2\pi} \gamma(t)dt, \quad s\in[0,2\pi].
\end{equation}

The operators $\bK$ and $\bJ$ satisfy
\begin{equation}\label{eq:KJ}
\bJ^2=\bJ, \quad \bK\bJ=\bJ\bK=0, \quad \bK^2=-\bI+\bJ, \quad \bK^3=-\bK,
\end{equation}
where $\bI$ is the identity operator, see~\cite{Hen,Nas-E,Nas-E2,Wegert}. 
In addition, $\bK$ has an $L_{2}$ operator norm of $\|\bK\|_2=1$~\cite{Hen}. Here, the $L_{2}$ norm is defined by $\|\gamma\|_2=\sqrt{(\gamma,\gamma)}$, and the underlying inner product is given by
\[
(\phi,\psi)=\frac{1}{2\pi}\int_{0}^{2\pi}\phi(t)\psi(t)dt.
\]
Furthermore, the spectrum of $\bK$ consists of exactly three eigenvalues: $\{0,\pm\i\}$. The eigenspace corresponding to $0$ consists of the constant functions, i.e., $\Null(\bK)=\spann\{1\}$. In contrast, the eigenvalues $\pm\i$ have infinite-dimensional eigenspaces, which implies that $\bK$ is not a compact operator.

\begin{theorem}[{\cite[p.~22]{Wegert}}]\label{th:K-RI}
Let $\gamma,\mu\in H$ such that $f(e^{\i t})=\gamma(t)+\i\mu(t)$ for $0\le t\le 2\pi$ are the boundary values of a function $f(z)$ analytic in the unit disk, with $f(0)=c_r+\i c_i$ for real constants $c_{r}$ and $c_{i}$. Then, the functions $\gamma$ and $\mu$ satisfy
\begin{equation}\label{eq:K-RI}
\mu = \bK \gamma+c_i, \quad \gamma=-\bK\mu+c_r.
\end{equation}
\end{theorem}

\begin{remark}\label{rem:K}
It follows from the definition of the operator $\bJ$ and from~\eqref{eq:KJ} that $\bJ\mu=c_i$ and $\bJ\gamma=c_r$. Hence, Equation~\eqref{eq:K-RI} implies that $\mu = \bK \gamma+\bJ\mu$ and $\gamma=-\bK\mu+\bJ\gamma$. Consequently, determining $\mu$ from $\gamma$ via the conjugation operator $\bK$ requires knowing $\bJ\mu$. Similarly, using $\bK$ to determine $\gamma$ from $\mu$ requires knowing $\bJ\gamma$. 
\end{remark}

Recall that the Steklov eigenvalue problem for the unit disk is formulated as
\begin{equation}\label{eq:st-1}
\Delta u = 0 \quad \text{in} \quad \DD, \qquad \frac{\partial u}{\partial n} = \lambda u \quad \text{on} \quad C = \partial \DD.
\end{equation}
The Steklov eigenvalues of the unit disk, obtained via the separation of variables method, are $0, 1, 1, 2, 2, \ldots, k, k, \ldots$. The eigenfunction $u_0$ corresponding to $\lambda_0 = 0$ is a constant function. For the eigenvalues $\lambda_{2k-1} = \lambda_{2k} = k$ with $k \geq 1$, the corresponding eigenfunctions are $u_{2k-1} = r^k \sin (k\theta)$ and $u_{2k} = r^k \cos (k \theta)$~\cite{Alhe}.

The harmonic function $u(z)$ can be considered as the real part of a function $f(z)$ analytic in the unit disk $\DD$. We assume that the boundary values of $f(z)$ are given by
\begin{equation}\label{eq:f-bdv-1}
f(\eta(t))=\gamma(t)+\i\mu(t), \quad 0\le t\le 2\pi.
\end{equation}
Thus, the boundary values of the real part of $f(z)$ on the unit circle are $\gamma(t)=u(\eta(t))$. Differentiating both sides of Equation~\eqref{eq:f-bdv-1} with respect to the parameter $t$ yields
\begin{equation}\label{eq:f'-bdv-1}
\eta'(t) f'(\eta(t)) = \gamma'(t)+\i\mu'(t).
\end{equation}
Since $|\eta'(t)|=1$, using~\eqref{eq:par-u-n} with the parametrization of the unit circle and taking the imaginary part of Equation~\eqref{eq:f'-bdv-1}, we obtain
\begin{equation}\label{eq:partial-1}
\left.\frac{\partial u}{\partial n}\right|_{\eta(t)} =\Re\left[-\i\eta'(t)f'(\eta(t))\right] = \Im[\eta'(t) f'(\eta(t))] = \mu'(t).
\end{equation}
Let $\bD$ be the differentiation operator with respect to the parameter $t$, defined by $\bD\mu(t) = \mu'(t)$. In view of~\eqref{eq:partial-1}, the boundary condition in~\eqref{eq:st-1} implies
\begin{equation}\label{eq:mu-lam}
\bD\mu(t)=\lambda \gamma(t).
\end{equation}

Since we are interested only in computing the harmonic function $u(z)=\Re f(z)$, we can assume without loss of generality that $c_i=\Im f(0)=0$. Then, in view of the first equation in~\eqref{eq:K-RI}, we can rewrite Equation~\eqref{eq:mu-lam} as 
\begin{equation}\label{eq:bKp-lam_gam-2}
\bD\bK\gamma(t)=\lambda \gamma(t).
\end{equation}
That is, the Dirichlet-to-Neumann operator $\bT$ on the unit circle is given by 
\begin{equation}\label{eq:DtN-K}
\bT=\bD\bK.
\end{equation}

The eigenvalues $\lambda$ and the corresponding eigenfunctions $\gamma$ are then computed numerically from Equation~\eqref{eq:bKp-lam_gam-2}. The function $\gamma$ is first represented on a grid of $n$ equidistant points 
\begin{equation}\label{eq:tj}
t_j = (j-1)\frac{2\pi}{n}, \quad j=1,2,\ldots,n,
\end{equation}
where $n$ is an even positive integer. 
Let $\bt = [t_1,t_2,\ldots,t_n]^T$ be an $n \times 1$ vector, where $T$ denotes the transpose. The $n\times1$ vector $\gamma(\bt)$ is then obtained by evaluating the function $\gamma(t)$ componentwise at the grid points $t_j$ for $j=1,2,\ldots,n$.

As explained in Appendix~A, the operator $\bD$ can be discretized by the matrix
\begin{equation}\label{eq:D_eq_FMF}
D = FWF^\ast,
\end{equation}
where $F$ is the $n \times n$ discrete Fourier transform matrix, $F^\ast$ is its Hermitian transpose, and $W$ is the $n\times n$ block matrix
\begin{equation}\label{eq:mat-W}
W = -\frac{\i}{n} \left[\begin{array}{cccc}
	~~0~~   & ~~0~~    &  ~~0~~   & 0  \\
	 0  & R  & 0  & 0  \\
	0    &0 & 0 & 0 \\
    0  & 0 & 0 &  -JRJ    \\
\end{array}\right]
\end{equation}
composed of the $(n/2-1)\times(n/2-1)$ matrices 
\[
R = \left[\begin{array}{cccc}
	~~1~~   & ~~0~~    &\cdots &0     \\
	0   & 2    &\cdots &0     \\
	\vdots     &\vdots      &\ddots &\vdots   \\
	0          & 0   &\cdots &n/2-1   \\
\end{array}\right], \quad
J = \left[\begin{array}{cccc}
~0~      &\cdots &~0~  &~1~   \\
0      &\cdots &1  &0     \\
\vdots        &\iddots  &\vdots    &\vdots   \\
1             &\cdots &0 &0   \\
\end{array}\right].
\]
The matrix $F$ can be computed using the MATLAB function {\tt dftmtx}.
The conjugation operator $\bK$ is discretized by the $n \times n$ Wittich matrix~\cite{Gai}, whose elements are given by
\begin{equation}\label{eq:K-mat}
(K)_{ij} = \begin{cases}
    0, \quad \text{if }~i-j~\text{ is even,}\\
    \frac{2}{n} \cot \frac{(i-j)\pi}{n}, \quad \text{if }~i-j~\text{ is odd.}
\end{cases}
\end{equation}
Thus, Equation~\eqref{eq:bKp-lam_gam-2} is discretized as the $n\times n$ algebraic eigenvalue problem
\begin{equation}\label{eq:DK-mat}
DK \gamma(\bt) = \lambda \gamma(\bt).
\end{equation}

Note that $0$ is an eigenvalue of both the operator $\bK$ and its discrete counterpart $K$. However, their null-spaces have different dimensions: $\dim(\Null(\bK))=1$, whereas $\dim(\Null(K))=2$. The operator $\bK$ has only the constant function as an eigenfunction corresponding to the eigenvalue $0$. In contrast, the matrix $K$ possesses two linearly independent eigenvectors for the eigenvalue $0$: the constant vector and the vector $\cos(\frac{n}{2}\bt)$, which is a sequence of alternating $+1$ and $-1$~\cite{Weg}. As a result, $0$ becomes an eigenvalue of the product matrix $DK$ with multiplicity $2$. We denote these eigenvalues as $\lambda_{-1}=\lambda_0=0$. This labeling keeps $\lambda_1$ as the first nonzero eigenvalue, consistent with the continuous problem, while accounting for the spurious zero eigenvalue introduced by the discretization.

We are interested in computing approximate values of only the first $\hat{n}$ eigenvalues $\lambda_1,\ldots,\lambda_{\hat{n}}$ of the matrix $DK$ and their corresponding eigenvectors $\gamma_1(\bt),\ldots,\gamma_{\hat{n}}(\bt)$ for $\hat{n} \ll n$. These eigenvalues and eigenvectors can be computed in MATLAB using the function {\tt eigs}. To obtain the eigenvalues in increasing order of their magnitudes, we set the target option {\tt sigma} to {\tt 'smallestabs'}. Furthermore, since $\lambda=0$ is a known eigenvalue of Problem~\eqref{eq:st-1}, applying {\tt eigs} directly to the singular matrix $DK$ can cause numerical difficulties. Therefore, we use {\tt eigs} to compute the eigenvalues of the shifted matrix $DK+I$, where $I$ is the identity matrix, and then subtract $1$ from the resulting values. This shifting procedure regularizes the system without modifying the corresponding eigenvectors. For each $k$, let $\lambda_{k,n}$ be the computed eigenvalue obtained with $n$ grid points, and let $\lambda_{k}$ be the exact eigenvalue. The relative error for each eigenvalue is given by
\begin{equation}\label{eq:ea_err}
 \dfrac{|\lambda_{k,n} - \lambda_{k}|}{\lambda_{k}}, \quad k=1,2,3, \ldots .
\end{equation}

The relative error for the first $10$ nonzero Steklov eigenvalues as a function of the number of grid points $n$ is illustrated in Figure~\ref{fig:ell_err}(a) for the baseline case $r=1$, which corresponds to the unit disk. It is clear that the error falls below $10^{-13}$ for $n \ge 20$. The corresponding eigenfunctions for the disk are likewise displayed in the first panel of Figure~\ref{fig:ex_ell_B}, again for $r=1$.

\section{The generalized conjugation operator \texorpdfstring{$\bE$}{E}}\label{sec:E}

The conjugation operator $\bK$ for the unit disk was generalized in~\cite{Nas-E} to an operator $\bE$ defined on simply connected domains with smooth boundaries via the boundary integral equation with the generalized Neumann kernel. This operator provides the key analytical foundation for extending the disk method to more general geometries.

Let $G$ be a bounded or unbounded simply connected domain in the $z$-plane, and let $\Gamma = \partial G$ be a smooth Jordan curve oriented such that $G$ lies to the left of $\Gamma$. The curve $\Gamma$ is parametrized by a twice continuously differentiable function $\eta(t)$ for $0 \leq t \leq 2\pi$ satisfying $\eta'(t) \neq 0$ on $[0,2\pi]$. We define a complex-valued function $A(t)$ for $t\in [0,2\pi]$ by
\begin{equation}\label{eq:A}
A(t)= \begin{cases}
    \eta(t)-\alpha, \quad &\text{if } G \text{ is bounded,}\\
    1, \quad &\text{if } G \text{ is unbounded,}
\end{cases}
\end{equation}
where $\alpha$ is a fixed point in $G$. We define the real kernels $M(s,t)$ and $N(s,t)$ as the real and imaginary parts of
\[
M(s,t)+\i N(s,t)= \frac{1}{\pi} \frac{A(s)\eta'(t)}{A(t)(\eta(t)-\eta(s))}.
\]
The kernel $N(s,t)$ is known as the generalized Neumann kernel~\cite{Weg-Nas}. It is continuous, with diagonal values given by
\[
N(t,t) = \frac{1}{\pi} \Im \left ( \frac{1}{2} \frac{\eta''(t)}{\eta'(t)}-\frac{A'(t)}{A(t)} \right ).
\]
The kernel $M(s,t)$ can be decomposed as
\[
M(s,t) = - \frac{1}{2\pi} \cot \frac{s-t}{2} + \tilde{M}(s,t),
\]
where $\tilde{M}(s,t)$ is a continuous kernel with the diagonal values
\[
\tilde{M}(t,t) = \frac{1}{\pi} \Re \left ( \frac{1}{2} \frac{\eta''(t)}{\eta'(t)}-\frac{A'(t)}{A(t)} \right ).
\]

The integral operators associated with the kernels $N(s,t)$ and $M(s,t)$ are defined on the space $H$ by 
\begin{eqnarray}
\bN \gamma(s) &=& \int_0^{2\pi} N(s,t)\gamma(t) dt, \quad s \in [0,2\pi],\\
\label{eq:bM}
\bM \gamma(s) &=& \int_0^{2\pi} M(s,t)\gamma(t) dt, \quad s \in [0,2\pi].
\end{eqnarray}
where the singular integral in~\eqref{eq:bM} is understood in the Cauchy principal value sense. The operator $\bN$ is compact, whereas $\bM$ is a singular integral operator. Both $\bN$ and $\bM$ are bounded operators mapping $H$ into itself. For further details, see~\cite{Weg-Nas}.

Two boundary integral equations were derived in~\cite{Weg-Nas} for simply connected domains: one for bounded domains~\cite[Eq.~(32)]{Weg-Nas} and another for unbounded domains~\cite[Eq.~(34)]{Weg-Nas}. 
Reversing the orientation of the boundary $\Gamma$ inverts the signs of both $\bM$ and $\bN$ in these equations. 
For bounded domains, we adopt a counterclockwise orientation, matching~\cite{Weg-Nas}, which ensures that the integral equation we use in this paper is the same as the integral equation in~\cite[Eq.~(32)]{Weg-Nas}. 
For unbounded domains, however, while~\cite{Weg-Nas} assumes a counterclockwise orientation, we assume a clockwise orientation to invert the signs of $\bM$ and $\bN$ in~\cite[Eq.~(34)]{Weg-Nas} to get the same integral equation as for bounded domains. 
Consequently, we obtain a unified integral equation valid for both bounded and unbounded domains, allowing us to combine the two separate statements of~\cite[Theorem~2]{Weg-Nas} into a single theorem as follows.

\begin{theorem}\label{th:ie}
Let $\gamma,\mu \in H$ be such that the formula 
\begin{equation}\label{eq: thm2_1}
A(t)g(\eta(t)) =  \gamma(t)  + \i\mu(t)
\end{equation}
defines the boundary values of a function $g$ analytic in $G$, with $g(\infty)=0$ if $G$ is unbounded. Then, the function $\mu$ satisfies the integral equation
\begin{equation}\label{eq:ie}
(\bI - \bN) \mu = -\bM \gamma.
\end{equation}
\end{theorem}

Furthermore, the function $A$ in this paper is explicitly given by~\eqref{eq:A}. For this choice of $A$, the following theorem follows from~\cite[Theorems~5 and~10]{Weg-Nas} and~\cite[Theorem~1]{Nas-AMC}.

\begin{theorem}\label{th:null}
The null-spaces of the operators $\Null(\bI\pm\bN)$ and $\bM$ are given by
\begin{eqnarray}
\label{eq:null-}
\Null(\bI-\bN)&=&\{0\},\\
\label{eq:null+}
\Null(\bI+\bN)&=&\Null(\bM)=\spann\{1\}.
\end{eqnarray}
\end{theorem}

Since $\Null(\bI-\bN)=\{0\}$, the Fredholm alternative implies that the solution of the integral equation~\eqref{eq:ie} is unique. Furthermore, it follows that $(\bI-\bN)^{-1}$ exists and is bounded~\cite{Atk}. Because the operator $\bM$ is also bounded on $H$~\cite{Weg-Nas}, the integral operator $\bE$ defined on $H$ by 
\[
\bE = - (\bI - \bN) ^{-1} \bM
\]
is bounded. The operator $\bE$ reduces to the conjugation operator $\bK$ when the boundary $\Gamma$ is the unit circle; for this reason, $\bE$ is called the generalized conjugation operator~\cite{Nas-E}. Furthermore, $\bE$ shares several properties analogous to those of $\bK$~\cite{Nas-E,Nas-E2}. For instance, the eigenvalues of $\bE$ are exactly $\{0,\pm\i\}$, where $\Null(\bE)=\spann\{1\}$ and the eigenspaces corresponding to $\pm\i$ are infinite-dimensional which implies that the operator $\bE$ is not compact~\cite{Nas-E,Nas-E2}.

To present additional properties of the operator $\bE$, let $\Psi$ be the conformal mapping from the bounded simply connected domain $G$ onto the unit disk $\DD$, normalized by $\Psi(\alpha)=0$ and $\Psi'(\alpha)>0$. We also let $\tau(t)$ be the boundary correspondence function of $\Psi$, defined by $\Psi(\eta(t))=e^{\i \tau(t)}$ for $0\le t\le 2\pi$. The function $\tau(t)$ is strictly increasing on $[0,2\pi]$ such that $\tau(t)-t$ is $2\pi$-periodic~\cite{Hen}. Consequently, $\tau'(t)$ is strictly positive and satisfies $\int_{0}^{2\pi}\tau'(t)\,dt=2\pi$. The boundary correspondence function $\tau(t)$ for the conformal mapping from the unbounded simply connected domain $G$ onto the exterior of the unit disk can be defined in a similar manner~\cite{Hen}.

With the help of the boundary correspondence function $\tau(t)$, we define a weighted $L^2$ norm $\|\cdot\|_{\tau}$ by 
\[
\|\gamma\|_\tau = \|\sqrt{\tau'}\,\gamma\|_2,
\] 
and an operator $\bJ_{\tau}$ by
\begin{equation}\label{eq:opJt}
\bJ_{\tau}\gamma(s)=  \frac{1}{2\pi}\int_0^{2\pi} \tau'(t)\gamma(t)\,dt, \quad s\in[0,2\pi].
\end{equation}
The operators $\bE$ and $\bJ_{\tau}$ then satisfy~\cite{Nas-E,Nas-E2}
\begin{equation}\label{eq:EJt}
\bJ_{\tau}^2=\bJ_{\tau}, \quad \bE\bJ_{\tau}=\bJ_{\tau}\bE=0, \quad \bE^2=-\bI+\bJ_{\tau}, \quad \bE^3=-\bE,
\end{equation}
and the operator $\bE$ has the norm $\|\bE\|_{\tau}=1$.

The following theorem generalizes Theorem~\ref{th:K-RI} from the unit disk to general simply connected domains.

\begin{theorem}\label{th:E-RI}
Let $\gamma,\mu\in H$ be such that 
\[
f(\eta(t))=\gamma(t)+\i\mu(t),\quad 0\le t\le 2\pi, 
\]
defines the boundary values of a function $f(z)$ analytic in the simply connected domain $G$, with $f(\alpha)=c_r+\i c_i$ if $G$ is bounded, and $f(\infty)=c_r+\i c_i$ if $G$ is unbounded, where $c_r$ and $c_i$ are real constants. Then, the functions $\gamma$ and $\mu$ satisfy
\begin{equation}\label{eq:E-RI}
\mu = \bE \gamma+c_i, \quad \gamma=-\bE\mu+c_r.
\end{equation}
\end{theorem}
\noindent\textbf{Proof.} 
Given the function $f(z)$ in the theorem, we define an auxiliary function $g(z)$ in $G$ by $g(z)=(f(z)-f(\alpha))/(z-\alpha)$ if $G$ is bounded, and by $g(z)=f(z)-f(\infty)$ if $G$ is unbounded. Then, $g(z)$ is analytic in $G$ with $g(\infty)=0$ when $G$ is unbounded, and its boundary values are given by
\[
A(t)g(\eta(t))=\gamma(t)-c_r+\i(\mu(t)-c_i),
\]
where the function $A$ is defined by~\eqref{eq:A}. Applying Theorem~\ref{th:ie} yields
\[
(\bI-\bN)(\mu-c_i)=-\bM(\gamma-c_r),
\]
and the definition of the operator $\bE$ implies that $\mu-c_i=\bE(\gamma-c_r)$. The first equation in~\eqref{eq:E-RI} then follows from the fact that $\bE c_r=0$ since $\Null(\bE)=\spann\{1\}$.

To prove the second equation, let $\hat{g}(z)=\i g(z)$. Then, $\hat{g}(z)$ is analytic in $G$ with $\hat{g}(\infty)=0$ when $G$ is unbounded, and the boundary values of $\hat{g}(z)$ are given by
\[
A(t)\hat{g}(\eta(t))=-(\mu(t)-c_i)+\i(\gamma(t)-c_r).
\]
Theorem~\ref{th:ie} then implies that 
\[
(\bI-\bN)(\gamma-c_r)= \bM(\mu-c_i),
\]
which simplifies to $\gamma-c_r=-\bE(\mu-c_i)$. Since $\bE c_i=0$, we obtain the second equation in~\eqref{eq:E-RI}. 
\hfill $\Box$

\begin{remark}\label{rem:E}
In analogy with Remark~\ref{rem:K} for the conjugation operator $\bK$, it follows from the definition of the operator $\bJ_{\tau}$ and from~\eqref{eq:EJt} that $\bJ_{\tau}\mu=c_i$ and $\bJ_{\tau}\gamma=c_r$. Hence, Equation~\eqref{eq:E-RI} implies that $\mu = \bE \gamma+\bJ_{\tau}\mu$ and $\gamma=-\bE\mu+\bJ_{\tau}\gamma$. That is, determining $\mu$ from $\gamma$ using the generalized conjugation operator $\bE$ requires knowing $\bJ_{\tau}\mu$. Similarly, using $\bE$ to determine $\gamma$ from $\mu$ requires knowing $\bJ_{\tau}\gamma$. 
\end{remark}

\section{Discretization and algorithm}\label{sec:sim}

With the help of the operator $\bE$ introduced in \S\ref{sec:E}, the method presented in \S\ref{sec:disk} for computing the eigenvalues for the unit disk is extended in this section to bounded and unbounded simply connected domains with smooth boundaries. 

Let $G$ be a simply connected domain with a boundary $\Gamma=\partial G$ as described in the previous section. We consider the Steklov eigenvalue problem
\begin{equation}\label{eq:st-2}
\Delta u = 0 \quad \text{in} \quad G, \qquad \frac{\partial u}{\partial \bn} = \lambda u \quad \text{on} \quad \Gamma.
\end{equation}
If $G$ is unbounded, the function $u$ is additionally required to satisfy $u(z) \to C$ as $|z| \to \infty$ for some constant $C$.

The harmonic function $u(z)$ can be expressed as the real part of a function $f(z)$ analytic in the domain $G$~\cite{Hen}. Let the boundary $\Gamma$ be parametrized by $\eta(t)$ for $0\le t\le 2\pi$, satisfying the assumptions from the previous section. Suppose the boundary values of $f(z)$ are given by
\begin{equation}\label{eq:f-bdv-2}
f(\eta(t))=\gamma(t)+\i\mu(t),
\end{equation}
where $\gamma(t) = u(\eta(t))$. We denote $f(\alpha)=c_r+\i c_i$ if $G$ is bounded, and $f(\infty)=c_r+\i c_i$ if $G$ is unbounded. Then, Theorem~\ref{th:E-RI} implies that
\begin{equation}\label{eq:mu-Egam2}
\mu(t) = \bE \gamma(t)+c_i, \quad \gamma(t) = -\bE \mu(t)+c_r.
\end{equation}
Since we are only interested in computing the harmonic function $u(z)=\Re f(z)$, we can set $c_i=0$ without loss of generality, which simplifies the first equation in~\eqref{eq:mu-Egam2} to
\begin{equation}\label{eq:mu-Egam}
\mu(t) = \bE \gamma(t).
\end{equation}

By replacing the outward normal derivative in equation~\eqref{eq:par-u-n} using the boundary condition of equation~\eqref{eq:st-2}, we obtain
\begin{equation}\label{eq:lam_etp_u}
    \lambda  u(\eta (t)) = \Re\left[-\i \frac{\eta '(t)}{|\eta'(t)|} f'(\eta(t))\right] 
		= \frac{1}{|\eta'(t)|}  \Im\left[\eta '(t) f'(\eta(t))\right].
\end{equation}
Differentiating both sides of~\eqref{eq:f-bdv-2} with respect to the parameter $t$ yields
\[
\eta'(t) f'(\eta(t)) = \gamma'(t)+\i\mu'(t).
\]
Then~\eqref{eq:lam_etp_u} implies that
\begin{equation}\label{eq:mu-lam-2}
    \mu'(t) = \lambda |\eta'(t)| \gamma(t).
\end{equation}
In view of~\eqref{eq:mu-Egam}, equation~\eqref{eq:mu-lam-2} can be written as 
\begin{equation}\label{eq:DE-lam}
  \bD\bE\gamma(t) = \lambda |\eta'(t)| \gamma(t),
\end{equation}
where $\bD$ is the differentiation operator. 
Let $\rho(t) = 1/|\eta'(t)|$. Thus equation~\eqref{eq:DE-lam} is rewritten as
\begin{equation}\label{eq:DE-lam-2}
  \rho(t)\bD\bE\gamma(t) = \lambda \gamma(t).
\end{equation}
Consequently, the Dirichlet-to-Neumann map of the function $\gamma(t)$ for the simply connected domain $G$ is given by 
\begin{equation}\label{eq:DtN-E}
  \bT\gamma(t)=\rho(t)\bD\bE\gamma(t).
\end{equation}

\begin{remark}
When $G$ is the unit disk and hence $\Gamma$ is the unit circle, the generalized conjugation operator $\bE$ reduces to the classical conjugation operator $\bK$~\cite{Nas-E}. Further, we will have $\rho(t)=1/|\eta'(t)|=1$. Hence, the Dirichlet-to-Neumann operator $\bT$ in~\eqref{eq:DtN-E} reduces to the operator $\bT$ in~\eqref{eq:DtN-K} and the method presented in this section reduces to the method presented in~\S\ref{sec:disk}.
\end{remark}

The operator $\bE$ is discretized by discretizing the integral equation~\eqref{eq:ie}. 
The function $\bM\gamma$ in the right-hand side of the integral equation~\eqref{eq:ie} is continuous since the function $\gamma$ is H\"older continuous. 
Further, since the integrals in~\eqref{eq:ie} are over $2\pi$-periodic functions, the integral equation~\eqref{eq:ie} can be best discretized by the Nystr\"om method with the trapezoidal rule~\cite{Atk}.
The stability and convergence of the Nystr\"om method is based on the compactness of the operator $\bN$ in the space of continuous functions equipped with the sup-norm, on the convergence of
the trapezoidal rule for all continuous functions, and on the theory of collectively
compact operator sequences~\cite{Atk}. 
The numerical solution of the integral equation will then converge with a similar rate of convergence as the trapezoidal rule~\cite[p.~322]{Atk}.	
The order of the convergence of the trapezoidal rule depends on the smoothness of the integrands in~\eqref{eq:DE-lam-2} which in turn depends on the smoothness of the boundary $\Gamma$.  
If the integrand is $q$ times continuously differentiable, then the rate of convergence of
the trapezoidal rule is $O(n^{-q})$. For analytic boundaries, the trapezoidal rule converges exponentially with $O(e^{-cn})$ where $c$ a positive constant~\cite{Tre-Wei}.

For any $2\pi$-periodic continuous function $\phi(t)$, the trapezoidal rule approximates the integral of $\phi(t)$ over the interval $[0,2\pi]$ by
\[
\int_{0}^{2\pi}\phi(t)dt \approx \frac{2\pi}{n}\sum_{k=1}^{n}\phi(t_k).
\]
We then use the trapezoidal rule to discretize the integrals in~\eqref{eq:ie}. The kernels of the operators  $\bN$ and $\tilde\bM$ are continuous. Hence, the Nystr\"om method with the trapezoidal rule can be used to find discretization matrices $B$ and $\tilde C$ of the operators $\bN$ and $\tilde\bM$, respectively, with the entries
\[
(B)_{ij} =  \frac{2\pi}{n}N(t_i,t_j),\quad
(\tilde C)_{ij} =  \frac{2\pi}{n}\tilde M(t_i,t_j).
\]
The operator $\bK$ is discretized by the $n \times n $ Wittich's matrix $K$ given by~\eqref{eq:K-mat}. 
Thus the integral operator $\bM$ is discretized by the $n \times n $ matrix $C$ with the entries
\[
(C)_{ij} = - (K)_{ij} + (\tilde C)_{ij},
\]
and the integral operator $\bE$ is then discretized by the $n \times n$ matrix 
\begin{equation}\label{eq:matE}
E = -(I - B)^{-1}C.
\end{equation}

Note that $0$ is a simple eigenvalue of both operators $\bK$ and $\bE$ with the constant function as the corresponding eigenfunction~\cite{Hen,Nas-E}. The other eigenvalues of $\bK$ and $\bE$ are $\pm\i$ where each of these eigenvalues have infinite number of corresponding eigenfunctions and hence both operators $\bK$ and $\bE$ are not compact~\cite{Hen,Nas-E}. 
Both discretization matrices $K$ and $E$ of the operators $\bK$ and $\bE$, respectively, have the eigenvalue $0$ with multiplicity $2$ with two linearly independent eigenvectors, namely, the vector whose elements are all $1$'s and the vector whose elements are a sequence of alternating of $+1$ and $-1$. All the remaining eigenvalues of the matrices are $\pm\i$ and each has the multiplicity $n/2-1$. Figure~\ref{fig:eig-ex3} shows the eigenvalues for matrices $K$ (left) and $E$ (center) for the domain $G_1$ of Example~\ref{ex:1} from~\S\ref{sec:ex}. Figure~\ref{fig:eig-ex3} (right) shows the absolute error $|\hat\lambda_{k}-\hat\lambda_{k,n}|$ for $k=1,2,\ldots,n$ with $n=2^{10}$ where, for each $k$, $\hat\lambda_{k}$ is the exact eigenvalue and $\hat\lambda_{k,n}$ is the numerically computed eigenvalue for both matrices $K$ and $E$. The numerical eigenvalues of both matrices $K$ and $E$ are computed using the MATLAB function {\tt eig}.

\begin{figure}[ht] %eig_e_k.m
\centerline{
\hfill
\scalebox{0.35}{\includegraphics{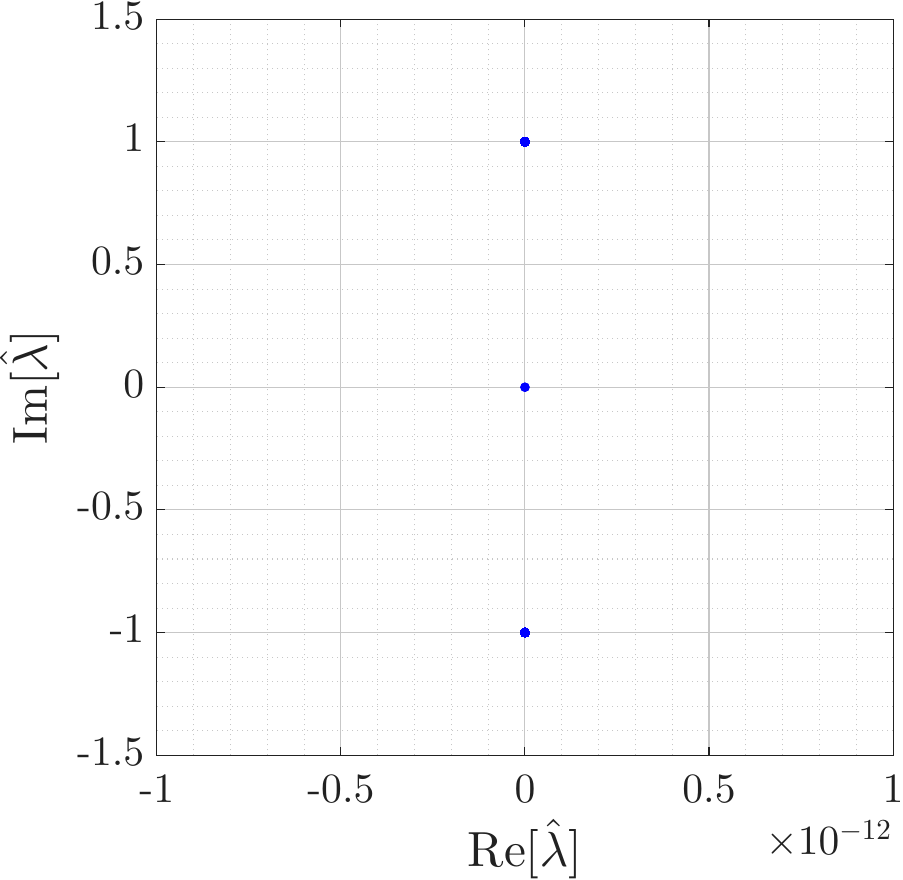}}
\hfill
\scalebox{0.35}{\includegraphics{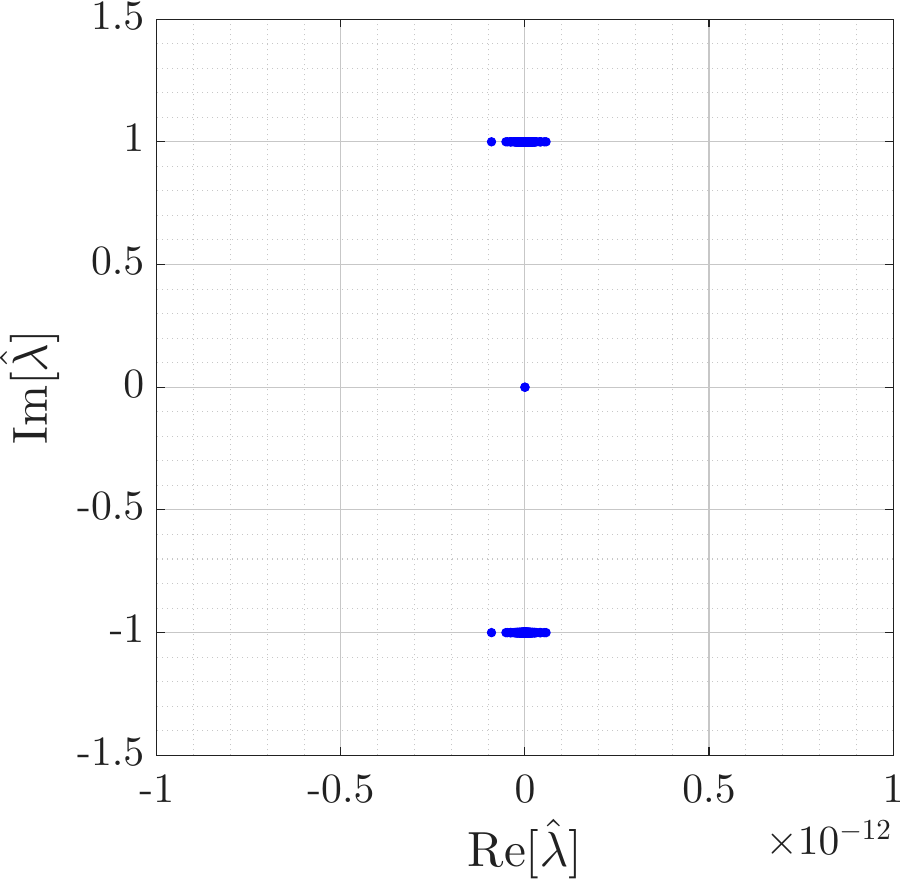}}
\hfill
\scalebox{0.35}{\includegraphics{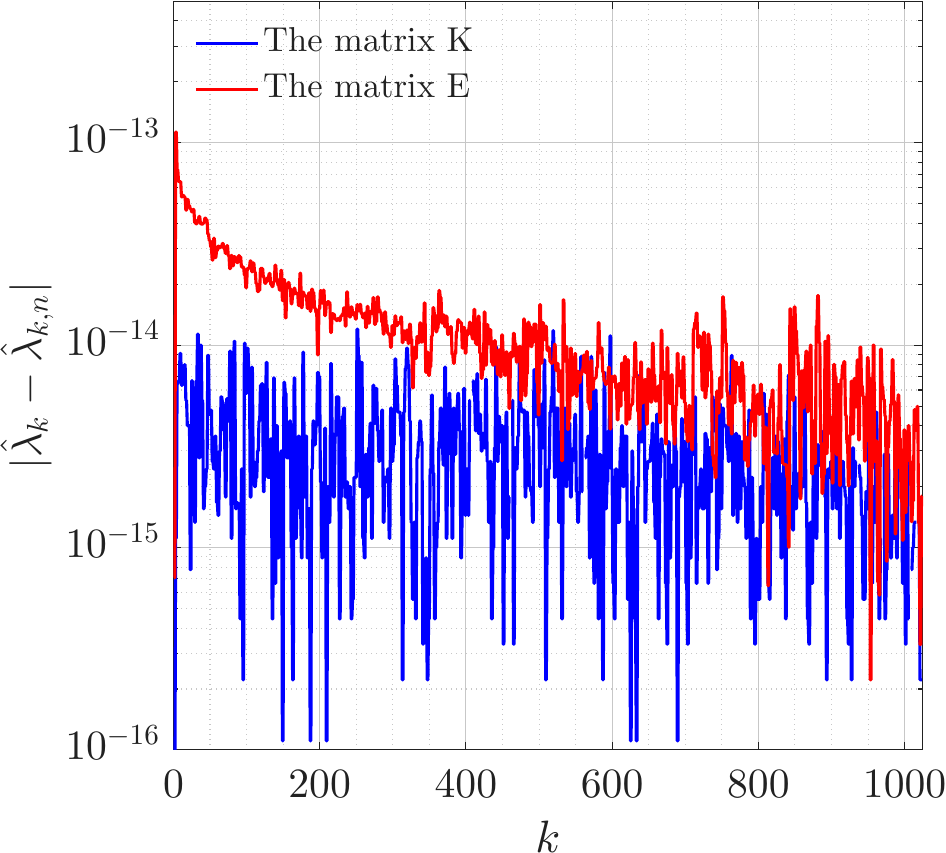}}
\hfill}
\caption{The eigenvalues of the matrices $K$ (left) and $E$ (center) for the bounded domain in Figure~\ref{fig:domains} (left) obtained with $n=2^{10}$. On the right, the absolute error of the numerically computed eigenvalue for both matrices $K$ and $E$.}
	\label{fig:eig-ex3}
\end{figure}

Let $P$ be the $n \times n $ diagonal matrix whose diagonal is the $n\times1$ vector $\rho(\bt)$.
Using~\eqref{eq:D_eq_FMF} and~\eqref{eq:matE}, equation~\eqref{eq:DE-lam-2} can be discretized to obtain the algebraic eigenvalue problem 
\begin{equation}\label{eq:matQ}
Q \bx = \lambda \bx, \quad Q=PFWF^\ast E, \quad \bx=\gamma(\bt).
\end{equation}
The algebraic eigenvalue problem~\eqref{eq:matQ} has $n$ eigenvalues which will be denoted by $\lambda_{j}$ for $j=-1,0,1,\ldots,n-2$, where $\lambda_{-1}=\lambda_{0}=0$. 

As discussed in~\S\ref{sec:disk}, we will compute approximate values of only the eigenvalues $\lambda_1,\ldots,\lambda_{\hat n}$ and their corresponding eigenvectors $\gamma_1(\bt),\ldots,\gamma_{\hat n}(\bt)$ for $\hat n<<n$. These approximate eigenvalues will be denoted by $\lambda_{1,n},\ldots,\lambda_{\hat n,n}$. We will use the MATLAB function {\tt eigs} to compute the eigenvalues of the $n\times n$ matrix $Q+I$ and then the eigenvalues of $Q$ are again obtained by subtracting $1$ from the computed eigenvalues.
The computed eigenvectors $\gamma_j(\bt)$ are normalized by
\[
\frac{2\pi}{n} \sum_{k=1}^{n} |\eta'(t_k)| \gamma_j^2(t_k)=1. 
\]
That is, we assume that $\int_0^{2\pi} \gamma_j(t)^2 | \eta'(t) | dt =1$.

Once the approximate eigenvector $\gamma_j(\bt)$ is computed, we compute 
\[
\mu_j (\bt) = E \gamma_j (\bt)
\]
and hence, by~\eqref{eq:f-bdv-2}, 
\[
f_j(\eta(\bt)) = \gamma_j(\bt) + \i\mu_j(\bt)
\]
is a discretization of a boundary value of an analytic function $f_j$ in the domain $G$ with $\Im f_j(\alpha)=0$ for bounded $G$ and $\Im f_j(\infty)=0$ for unbounded $G$. 
Approximate values of $f_j(z)$ for $z\in G$ can be computed by the Cauchy integral formula. The values of $f_j(z)$ can be computed by using MATLAB function {\tt fcau} from~\cite{Nas-ETNA}.
When $G$ is unbounded, using the Cauchy integral formula to compute $f_j(z)$ requires knowing the value of $c_r=f_j(\infty)$. In this case, it follows from the second equation in~\eqref{eq:mu-Egam2} that $c_r=\gamma+\bE\mu$, which can be approximated as the average of the vector $\gamma(\bt)+E\mu(\bt)$.
Thus $u_j(z) = \Re[f_j(z)]$ is an approximation of an eigenfunction for the Steklov eigenvalue problem~\eqref{eq:st-2} corresponding to the eigenvalue $\lambda_j$. The above proposed BIE method is summarized in Algorithm~\ref{alog:1}.

\begin{algorithm}[h]
\caption{Steklov eigenvalue computation}\label{alog:1}
\begin{enumerate}
\item Parametrize $\Gamma$ by $\eta(t)$, $0 \le t \le 2\pi$, on $n$ equidistant grid points.
\item Form the Nystr\"om matrices for the kernels $N$ and $M$.
\item Assemble $E = -(I - B)^{-1}C$.
\item Form $Q = P \cdot D \cdot E$ where $D = F W F^\ast$ and $P = \mathrm{diag}(\rho(\mathbf{t}))$.
\item Compute eigenvalues of $Q + I$ via \texttt{eigs}; subtract 1.
\item Recover eigenfunctions: $\mu_j = E \gamma_j$, extend to $G$ via Cauchy integral formula.
\end{enumerate}
\end{algorithm}

\section{Numerical examples}\label{sec:ex}

In this section, the BIE method presented in~\S\ref{sec:sim} is used to compute
eigenvalues and eigenfunctions for several smooth bounded and unbounded simply
connected domains. The numerical evidence is organized around three questions:
whether it reproduces known results on standard smooth test domains, what it
reveals about geometry-dependent spectral behavior in smooth
parameter-dependent families, and how its results compare with those of an
independent high-order volumetric solver on a benchmark geometry.
The numerical experiments are outlined in Table~\ref{tab:exp-roadmap}.

\begin{table}[ht]
\caption{Roadmap of the numerical experiments in Section~\ref{sec:ex}.}\label{tab:exp-roadmap}
\centering
\footnotesize
\begin{tabular}{@{}p{3.5cm}p{4cm}p{7.5cm}@{}}
\hline\noalign{\smallskip}
Experiment & Purpose & Main observation \\
\noalign{\smallskip}\hline\noalign{\smallskip}
\hyperref[ex:1]{Benchmark domains $G_1$ and $G_2$ (Ex.~\ref*{ex:1})} & Reproduce published smooth-domain data & Agreement with published spectra on two benchmark geometries and rapid observed convergence on analytic boundaries \\
\hyperref[ex:2]{Kite interior/exterior pair (Ex.~\ref*{ex:2})} & Test the unified bounded/exterior formulation & The same BIE framework resolves both interior and exterior problems, and the exterior spectrum enters the asymptotic regime earlier \\
\hyperref[ex:ell]{Ellipses with fixed perimeter (Ex.~\ref*{ex:ell})} & Study geometry dependence under smooth elongation & Interior branches exhibit crossings and non-monotonicity, while the computed exterior branches are monotone; classical inequalities are validated numerically \\
\hyperref[ex:star2]{Symmetric star-like domains (Ex.~\ref*{ex:star2})} & Track response under strong smooth deformation & The method remains effective as the geometry approaches a near-pinched regime and captures substantial changes in the low modes \\
\hyperref[sec:hpfemcomp]{$G_1$ $hp$-FEM benchmark (Sec.~\ref*{sec:hpfemcomp})} & Independent cross-validation & BIE and $hp$-FEM agree closely on the first ten area-scaled eigenvalues, and the FEM $p$-study provides an additional reference value \\
\noalign{\smallskip}\hline
\end{tabular}%
\end{table}

\subsection{Benchmark domains from the literature}

In this subsection, for comparison, we consider two examples with domains that have been
studied before in the literature. 

\begin{example}\label{ex:1}
We consider the following two bounded simply connected domains from~\cite[Figures~5 and~6]{Alhe}. The boundary of the first domain $G_1$ is parametrized by (see Figure~\ref{fig:domains} (left))
\begin{equation}
    \eta_1(t) = 8 + 5e^{\i t} + 0.5e^{6\i t}, \quad 0\le t\le 2\pi.
\end{equation}
The boundary of the second domain $G_2$ is parametrized by (see Figure~\ref{fig:domains} (right))
\begin{equation}
    \eta_2(t) = 0.4\i e^{\i t} \sqrt{\dfrac{2}{1.16-.84e^{2\i t}}}, \quad 0\le t\le 2\pi.  
\end{equation}
The domain $G_2$ here is a rotation of the domain in~\cite[Figure~6]{Alhe} and hence the domain $G_2$ and the domain in~\cite[Figure~6]{Alhe} have the same Steklov spectrum.
\end{example}

We use the above method with $n=2^{10}$ grid points to compute the first $10$
nonzero eigenvalues $\lambda_k$ for both domains $G_1$ and $G_2$.
To compare our results with the results presented in~\cite{Alhe}, Table~\ref{tab:ex1} presents the first $10$ nonzero area-scaled eigenvalues $\tilde{\lambda}_k = \lambda_k\sqrt{|G|}$ where $|G|$ is the area of the domain $G$. 
Table~\ref{tab:ex1} also presents the results obtained with $2^{10}$ grid points in~\cite[Table~3, Table~4]{Alhe}.
As shown in the table, our method agrees at least with ten decimal places with the results presented in~\cite{Alhe}.
The relative error computed by~\eqref{eq:ea_err} for each of the first $10$ nonzero Steklov eigenvalues versus the number of grid points $n$ for both domains $G_1$ and $G_2$ is presented in Figure~\ref{fig:dom1_error}. 
Since the exact eigenvalues are unknown, we consider the eigenvalues $\lambda_k$
calculated with $2^{10}$ grid points as reference values. The approximate
eigenvalues $\lambda_{k,n}$ are then calculated for $20 \le n \le 400$ grid
points. Figure~\ref{fig:dom1_error} suggests exponential-like convergence in
this benchmark, which is consistent with the analyticity of the boundaries.

The number of iterations and the CPU time (sec) required for the convergence of the MATLAB function {\tt eigs} for the first $10$ nonzero eigenvalues for the domains $G_1$ and $G_2$ are presented in Figure~\ref{fig:dom1_time}. 
The condition number of the matrix $Q+I$, for the matrix $Q$ defined in~\eqref{eq:matQ}, is also presented in Figure~\ref{fig:dom1_time}. 
Figure~\ref{fig:dom1_time} indicates that, although the condition number of the
matrix $Q+I$ increases with $n$, the number of iterations in {\tt eigs} remains
nearly independent of $n$ over the tested range. To illustrate the asymptotic
behavior of the Steklov eigenvalues~\eqref{eq:asym_b},
Figures~\ref{fig:dom1_eig} and~\ref{fig:dom2_eig} present the approximate
nonzero eigenvalues $\lambda_k$ (for $1\le k\le 10$, $30\le k\le 40$, and
$90\le k\le 100$) computed using the above method with $n=2^{10}$ for both
$G_1$ and $G_2$. The plots indicate that the eigenvalues for the domain $G_2$
enter the asymptotic regime earlier than those for the domain $G_1$.  
The eigenmodes and their boundary traces corresponding to the first $8$ nonzero
eigenvalues computed with $n=2^{10}$ for both domains $G_1$ and $G_2$ are shown
in Figures~\ref{fig:dom_1_eigf} and~\ref{fig:dom_1_eigf_b}. The boundary traces
reflect the oscillatory structure of the corresponding harmonic modes and
provide a complementary view of the computed eigenfunctions on $\Gamma$.

\begin{table}[ht]
\caption{The first $10$ nonzero area-scaled eigenvalues $\tilde{\lambda}_k$ for the domains $G_1$ and $G_2$ in Example~\ref{ex:1} obtained with $n=2^{10}$ and the corresponding eigenvalues from~\cite{Alhe}.}\label{tab:ex1}
\centering
\small
\setlength{\tabcolsep}{12pt}
\begin{tabular}{ccccc}
\hline\noalign{\smallskip}
  & $G_1$ & \cite[Table 3]{Alhe} & $G_2$ & \cite[Table 4]{Alhe} \\
\noalign{\smallskip}\hline\noalign{\smallskip}
 $\tilde\lambda_1$    & 1.61465185265077 & 1.61465185265076 & 0.82158389917705 & 0.82158389917723 \\
 $\tilde\lambda_2$    & 1.61465185265086 & 1.61465185265091 & 2.88853778576938 & 2.88853778576941 \\
 $\tilde\lambda_3$    & 2.97737736702950 & 2.97737736702987 & 2.94484661549781 & 2.94484661549826 \\
 $\tilde\lambda_4$    & 2.97737736702974 & 2.97737736702991 & 3.34172628966417 & 3.34172628966405\\
 $\tilde\lambda_5$    & 5.48337898612383 & 5.48337898612405 & 4.55074794910963 & 4.55074794911011\\
 $\tilde\lambda_6$    & 5.48337898612393 & 5.48337898612444 & 5.03673963982603 & 5.03673963982608\\
 $\tilde\lambda_7$    & 6.70773879741621 & 6.70773879741643 & 6.23305352696130 & 6.23305352696119\\
 $\tilde\lambda_8$    & 6.70773879741642 & 6.70773879741657 & 6.32549098892433 & 6.32549098892451 \\
 $\tilde\lambda_9$    & 7.65773980917837 & 7.65773980917866 & 7.80580771944321 & 7.80580771944354\\
 $\tilde\lambda_{10}$ & 9.01958292273808 & 9.01958292273825 & 7.90841610595226 & 7.90841610595226\\
\noalign{\smallskip}\hline
\end{tabular}
\end{table}

\begin{figure}[ht] %example_table3_E.m, example_table4_E.m
\centerline{
\hfill
\scalebox{0.35}{\includegraphics{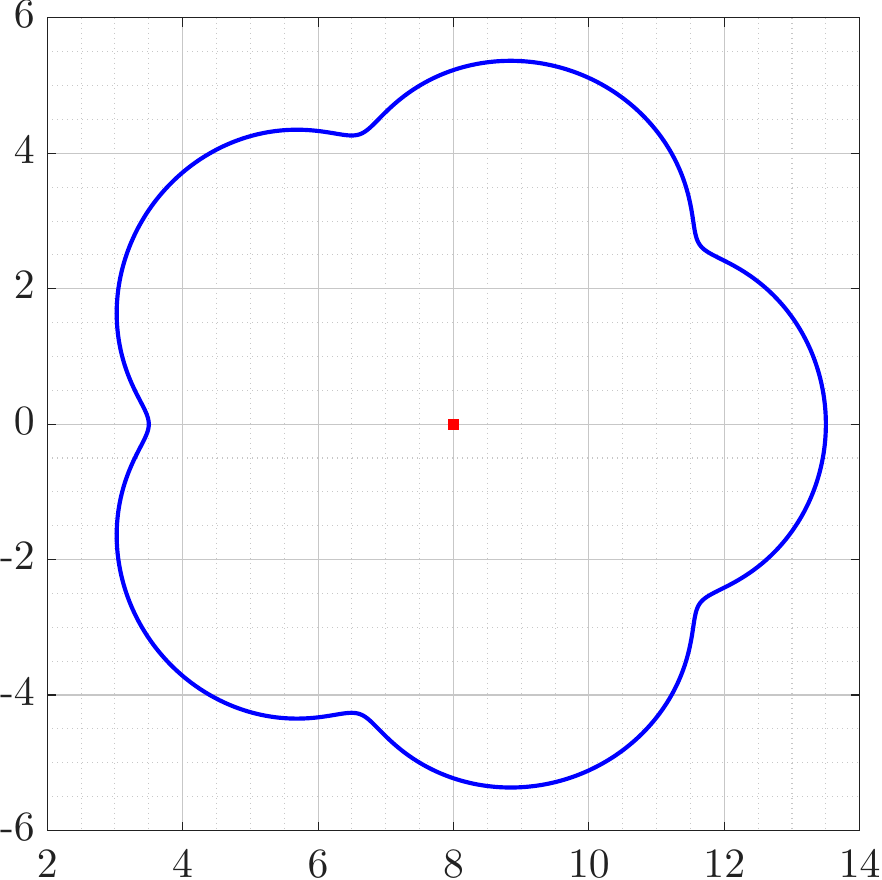}}
\hfill
\scalebox{0.35}{\includegraphics{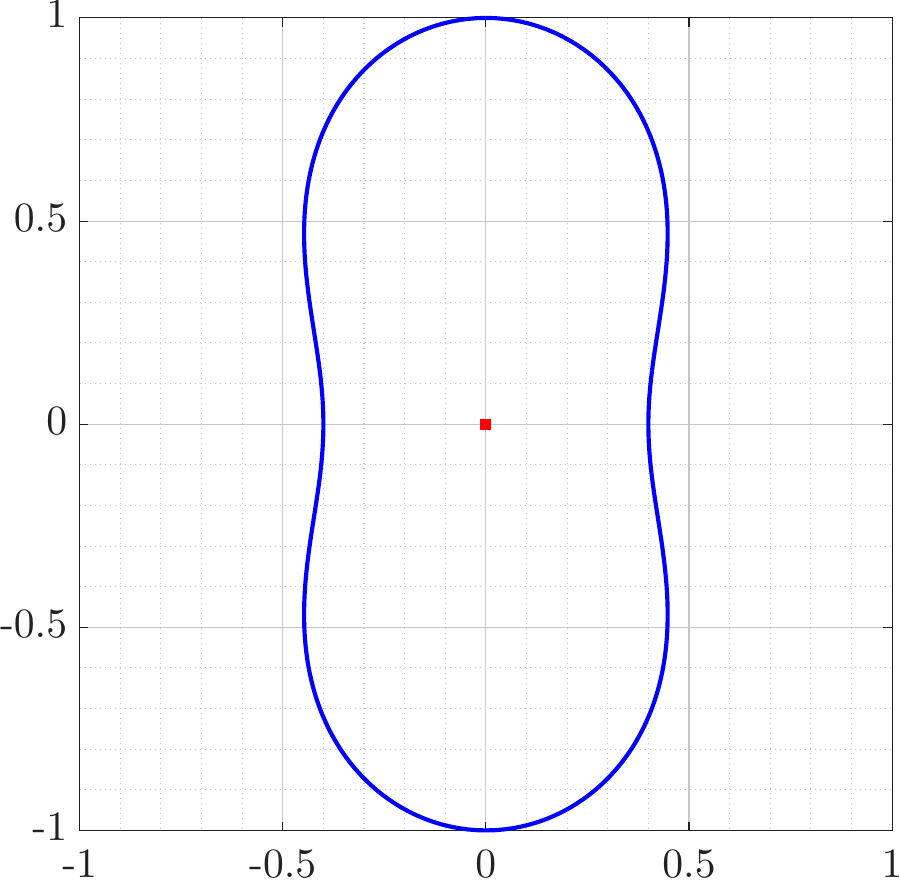}}
\hfill}
\caption{The boundary $\Gamma$ for the domain $G_1$ (left) and the domain $G_2$ (right) for Example~\ref{ex:1}.}
	\label{fig:domains}
\end{figure}

\begin{figure}[ht] 
\centerline{
\hfill
\scalebox{0.35}{\includegraphics{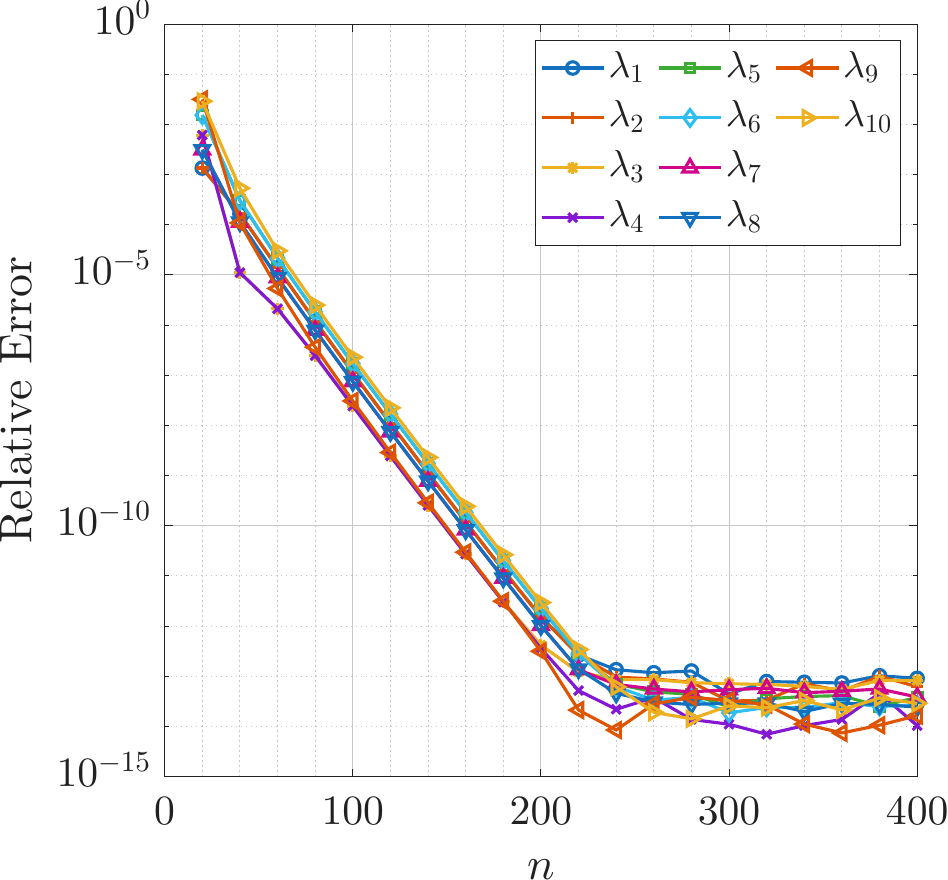}}
\hfill
\scalebox{0.35}{\includegraphics{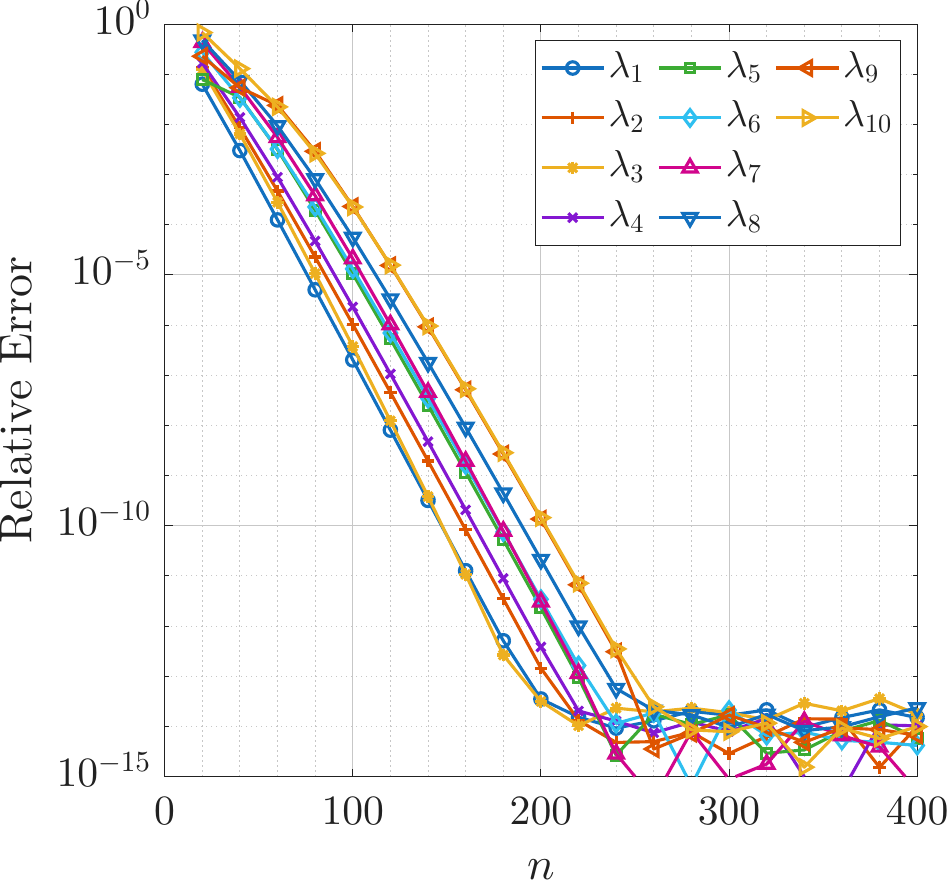}}
\hfill}
\caption{The relative error for the first $10$ nonzero eigenvalues for the domains $G_1$ (left) and $G_2$ (right) in Example~\ref{ex:1}.}
	\label{fig:dom1_error}
\end{figure}

\begin{figure}[ht] 
\centerline{
\hfill
\scalebox{0.35}{\includegraphics{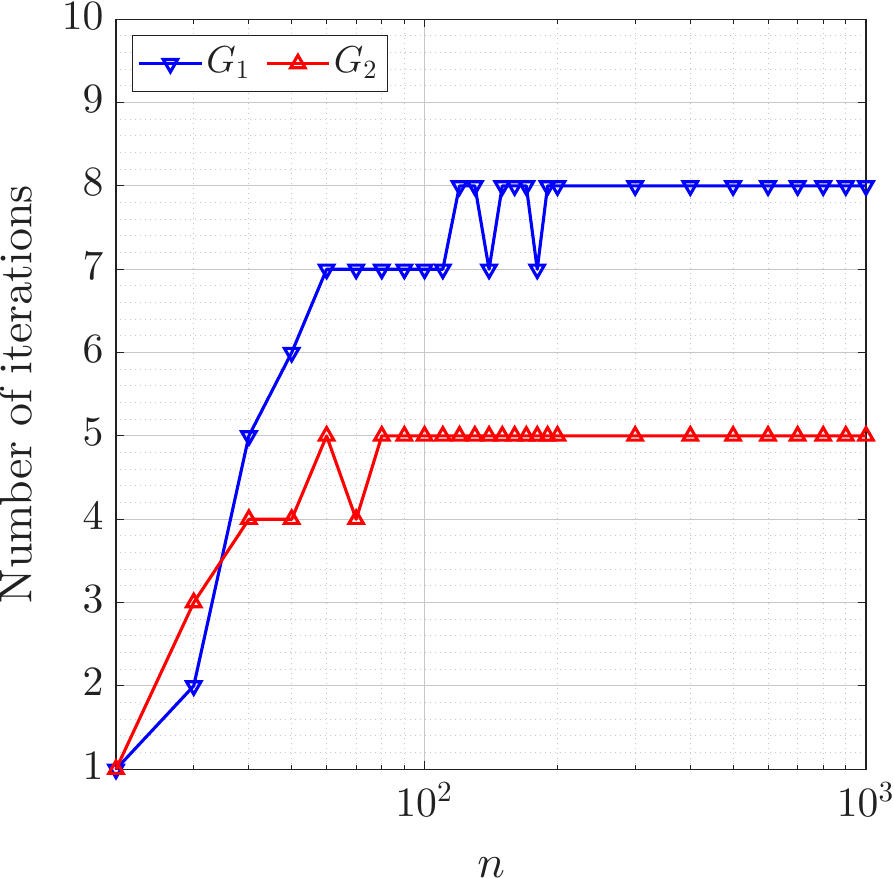}}
\hfill
\scalebox{0.35}{\includegraphics{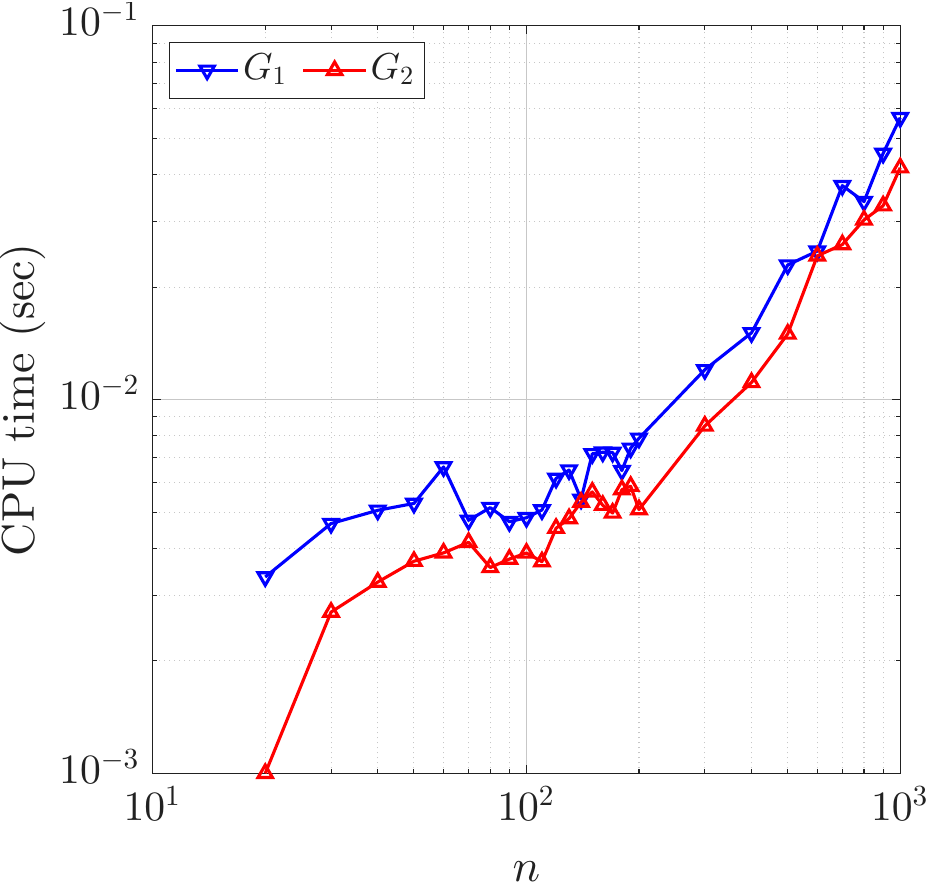}}
\hfill
\scalebox{0.35}{\includegraphics{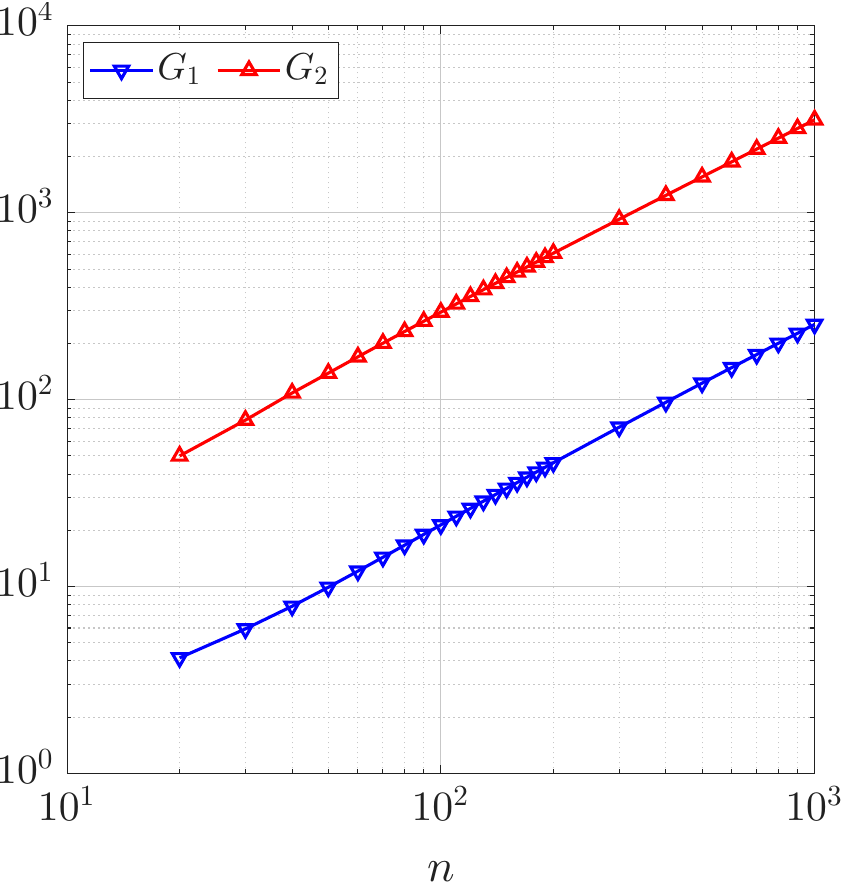}}
\hfill}
\caption{The number of iterations and CPU time (sec) required for the convergence of the MATLAB function {\tt eigs} for the first $10$ nonzero eigenvalues and the condition number of the matrix $Q+I$ for the domains $G_1$ and $G_2$ in Example~\ref{ex:1}.}
	\label{fig:dom1_time}
\end{figure}

\begin{figure}[ht] %example_new_ext.m
\centerline{
\hfill
\scalebox{0.35}{\includegraphics{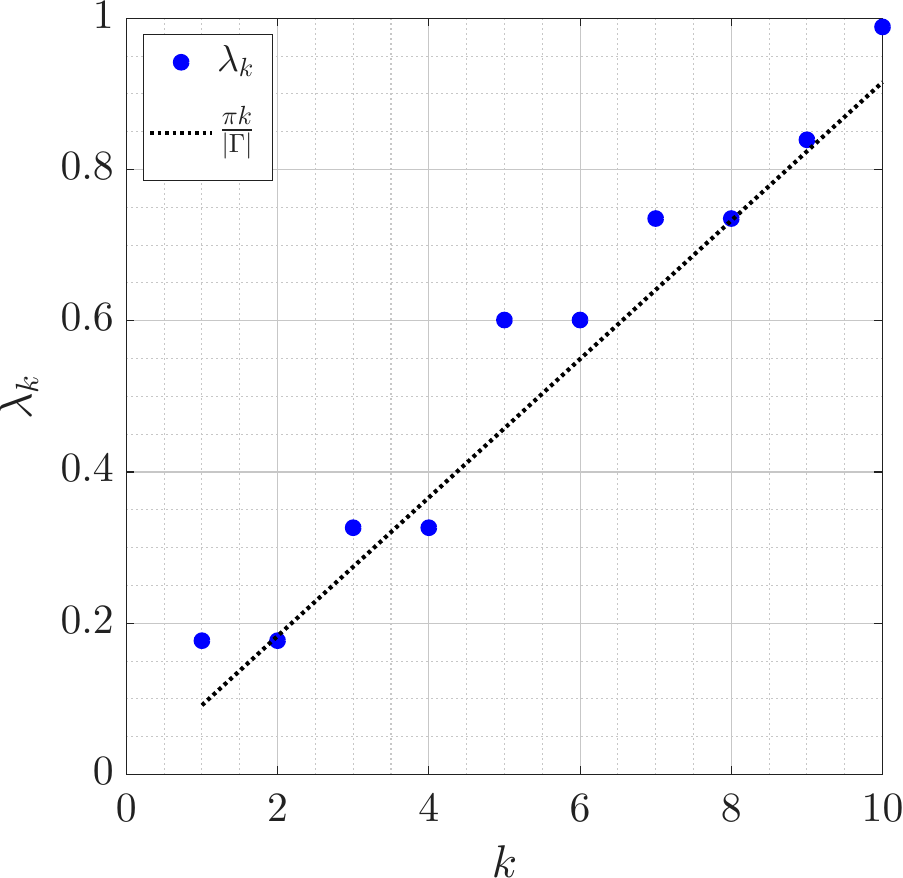}}
\hfill
\scalebox{0.35}{\includegraphics{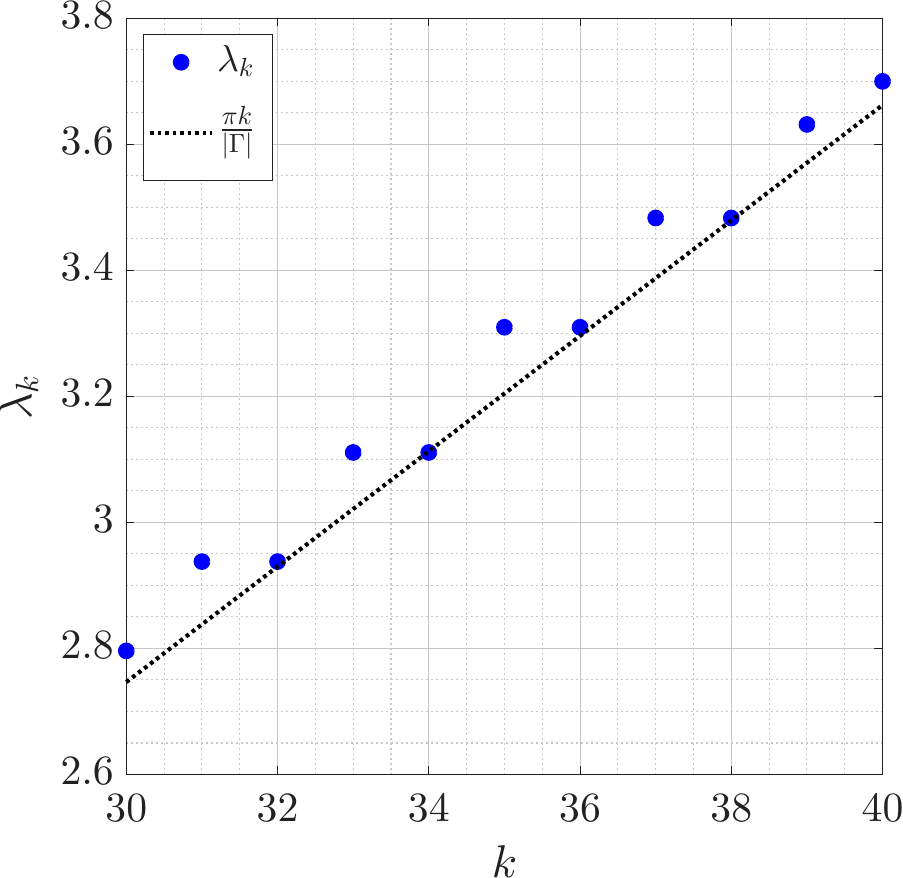}}
\hfill
\scalebox{0.35}{\includegraphics{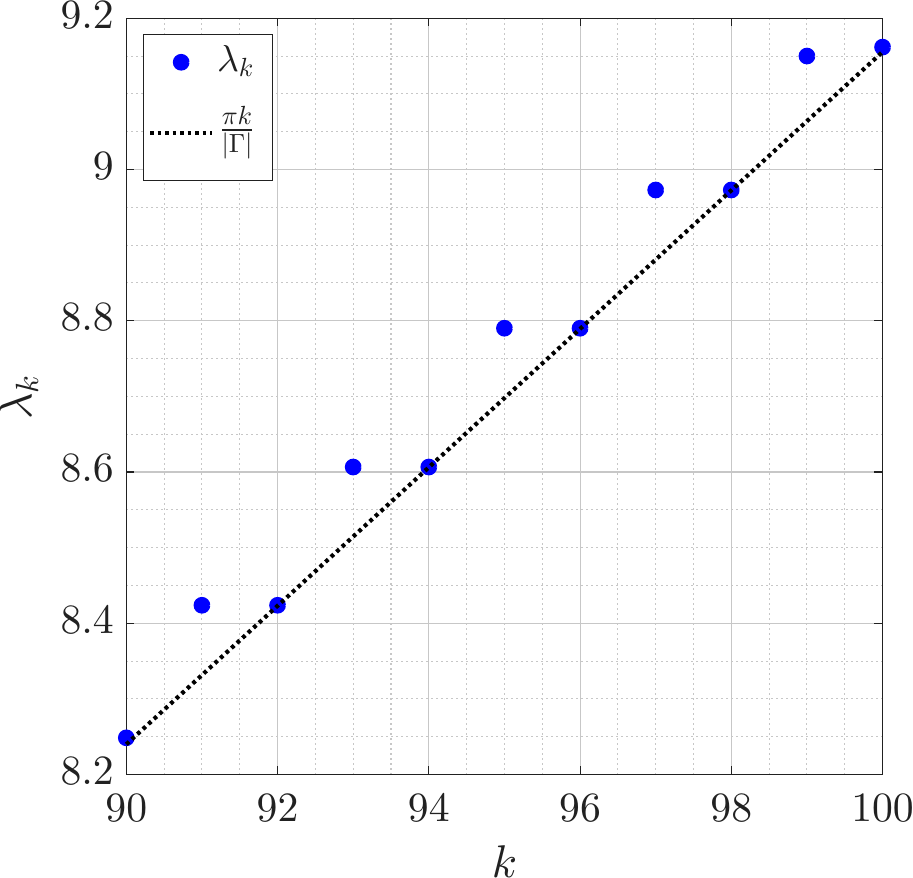}}
\hfill}
\caption{The eigenvalues $\lambda_k$ for $1\le k \le 10$ (left), $30 \le k \le 40$ (center), and $90 \le k \le 100$ (right) for the domain $G_1$ in Example~\ref{ex:1}.}
	\label{fig:dom1_eig}
\end{figure}

\begin{figure}[ht] %example_new_ext.m
\centerline{
\hfill
\scalebox{0.35}{\includegraphics{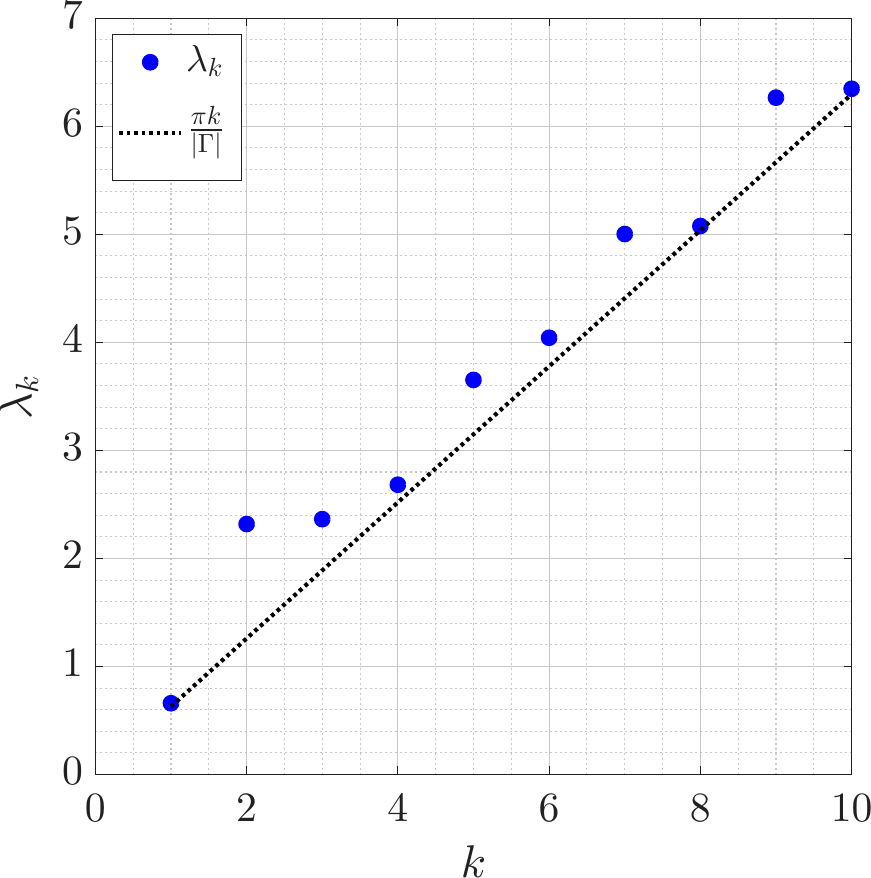}}
\hfill
\scalebox{0.35}{\includegraphics{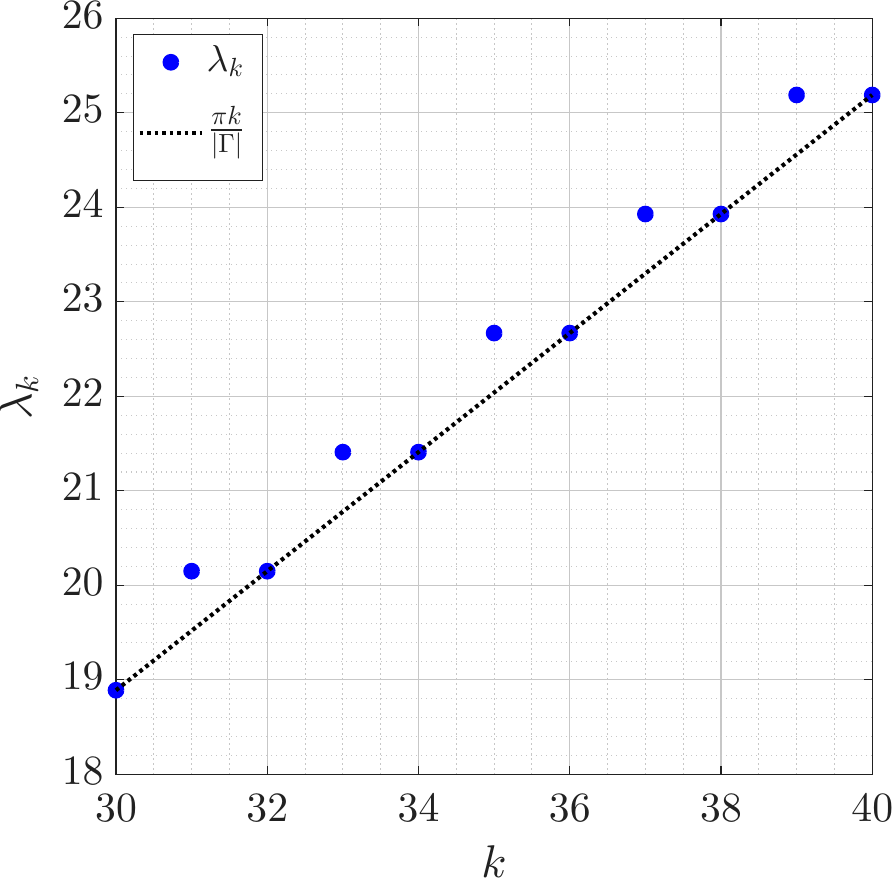}}
\hfill
\scalebox{0.35}{\includegraphics{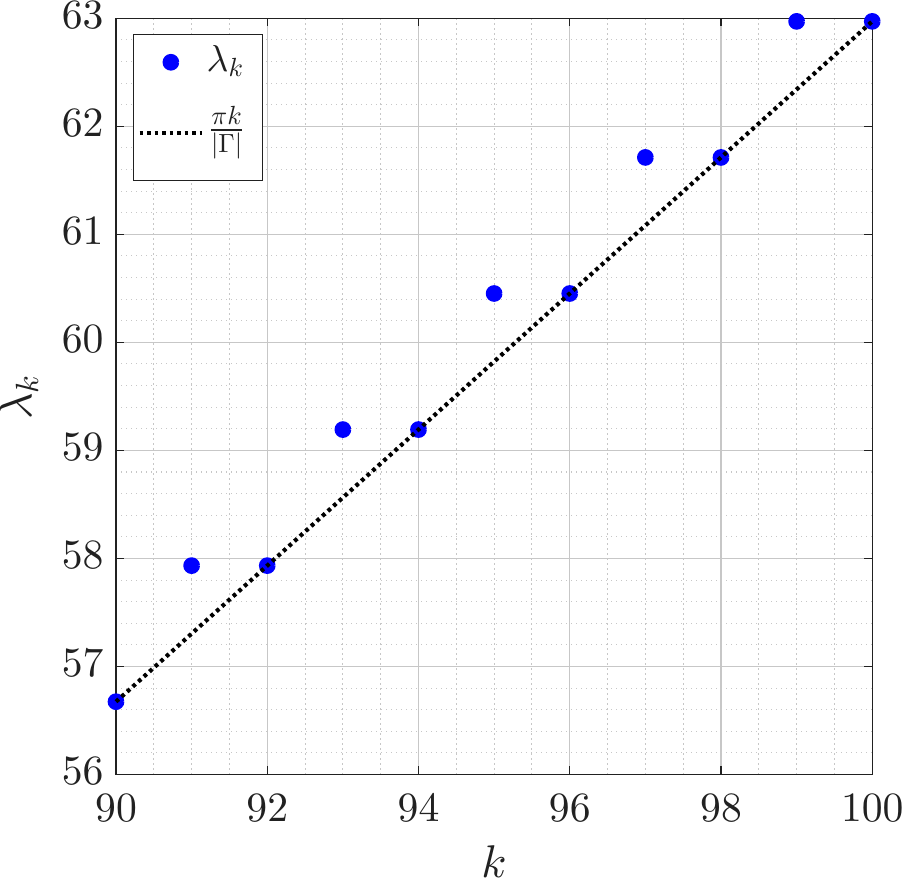}}
\hfill}
\caption{The eigenvalues $\lambda_k$ for $1\le k \le 10$ (left), $30 \le k \le 40$ (center), and $90 \le k \le 100$ (right) for the domain $G_2$ in Example~\ref{ex:1}.}
	\label{fig:dom2_eig}
\end{figure}

\begin{figure}[htb] %
\centerline{
\hfill
\scalebox{0.4}{\includegraphics[trim=2.5cm 1.0cm 2.5cm 1.0cm,clip]{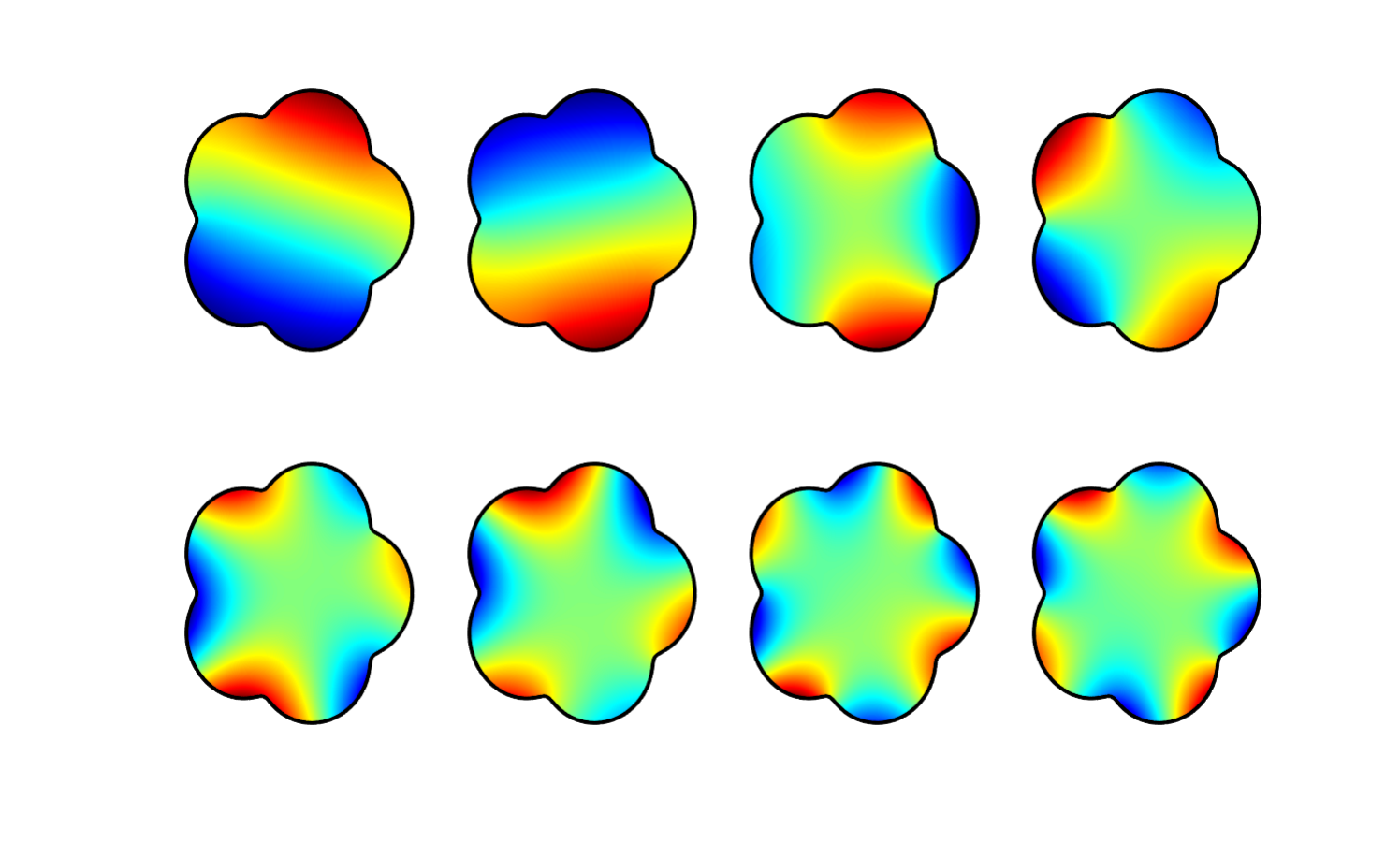}}
\hfill
\scalebox{0.4}{\includegraphics[trim=2.5cm 1.0cm 2.5cm 1.0cm,clip]{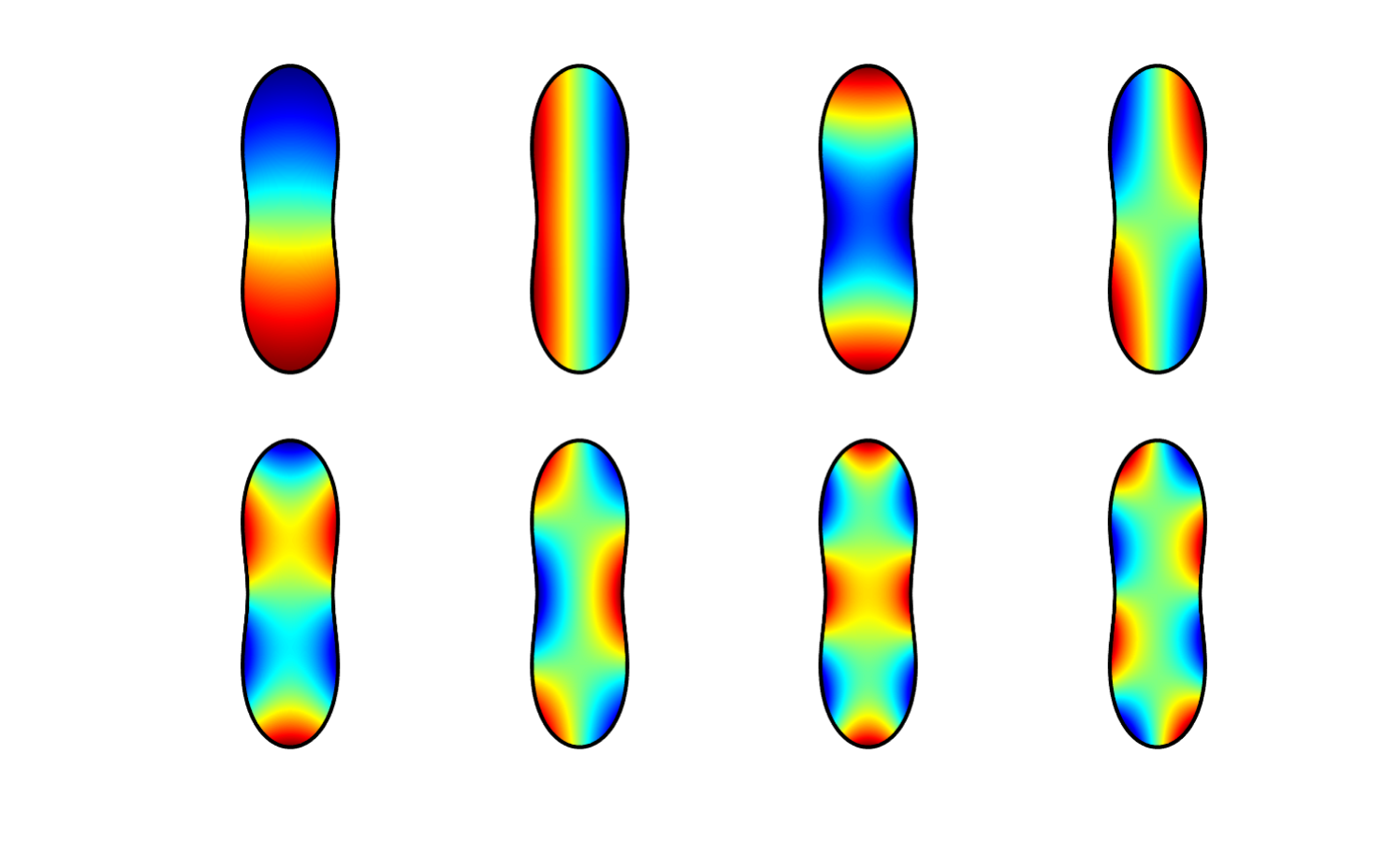}}
\hfill}
\caption{Eigenmodes for the first $8$ nonzero eigenvalues $\lambda_1,\lambda_2,\lambda_3,\lambda_4$ (first row, from left to right) and $\lambda_5,\lambda_6,\lambda_7,\lambda_8$ (second row, from left to right) for the domains $G_1$ (left) and $G_2$ (right) in Example~\ref{ex:1}. }
\label{fig:dom_1_eigf}
\end{figure}

\begin{figure}[htb] %
\centerline{
\hfill
\scalebox{0.35}{\includegraphics[trim=0.25cm 1.0cm 0.25cm 0.25cm,clip]{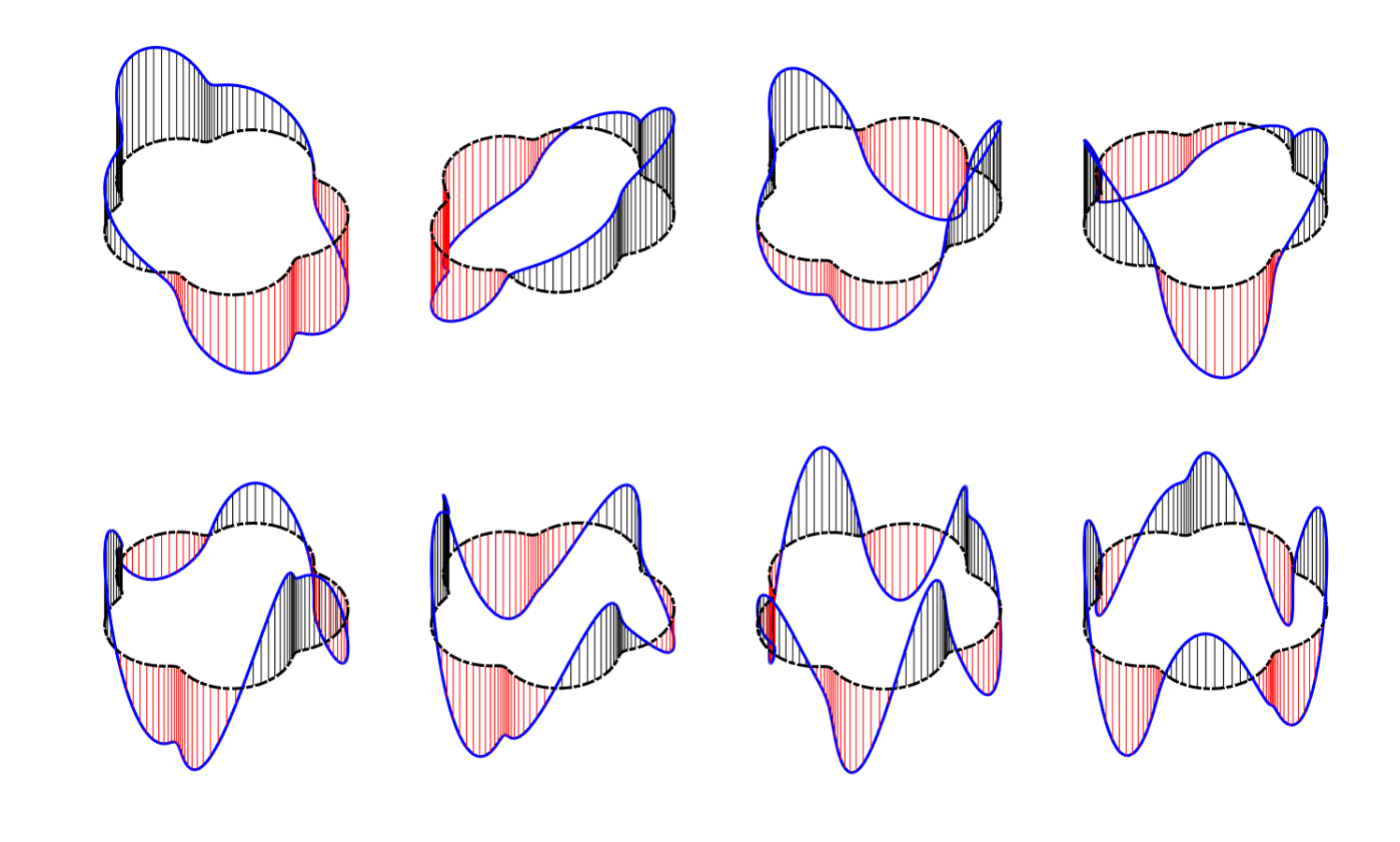}}
\hfill
\scalebox{0.35}{\includegraphics[trim=0.25cm 1.0cm 0.25cm 0.25cm,clip]{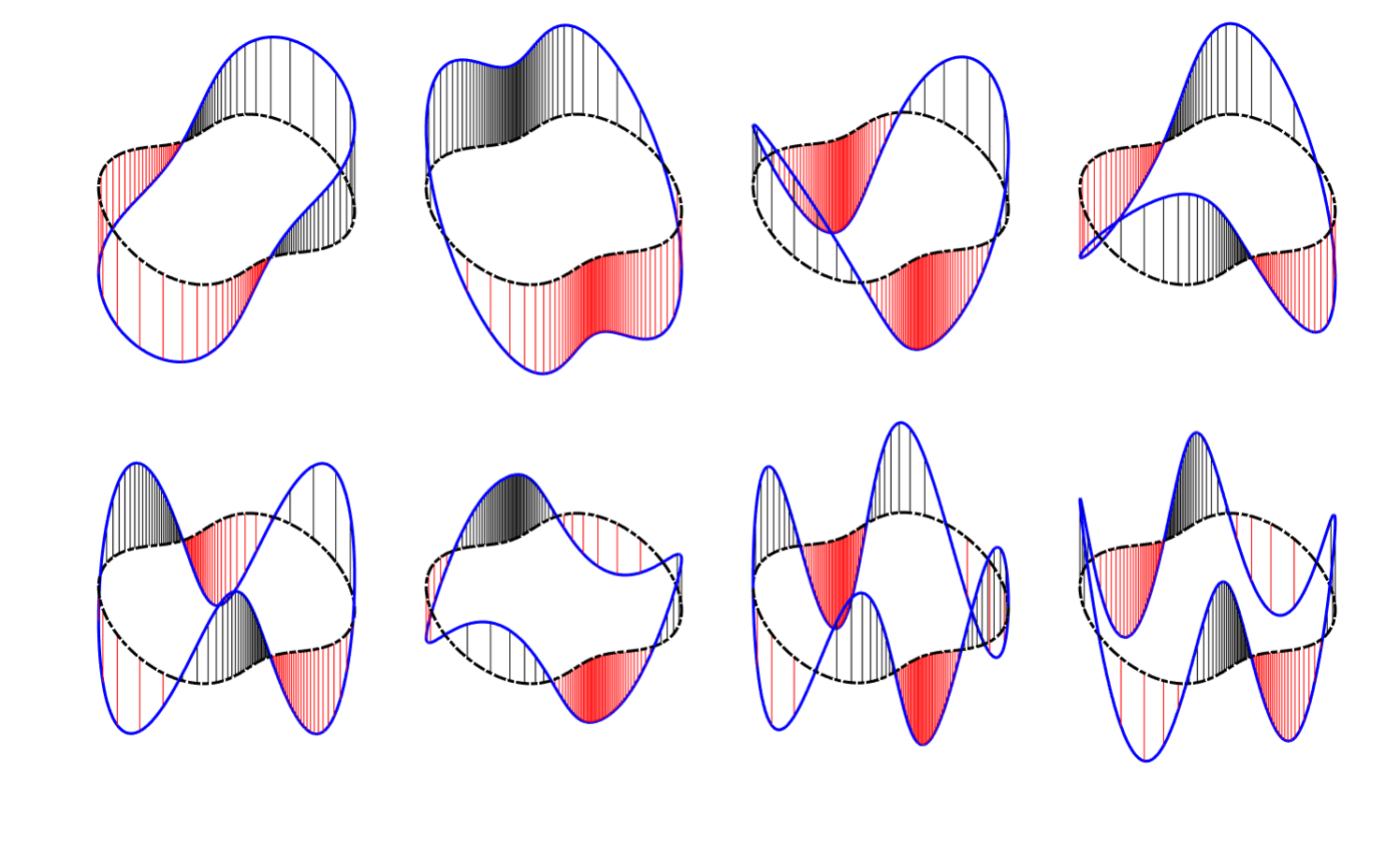}}
\hfill}
\caption{Boundary traces of the eigenmodes corresponding to the first $8$ nonzero eigenvalues $\lambda_1,\lambda_2,\lambda_3,\lambda_4$ (first row, from left to right) and $\lambda_5,\lambda_6,\lambda_7,\lambda_8$ (second row, from left to right) for the domains $G_1$ (left) and $G_2$ (right) in Example~\ref{ex:1}. }
\label{fig:dom_1_eigf_b}
\end{figure}

\begin{example}\label{ex:2}
We consider the following bounded and unbounded simply connected domains from~\cite[Figure~4]{Bun}. Let the bounded simply connected domain $G_1$ and the unbounded simply connected domain $G_2$ be the interior and exterior of an asymmetric ``kite'' $\Gamma$, respectively (see Figure~\ref{fig:dom3}). The orientation of $\Gamma$ for bounded domain $G_1$ is assumed to be counterclockwise and $\Gamma$ is parametrized by
\begin{equation}
\eta(t) = 1.5 \cos(t)+0.7 \cos(2t)-0.4+\i( 1.5\sin(t)-0.3\cos(t)), 
\end{equation}
$0 \leq t \le 2\pi$. For unbounded domain $G_2$, the orientation of $\Gamma$ is assumed to be clockwise and hence $\Gamma$ will be parametrized by
\begin{equation}
\eta(t) = 1.5\cos(t)+0.7\cos(2t)-0.4+\i( -1.5\sin(t)-0.3\cos(t)),
\end{equation}
$0 \leq t \le 2\pi$.
\end{example}

\begin{figure}[ht] %example_new_ext.m, example_new_ext_U.m
\centerline{
\hfill
\scalebox{0.35}{\includegraphics{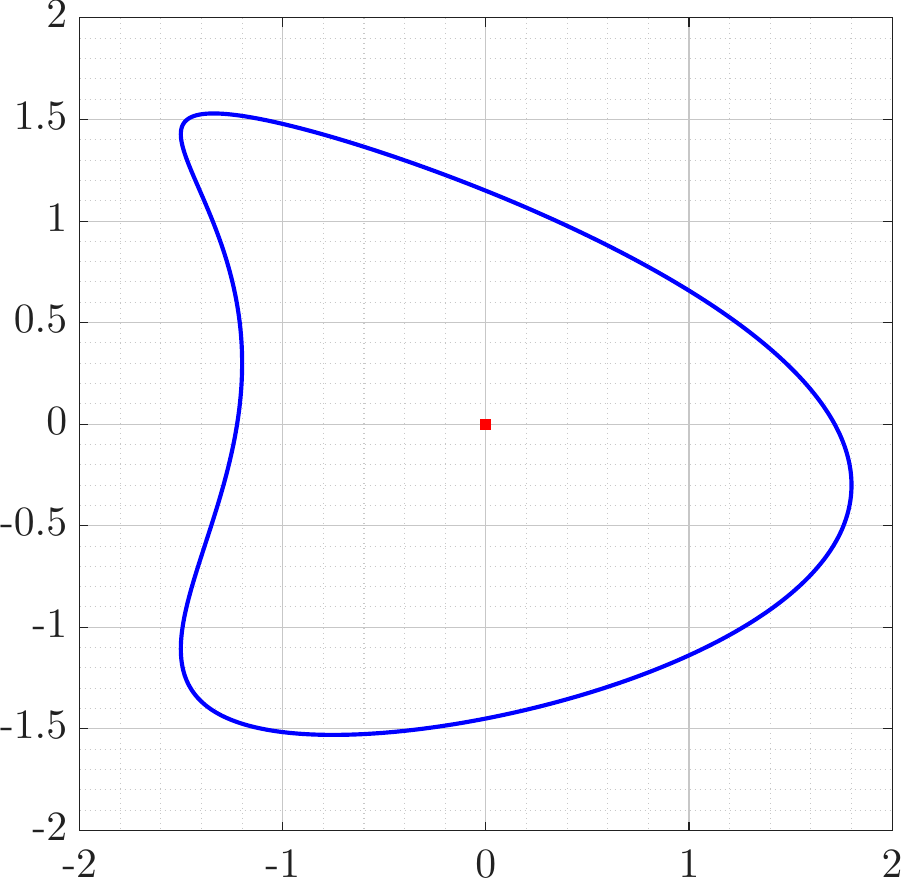}}
\hfill}
\caption{The boundary $\Gamma$ for Example~\ref{ex:2}}
\label{fig:dom3}
\end{figure}

For this example, we present in Table~\ref{tab:ex2} the first $9$ nonzero eigenvalues $\lambda_k$ calculated using the presented method with $n=2^{10}$ grid points to compare our results with those presented in~\cite[Table 1]{Bun}.
As in the previous example, the approximate eigenvalues $\lambda_k$ calculated with $n=2^{10}$ grid points will be considered as the exact values. The relative error for each of the approximate values $\lambda_{k,n}$ of the first $10$ nonzero Steklov eigenvalues for $20\le n\le 400$ will be computed using~\eqref{eq:ea_err}. 
Figure~\ref{fig:dom3_bou_error} shows these relative errors for the bounded
domain $G_1$ (left) and the unbounded domain $G_2$ (right) versus the number of
grid points $n$. For both domains, the error level is approximately $10^{-14}$
once $n \ge 150$ in this experiment. Figure~\ref{fig:kite_time} presents the
number of iterations and the CPU time (sec) required for the convergence of the
MATLAB function {\tt eigs} for the first $10$ nonzero eigenvalues as well as
the condition number of the matrix $Q+I$ for both domains $G_1$ and $G_2$. The
iteration counts are again nearly independent of $n$ over the tested range.
The approximate nonzero eigenvalues $\lambda_k$ (for $1\le k\le 10$, $30\le k\le 40$, and $90\le k\le 100$) computed using the above method with $n=2^{10}$ for both domains $G_1$ and $G_2$ are presented in Figure~\ref{fig:kite_eig}. This figure illustrates that the eigenvalues start following the asymptotic behavior~\eqref{eq:asym_b} for the unbounded domain $G_2$ earlier than the bounded domain $G_1$.  
The eigenmodes corresponding to the first eight nonzero eigenvalues and their boundary traces computed with $n=2^{10}$ for both domains $G_1$ and $G_2$ are shown in Figures~\ref{fig:kite_eigf} and~\ref{fig:kite_eigf_b}.

\begin{table}[ht]
\caption{The first $9$ nonzero eigenvalues $\lambda_k$ for the bounded domain $G_1$ and unbounded domain $G_2$ in Example~\ref{ex:2} obtained with $n=2^{10}$ and the corresponding eigenvalues from~\cite{Bun}.}\label{tab:ex2}
\centering
\setlength{\tabcolsep}{12pt}
\begin{tabular}{ccccc}
\hline\noalign{\smallskip}
  & $G_1$ & \cite[Table 1]{Bun} & $G_2$ & \cite[Table 1]{Bun} \\
\noalign{\smallskip}\hline\noalign{\smallskip}
 $\lambda_1$ & 0.40305996416748 & 0.403 & 0.54467770056080 & 0.545 \\
 $\lambda_2$ & 0.52424200142763 & 0.524 & 0.57081699412402 & 0.571 \\
 $\lambda_3$ & 1.18270198665242 & 1.183 & 1.12953414359678 & 1.130 \\
 $\lambda_4$ & 1.38370805250322 & 1.384 & 1.30930577399346 & 1.309 \\
 $\lambda_5$ & 1.72113574153495 & 1.721 & 1.74564067269481 & 1.746 \\
 $\lambda_6$ & 2.01779563230560 & 2.018 & 1.82146960379857 & 1.821 \\
 $\lambda_7$ & 2.20083979220023 & 2.201 & 2.29287781627365 & 2.293 \\
 $\lambda_8$ & 2.70613635836981 & 2.706 & 2.44997484632459 & 2.450 \\
 $\lambda_9$ & 2.78466524903348 & 2.785 & 2.90350756743372 & 2.903 \\
\noalign{\smallskip}\hline
\end{tabular}
\end{table}

\begin{figure}[ht] %example_new_ext_err_E.m, example_new_ext_U_err_E.m
\centerline{
\hfill
\scalebox{0.35}{\includegraphics{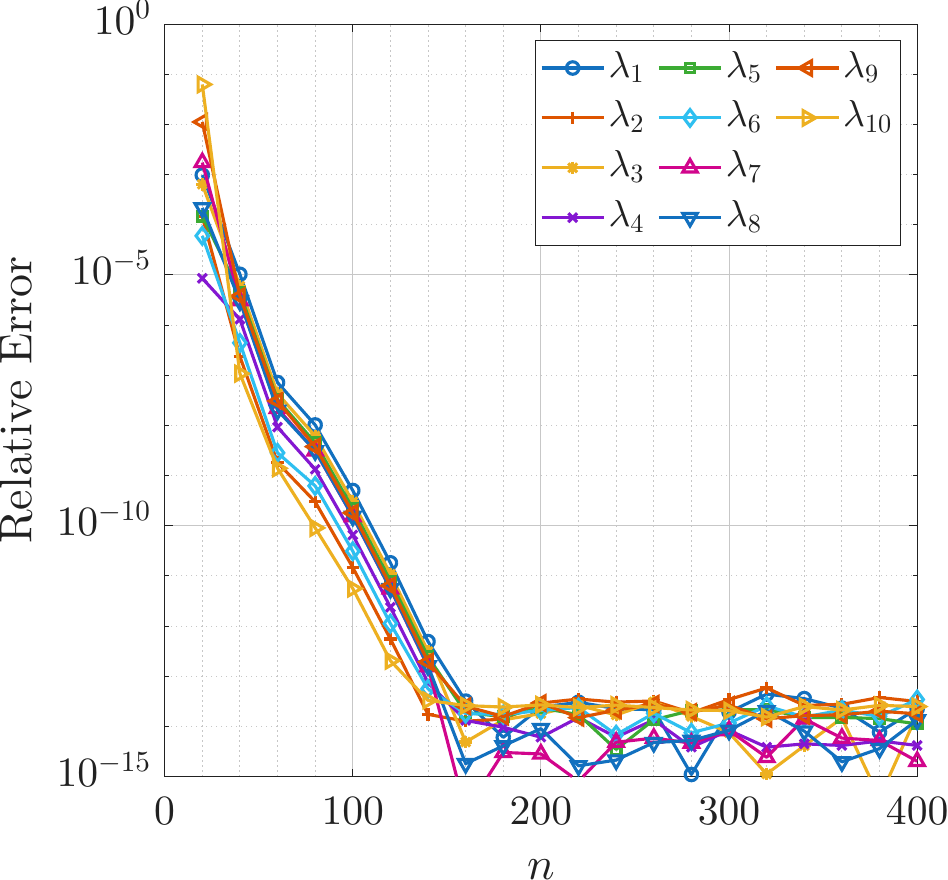}}
\hfill
\scalebox{0.35}{\includegraphics{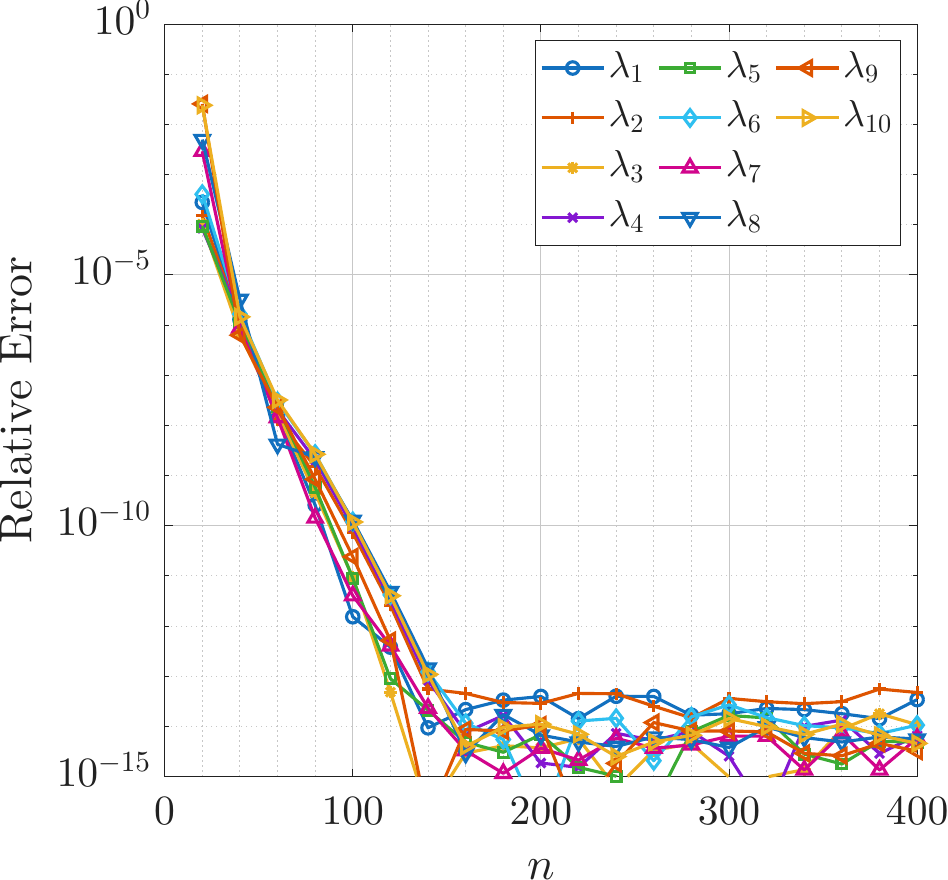}}
\hfill}
\caption{The relative error for the first 10 nonzero eigenvalues for bounded domain $G_1$ (left) and unbounded domain $G_2$ (right) in Example~\ref{ex:2}.}
\label{fig:dom3_bou_error}
\end{figure}

\begin{figure}[ht] 
\centerline{
\hfill
\scalebox{0.35}{\includegraphics{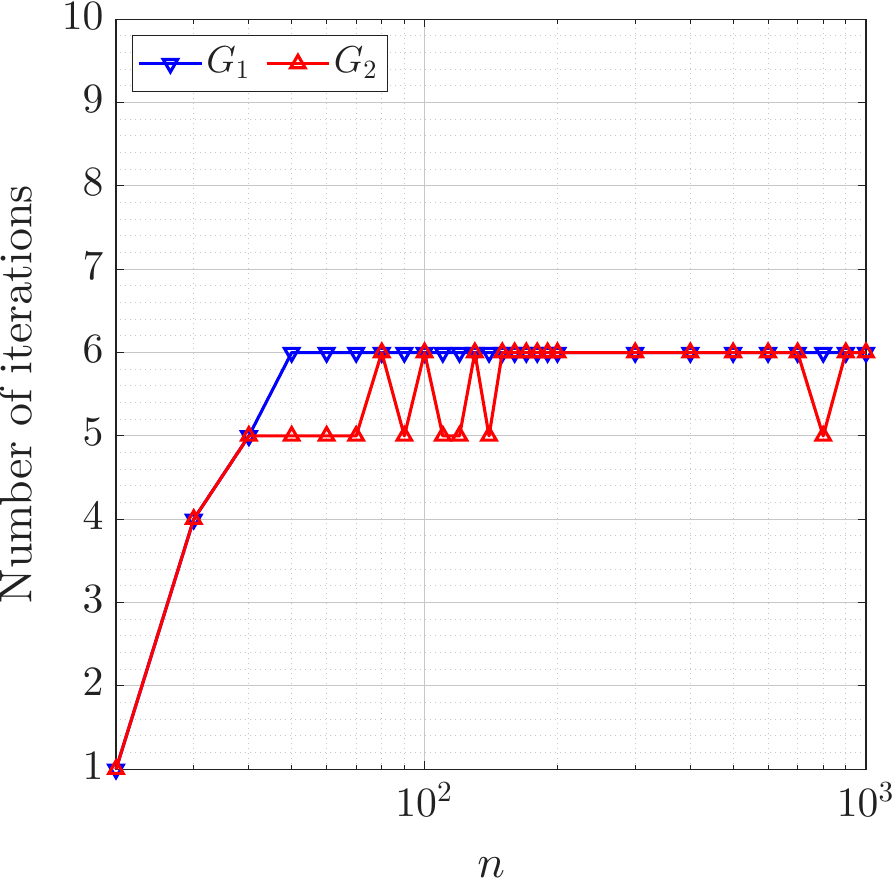}}
\hfill
\scalebox{0.35}{\includegraphics{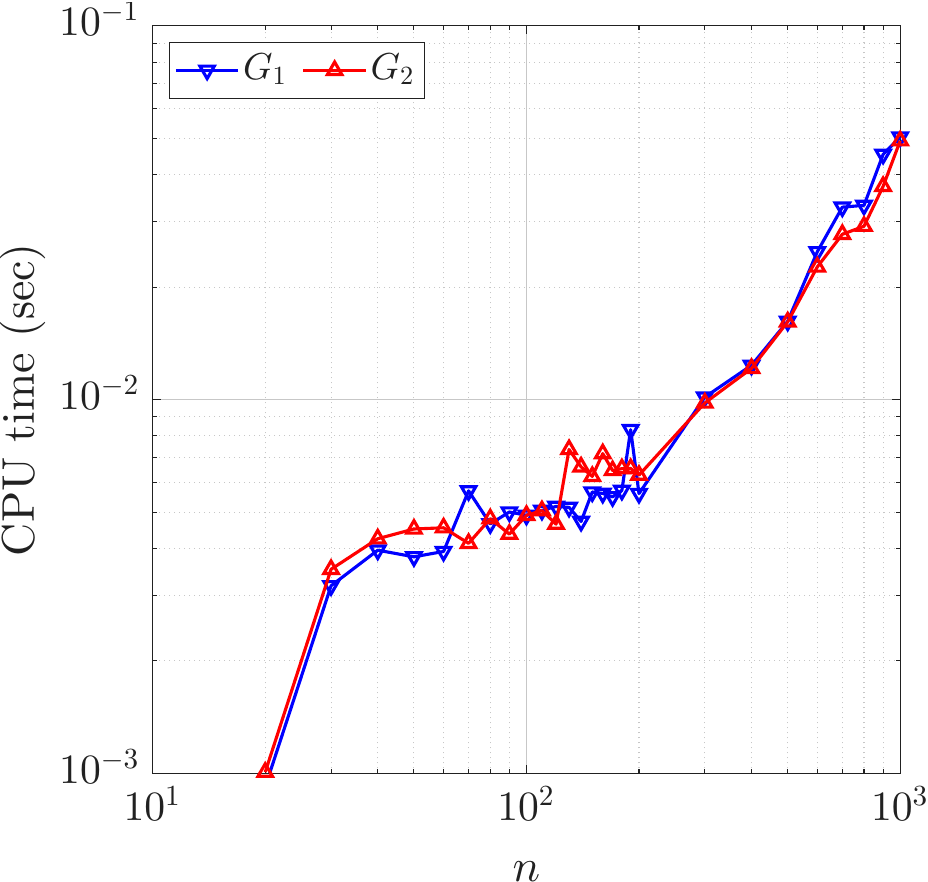}}
\hfill
\scalebox{0.35}{\includegraphics{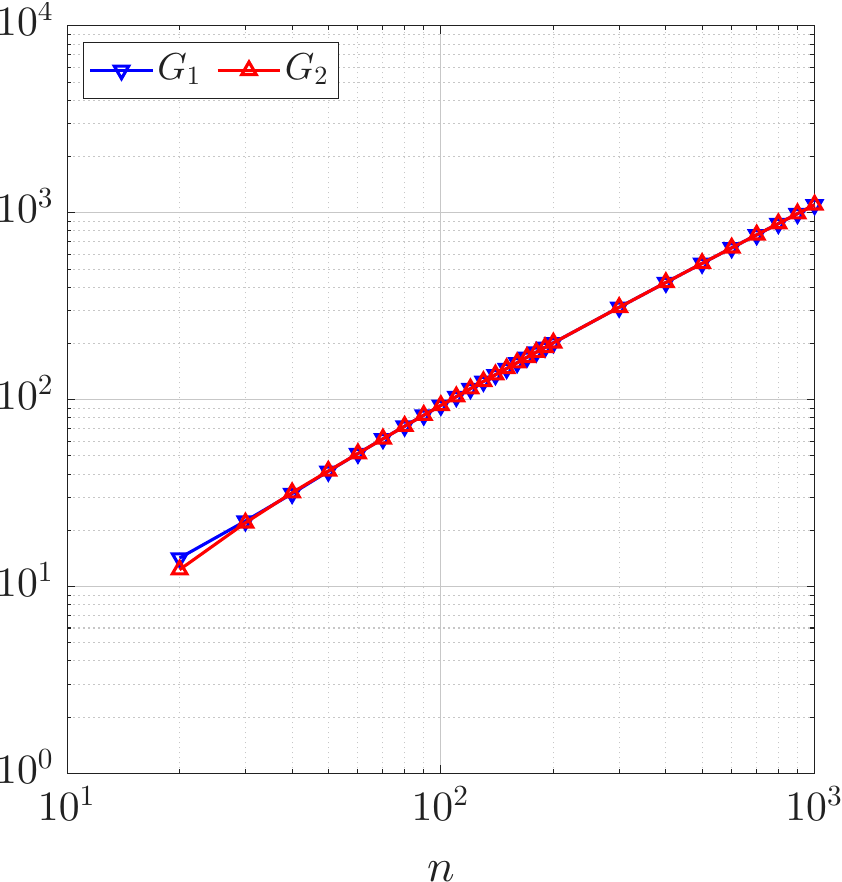}}
\hfill}
\caption{The number of iterations and CPU time (sec) required for the convergence of the MATLAB function {\tt eigs} for the first $10$ nonzero eigenvalues and the condition number of the matrix $Q+I$ for the bounded domain $G_1$ and the unbounded domain $G_2$ in Example~\ref{ex:2}.}
\label{fig:kite_time}
\end{figure}

\begin{figure}[ht] 
\centerline{
\hfill
\scalebox{0.35}{\includegraphics{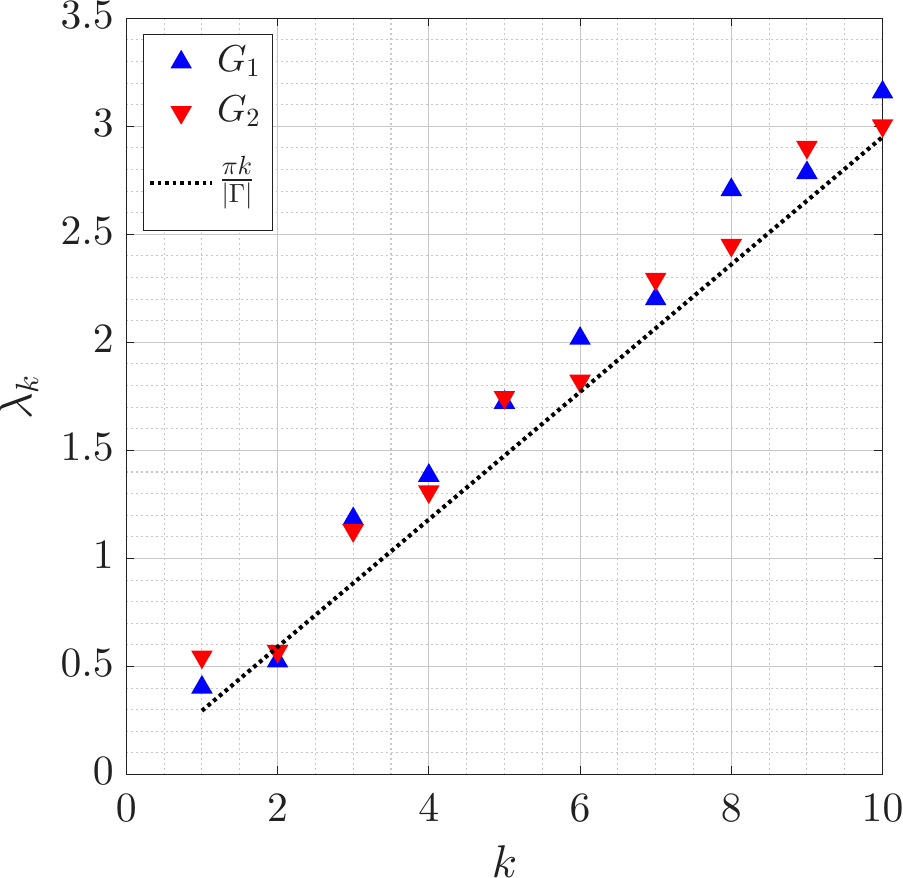}}
\hfill
\scalebox{0.35}{\includegraphics{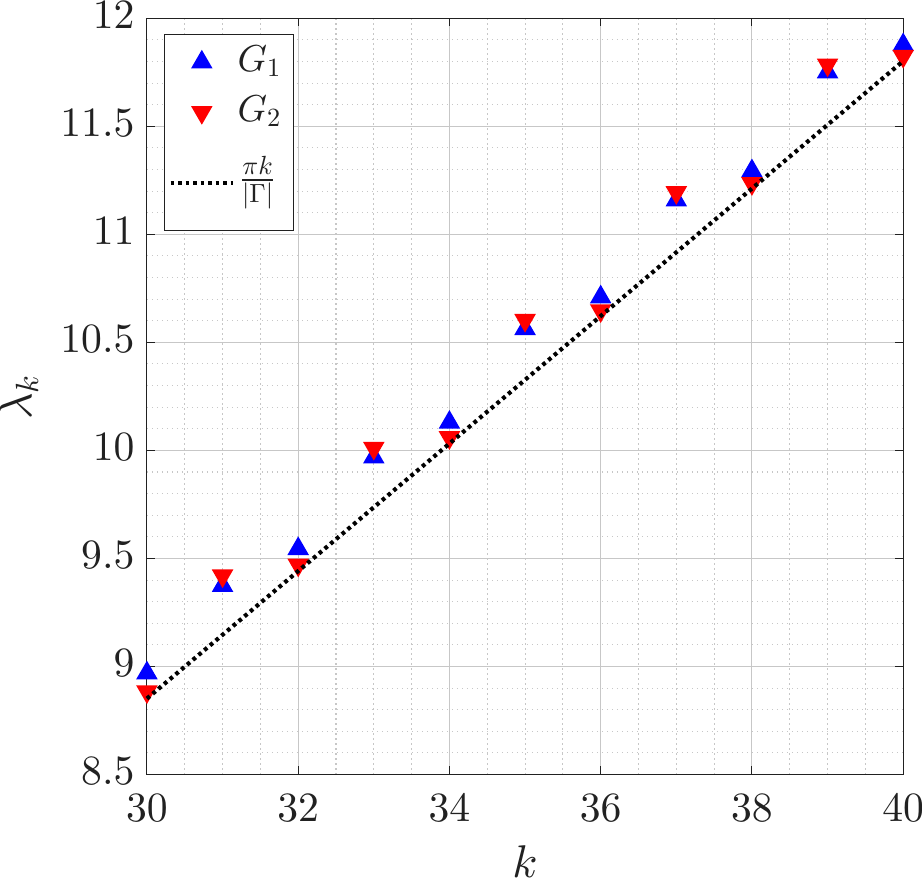}}
\hfill
\scalebox{0.35}{\includegraphics{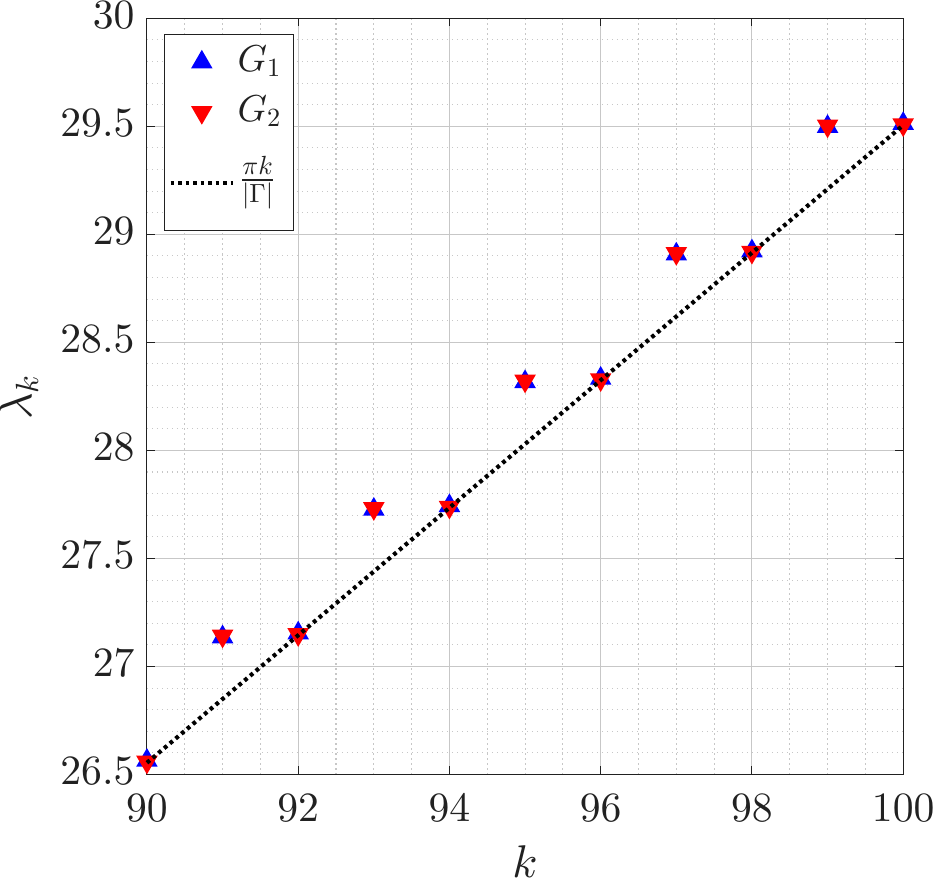}}
\hfill}
\caption{The eigenvalues for $1\le k \le 10$ (left), $30 \le k \le 40$ (center), $90 \le k \le 100$ (right) for both domains $G_1$ and $G_2$ in Example~\ref{ex:2}.}
\label{fig:kite_eig}
\end{figure}

\begin{figure}[htb] %
\centerline{
\hfill
\scalebox{0.32}{\includegraphics[trim=0.0cm 0.0cm 0.0cm 0.0cm,clip]{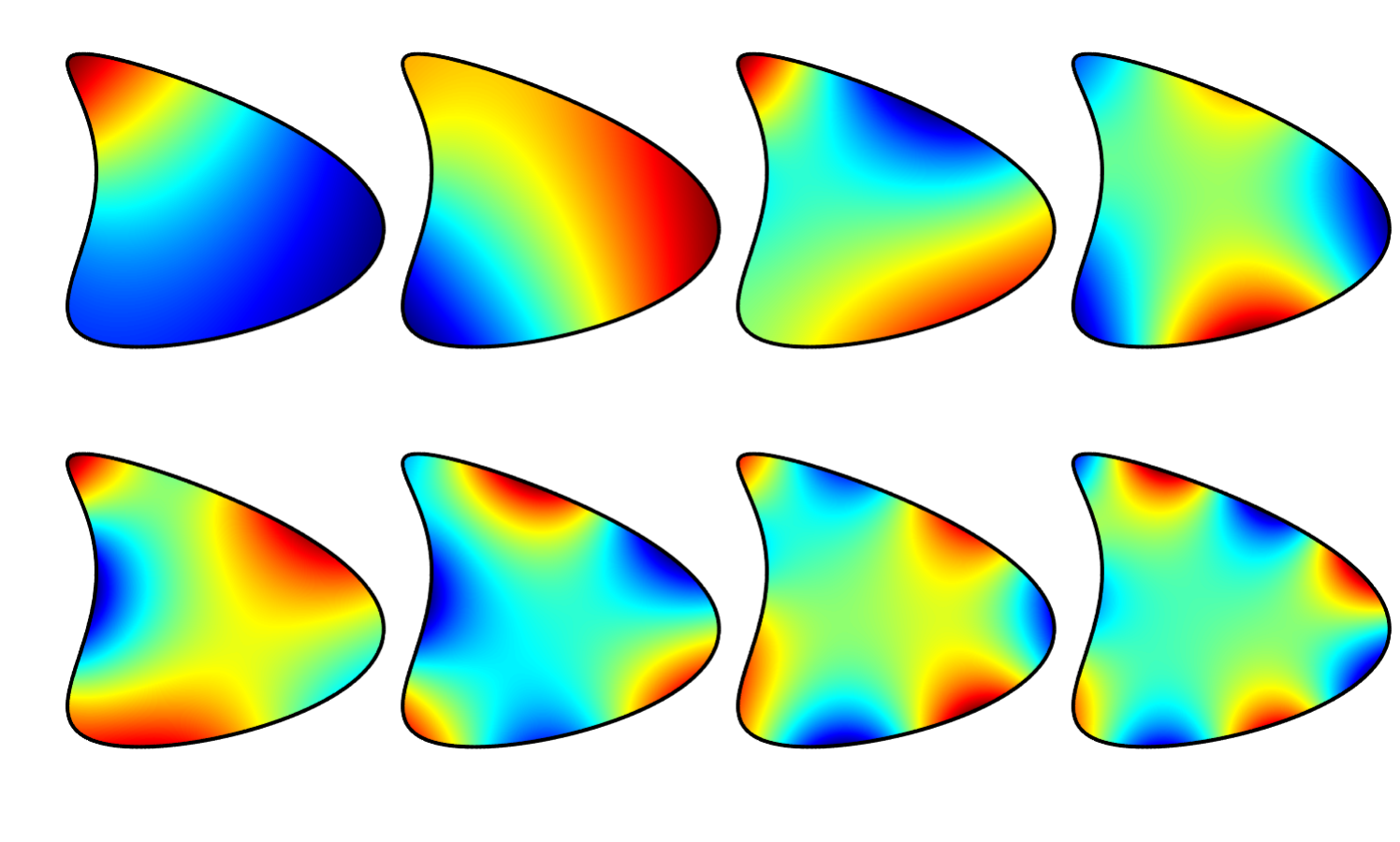}}
\hfill
\scalebox{0.32}{\includegraphics[trim=0.0cm 0.0cm 0.0cm 0.0cm,clip]{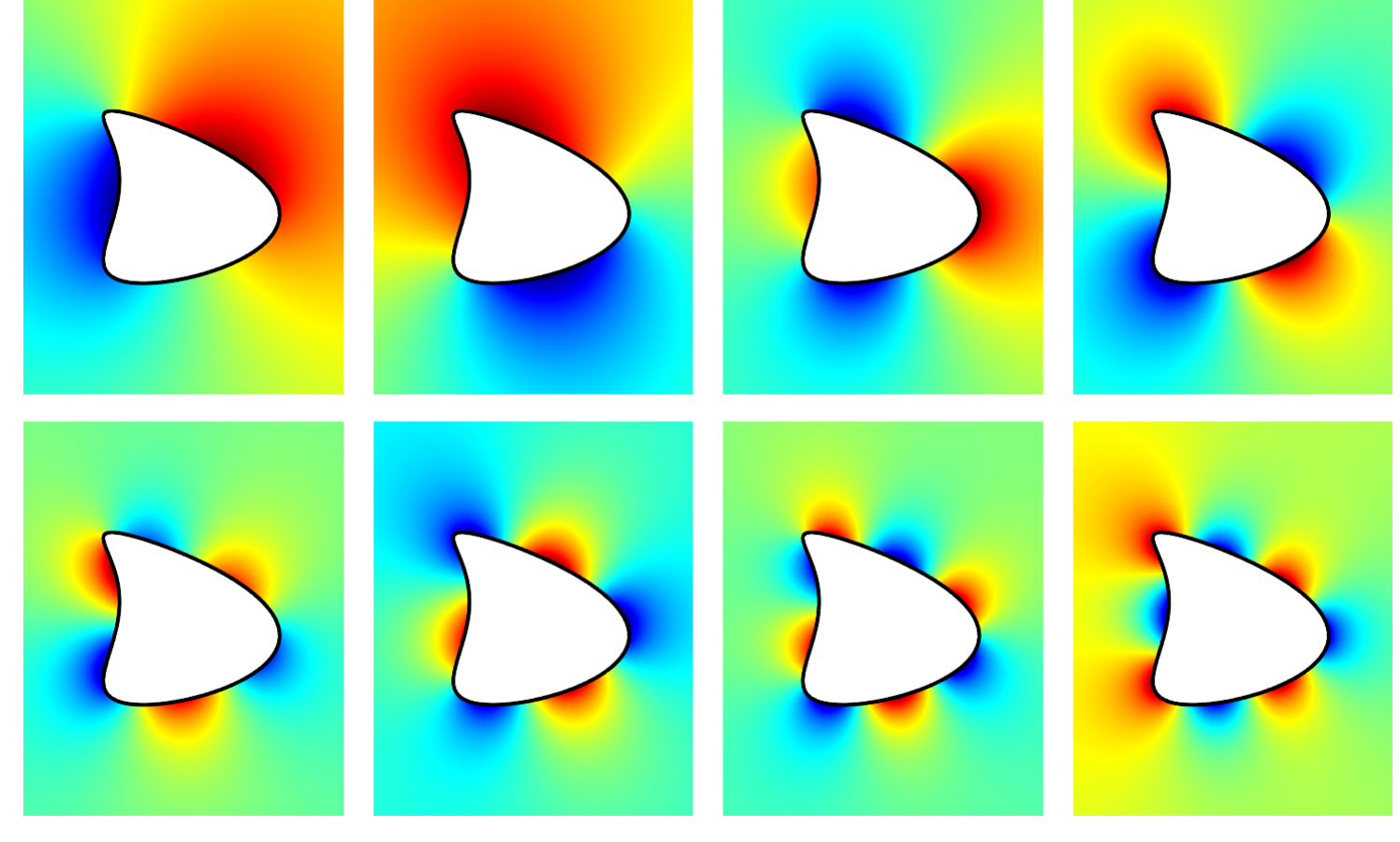}}
\hfill}
\caption{Eigenmodes for the first $8$ nonzero eigenvalues for the bounded domain $G_1$ (left) and the unbounded domain $G_2$ (right) in Example~\ref{ex:2}. }
\label{fig:kite_eigf}
\end{figure}

\begin{figure}[htb] %
\centerline{
\hfill
\scalebox{0.35}{\includegraphics[trim=0.5cm 1.5cm 0.5cm 1.0cm,clip]{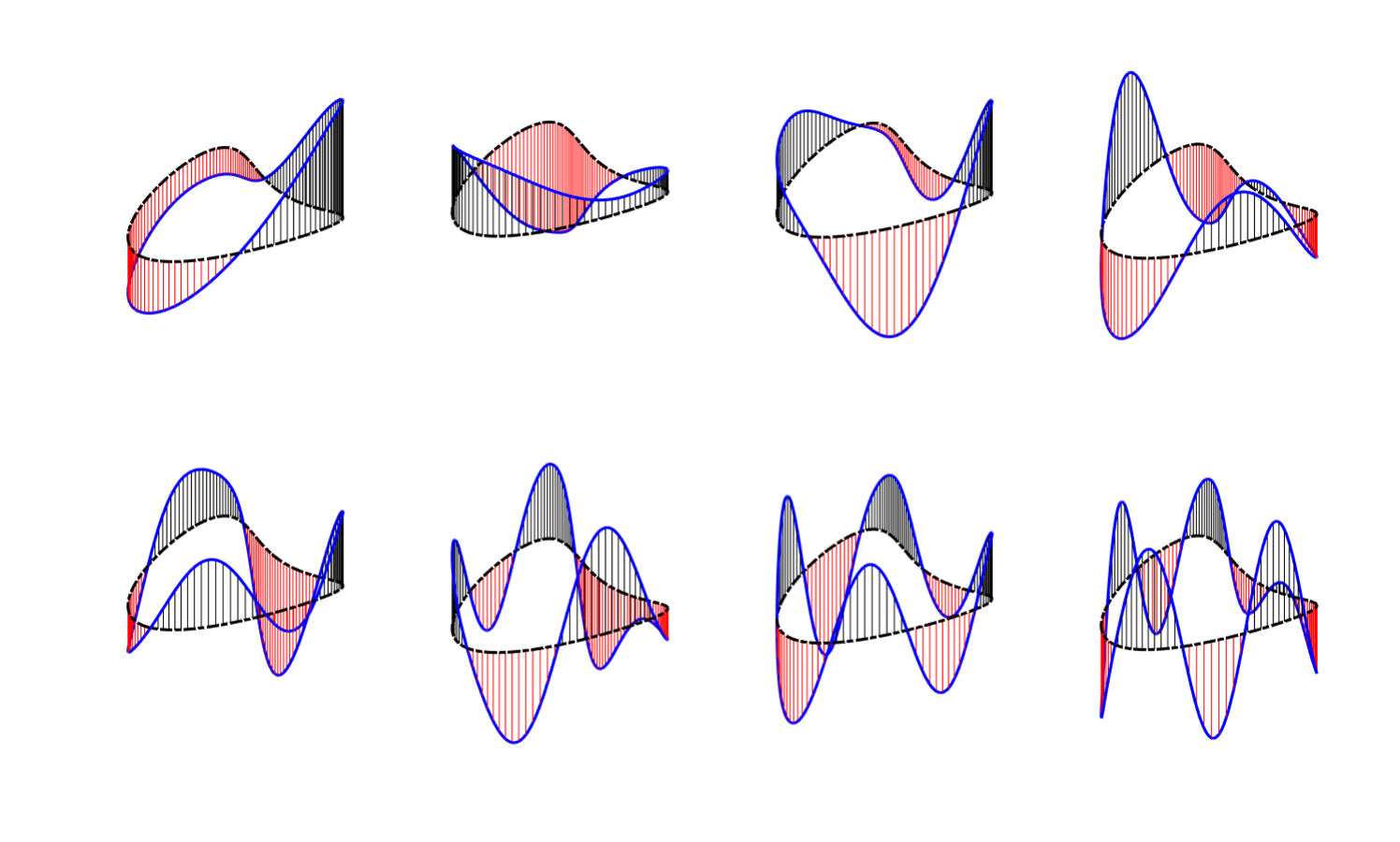}}
\hfill
\scalebox{0.35}{\includegraphics[trim=0.5cm 1.5cm 0.5cm 1.0cm,clip]{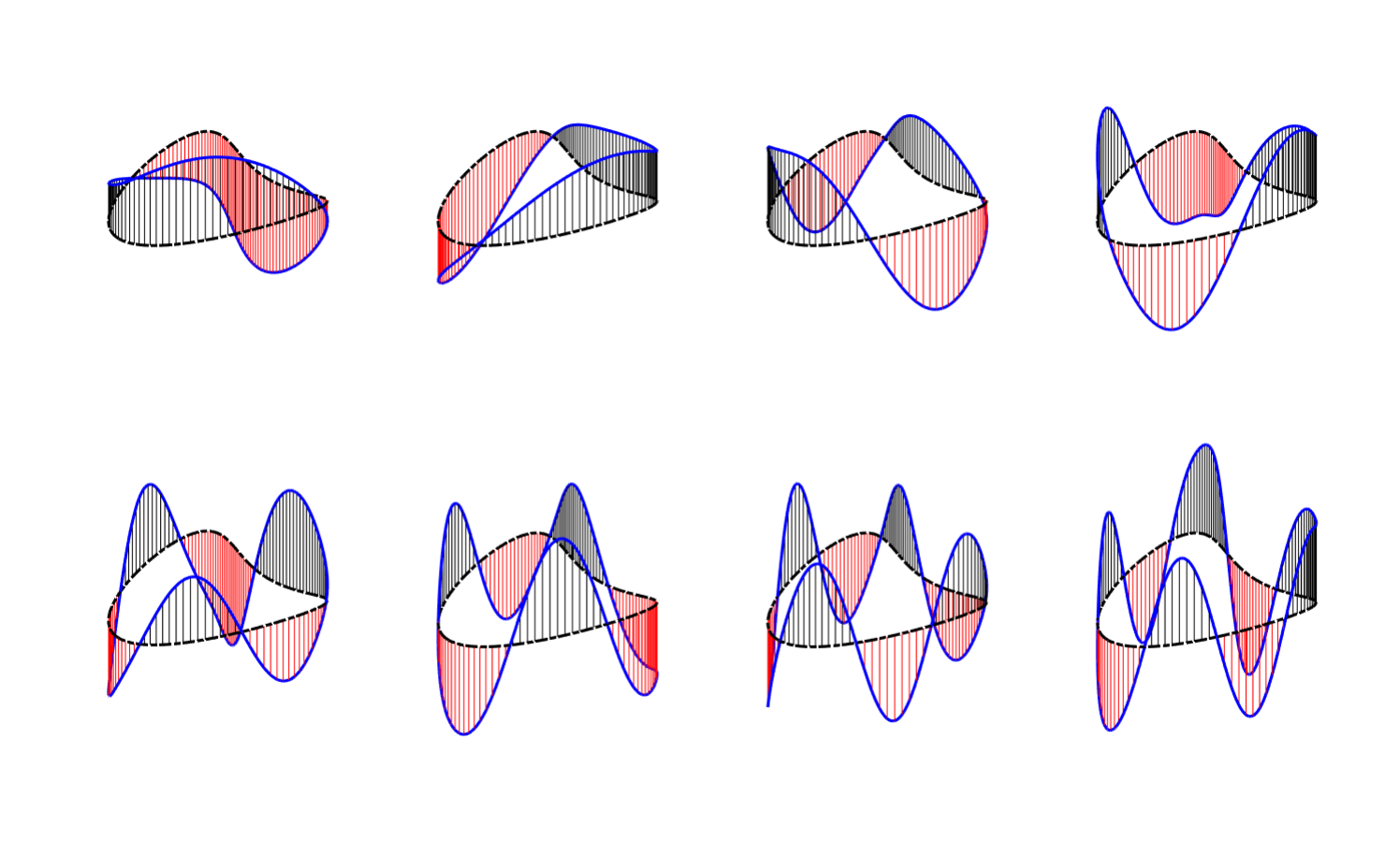}}
\hfill}
\caption{Boundary traces of the eigenmodes corresponding to the first $8$ nonzero eigenvalues for the bounded domain $G_1$ (left) and the unbounded domain $G_2$ (right) in Example~\ref{ex:2}. }
\label{fig:kite_eigf_b}
\end{figure}

\subsection{Parameter-dependent domains}

In this subsection, we consider two parameter-dependent families of smooth
domains. First, we investigate the eigenvalues and eigenfunctions for the
simply connected domains in the interior and the exterior of an ellipse. We
study the effect of the semiaxes ratio of the ellipse on the eigenvalues for
both the interior and the exterior problems. Next, we consider a symmetric
star-like family whose parametrization involves a parameter controlling the
geometry. We study the effect of this parameter on the eigenvalues for the
domains interior and exterior to this curve.

\begin{example}\label{ex:ell}
Let $\Gamma$ be the ellipse with the length of the minor axis $2a$ and the length of the major axis $2ar$ for $r\ge 1$. 
We consider the bounded simply connected domain $G_1$ and the unbounded simply connected domain $G_2$ in the interior and the exterior of $\Gamma$, respectively. 
The curve $\Gamma$ is parametrized for the bounded domain $G_1$ by 
\begin{equation}
\eta(t) = a(\cos t +\i r \sin t), \quad 0\le t\le 2\pi,
\end{equation}
and for the unbounded domain $G_2$ by 
\begin{equation}
\eta(t) = a(\cos t -\i r \sin t), \quad 0\le t\le 2\pi.
\end{equation}
Hence $\Gamma$ is oriented counterclockwise for the domain $G_1$ and clockwise for the domain $G_2$.
\end{example}

In this example, we fix the length of $\Gamma$ to be $2\pi$ and study the
effect of varying $r$ on the first $10$ nonzero eigenvalues and their corresponding
eigenmodes.
The length of the curve $\Gamma$ can be calculated by $|\Gamma| = \int_0^{2\pi} |\eta'(t)|dt = aI$ where $I=\int_0^{2\pi}\sqrt{\cos^2 t + r^2 \sin^2 t}dt$. Thus, for a given value of $r$, the value of $a$ is given by $a = 2\pi/I$ where $I$ can be approximated by $I \approx (2\pi/n) \sum_{j=1}^{n} \sqrt{\cos^2 t_j + r^2 \sin^2 t_j}$ with the equidistant grid points $t_j$ given by~\eqref{eq:tj}.
For large values of $r$, the ellipse $\Gamma$ will be thin, so that large values of $n$ need to be used to obtain accurate results using the proposed method. 
We compute the  first $10$ nonzero eigenvalues $0<\lambda_1\le\cdots\le\lambda_{10}$ as functions of $r$ with $1\le r\le 10$ for both domains $G_1$ and $G_2$.
We choose $n=2^{10}$ for $1 \le r \le 5$ and $n= 2^{11}$ for $5 < r \le 10$.  
The obtained results are shown in Figure~\ref{fig:ex_ell_eig10}.

\begin{figure}[ht] %example_ellipse_r10.m, example_ellipse_U_r10.m
\centerline{
\hfill
\scalebox{0.35}{\includegraphics{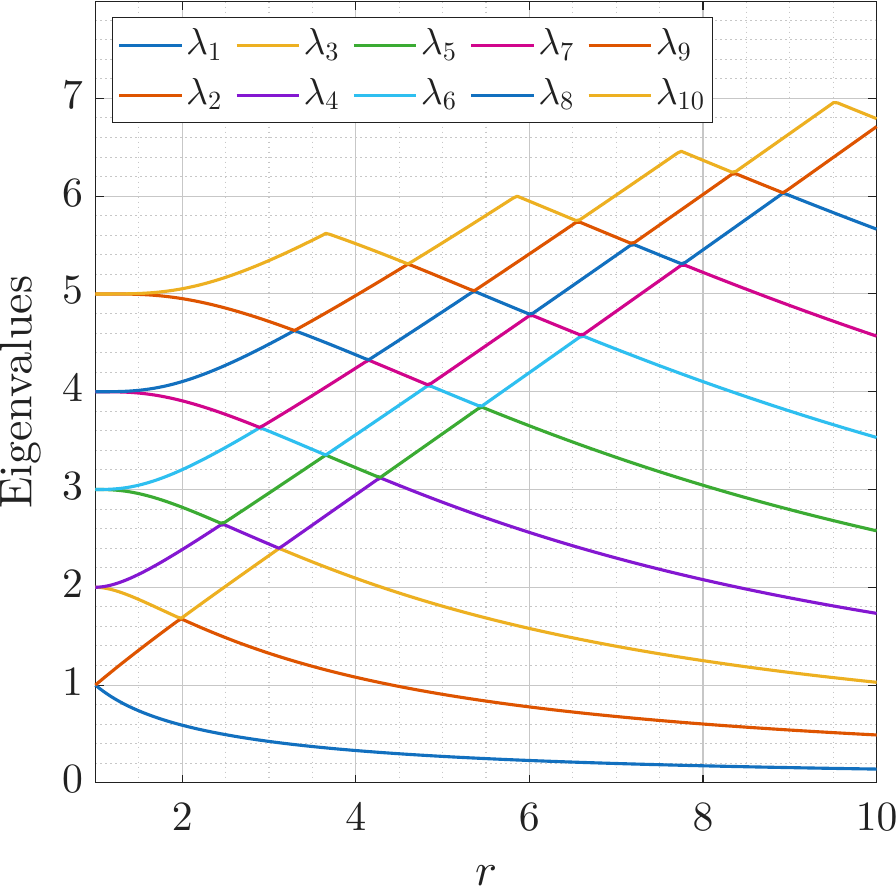}}
\hfill
\scalebox{0.35}{\includegraphics{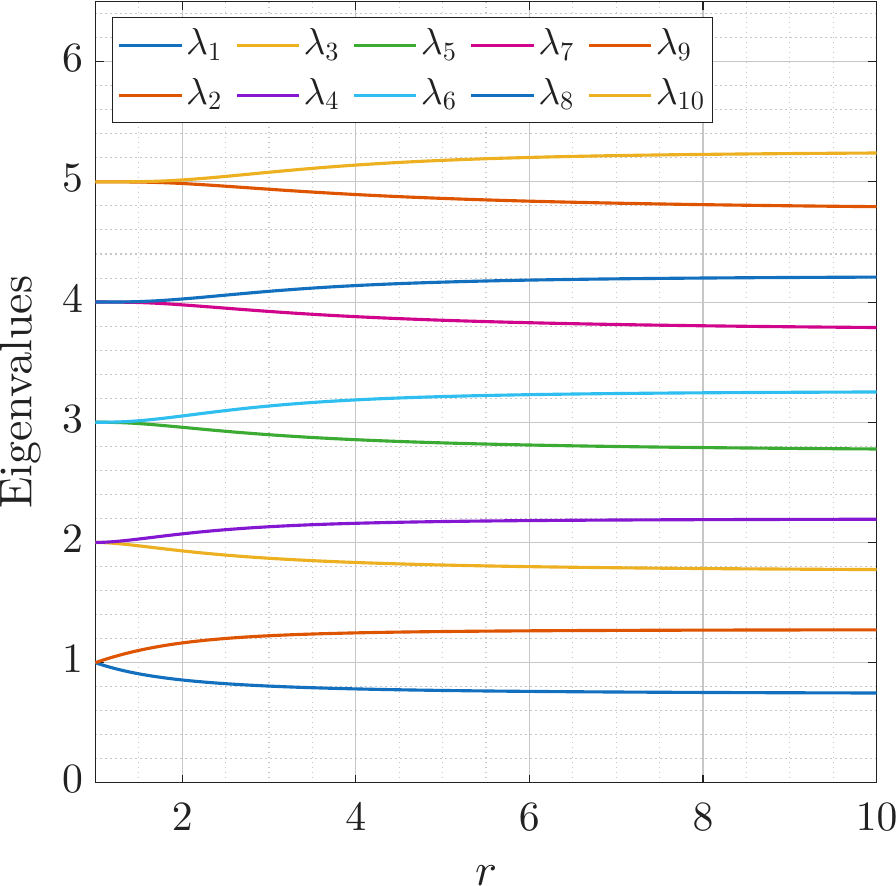}}
\hfill}
\caption{The first 10 nonzero eigenvalues as functions of $r$ for the bounded domain $G_1$ (left) and the unbounded domain $G_2$ (right) in Example~\ref{ex:ell}}
\label{fig:ex_ell_eig10}
\end{figure}

The relative error for each of the approximate values $\lambda_{k,n}$ of the first $10$ nonzero Steklov eigenvalues for $20\le n\le 400$ and $r=1,2,5,10$ is presented in Figure~\ref{fig:ell_err} where the approximate eigenvalues $\lambda_k$ calculated with $n=2^{11}$ grid points are considered as the exact values.
It is clear that the error increases as $r$ increases and the error is approximately $10^{-14}$ when $n \ge 150$ for both domains for the considered values of $r$. 
Figure~\ref{fig:ell_time} presents the number of iterations and the CPU time (sec) required for the convergence of the MATLAB function {\tt eigs} for the first $10$ nonzero eigenvalues for $r=1,2,5,10$ as well as the condition number of the matrix $Q+I$ for both domains $G_1$ and $G_2$. It is clear that the number of iterations in {\tt eigs} is almost independent of $n$ and $r$.

\begin{figure}[ht] 
\centering
\subfloat[$r=1$]{
\scalebox{0.24}{\includegraphics{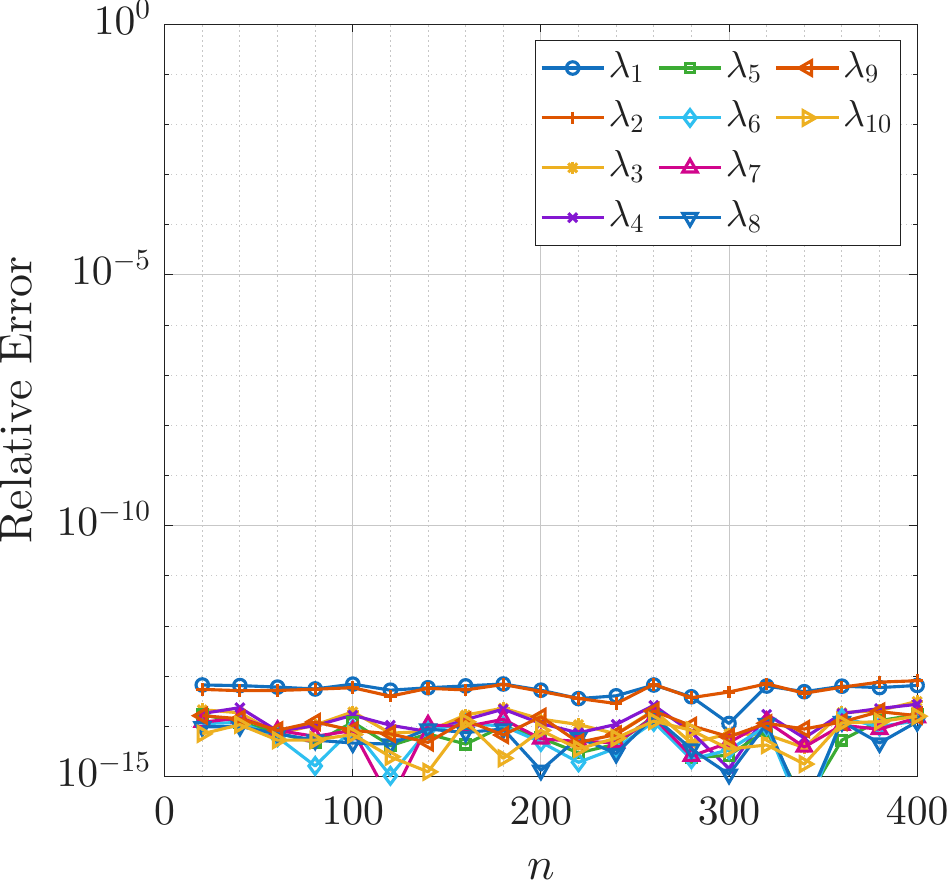}}
}
\subfloat[$r=2$]{
\scalebox{0.24}{\includegraphics{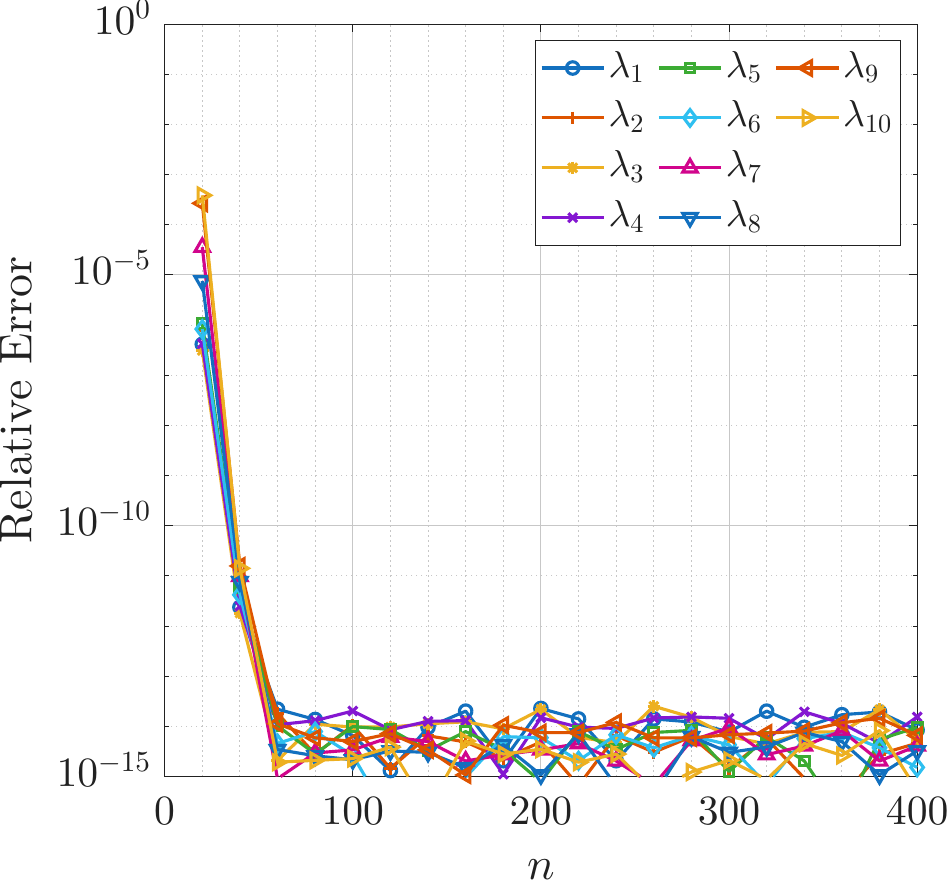}}
}
\subfloat[$r=5$]{
\scalebox{0.24}{\includegraphics{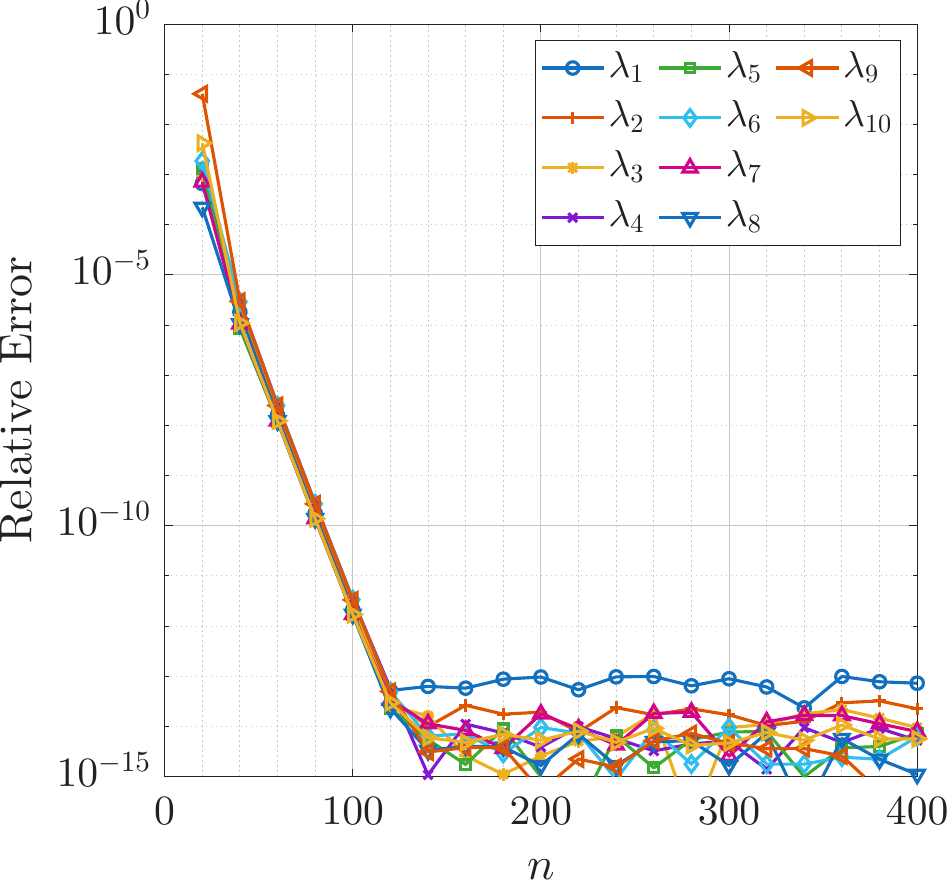}}
}
\subfloat[$r=10$]{
\scalebox{0.24}{\includegraphics{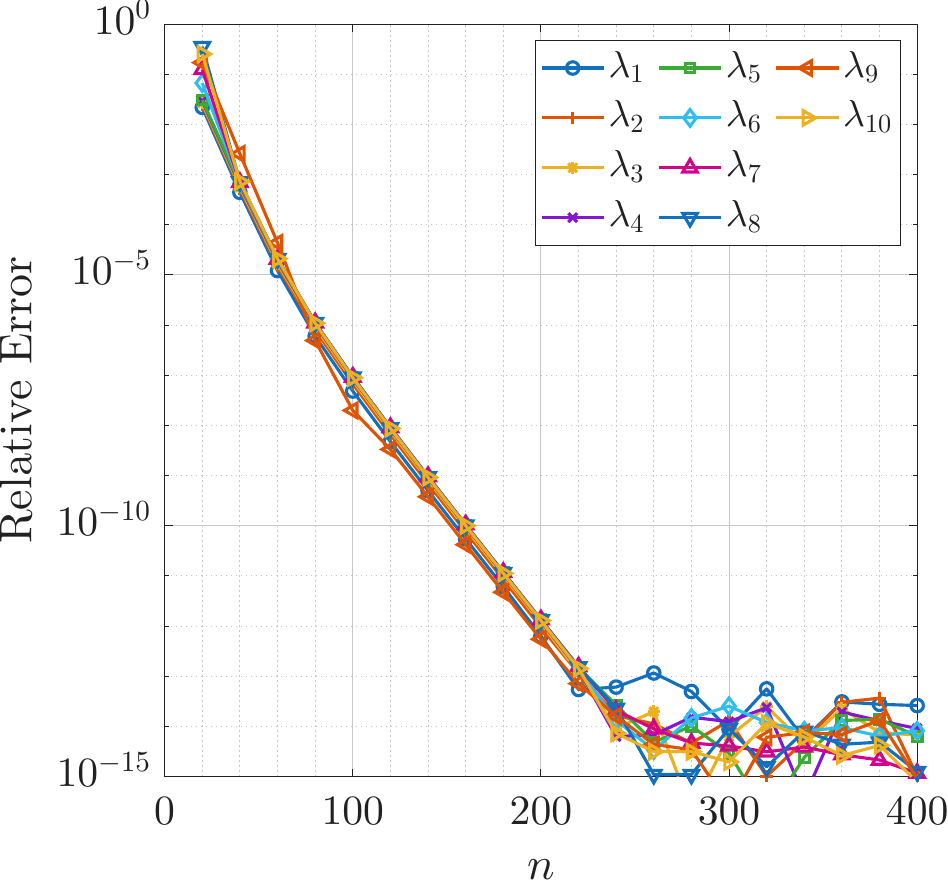}}
}\\
\centering
\subfloat[$r=1$]{
\scalebox{0.24}{\includegraphics{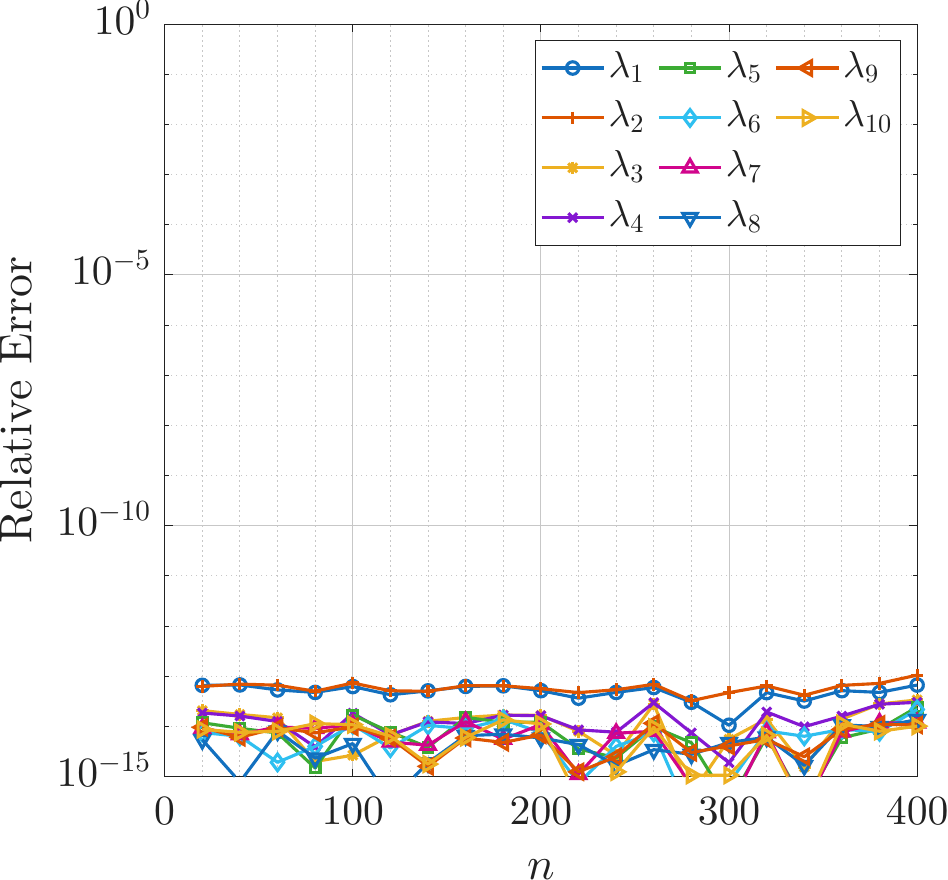}}
}
\subfloat[$r=2$]{
\scalebox{0.24}{\includegraphics{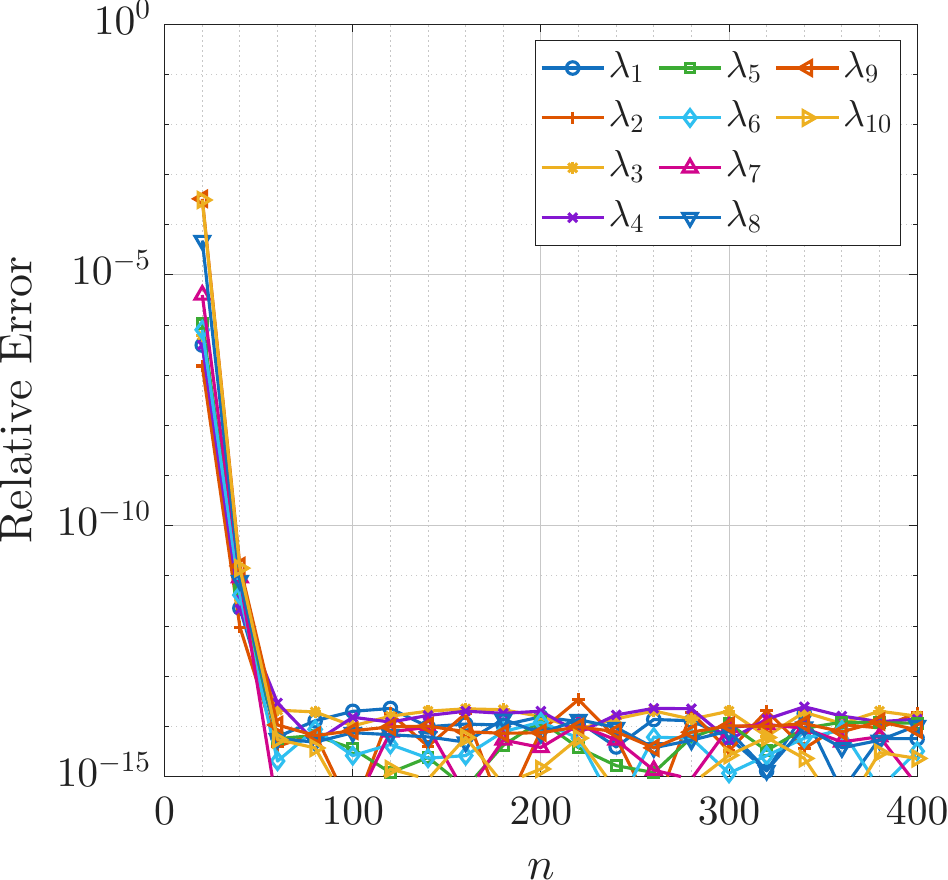}}
}
\subfloat[$r=5$]{
\scalebox{0.24}{\includegraphics{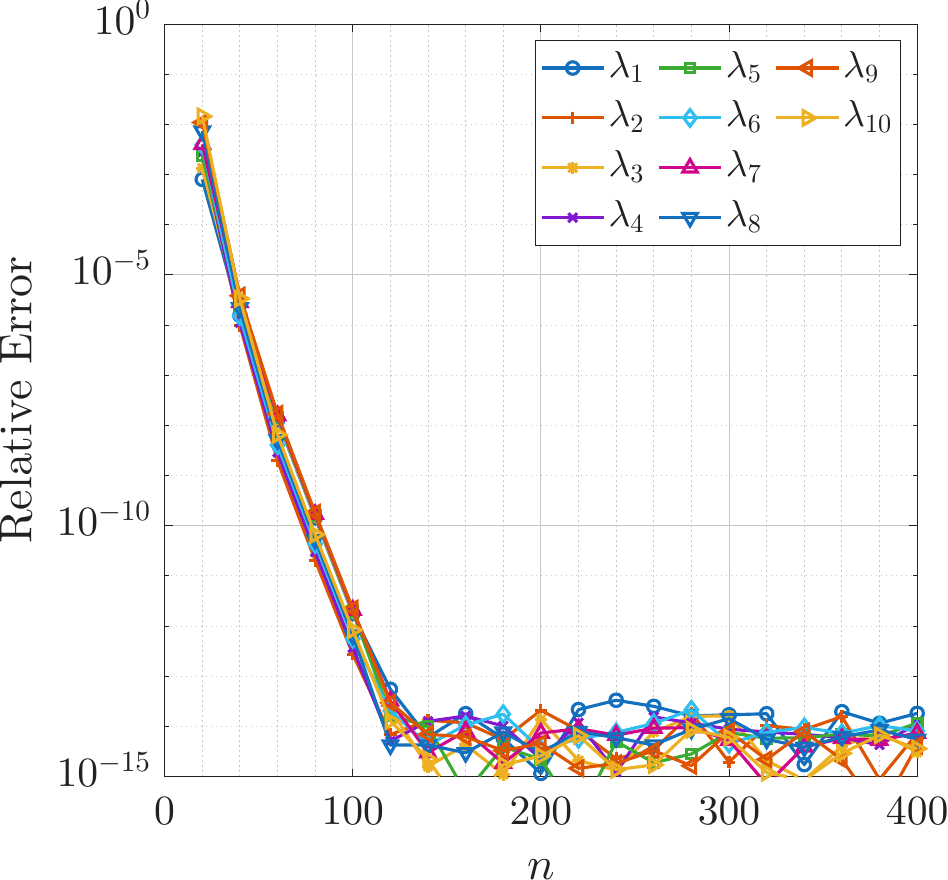}}
}
\subfloat[$r=10$]{
\scalebox{0.24}{\includegraphics{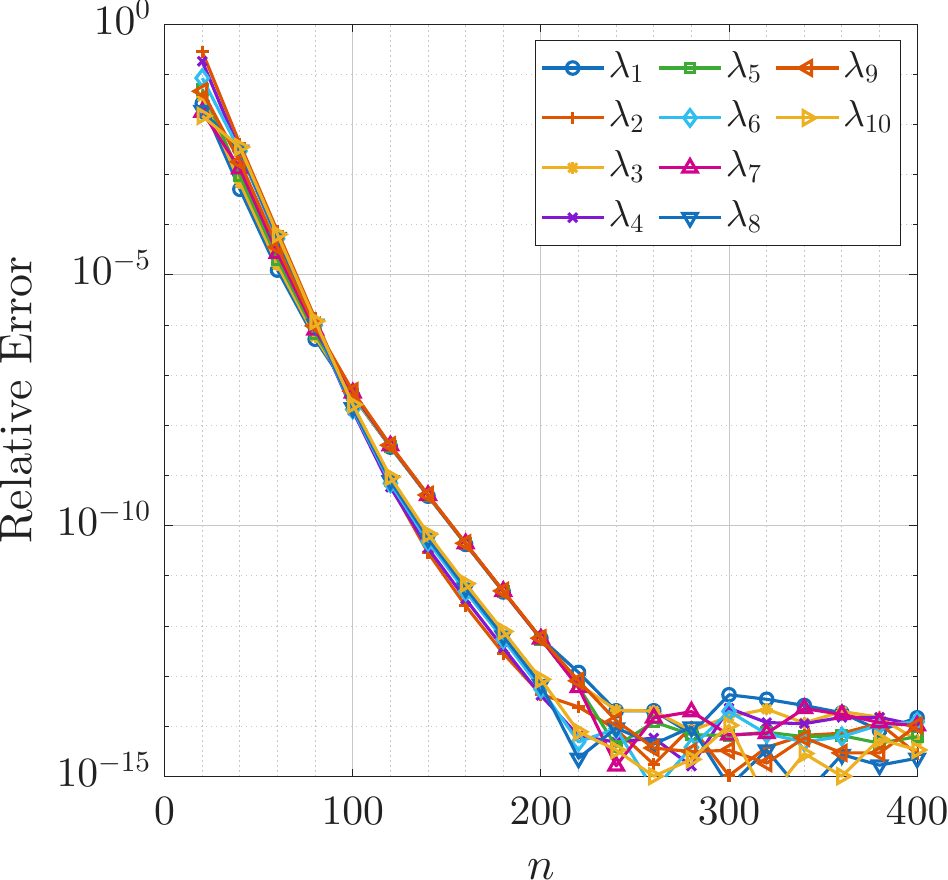}}
}
\caption{The relative error for the first $10$ nonzero eigenvalues for the bounded domain $G_1$ (first row) and the unbounded domain $G_2$ (second row) in Example~\ref{ex:ell}.}
\label{fig:ell_err}
\end{figure}

\begin{figure}[ht] 
\centerline{
\hfill
\scalebox{0.35}{\includegraphics{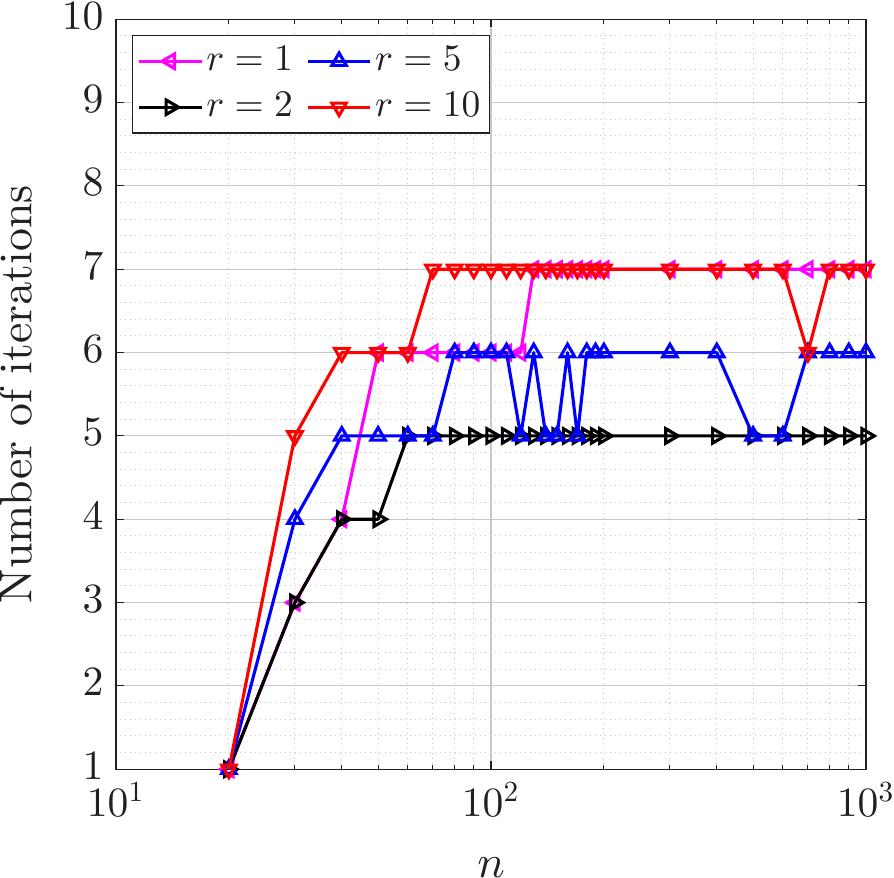}}
\hfill
\scalebox{0.35}{\includegraphics{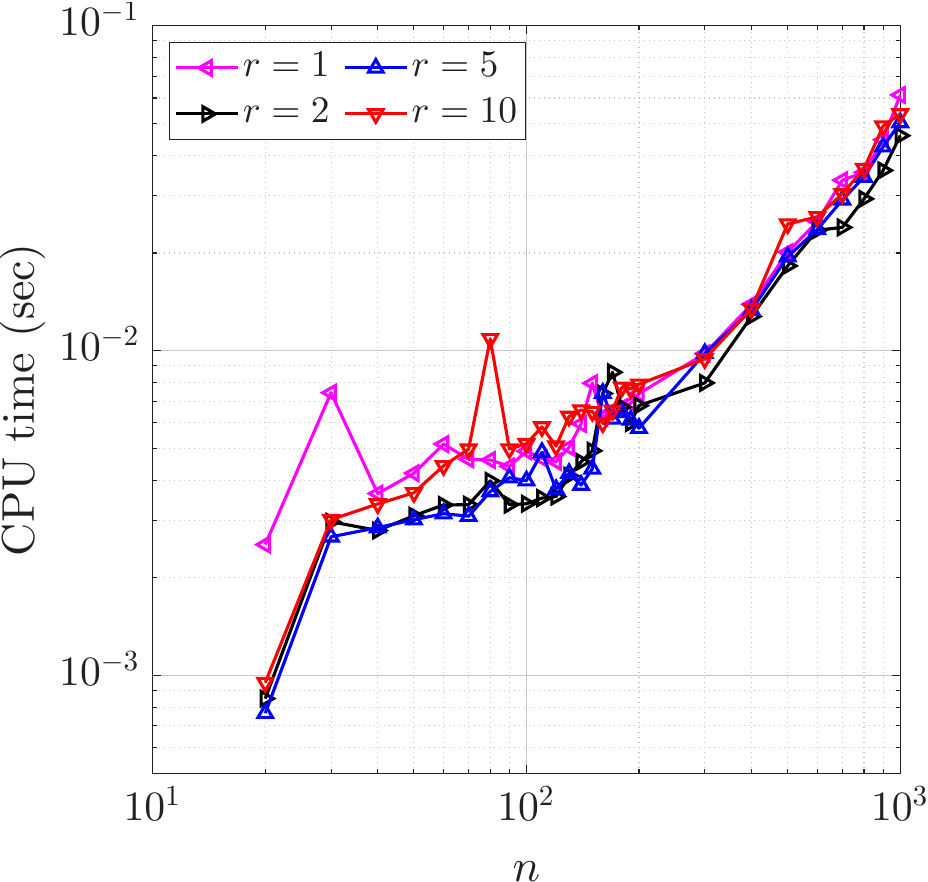}}
\hfill
\scalebox{0.35}{\includegraphics{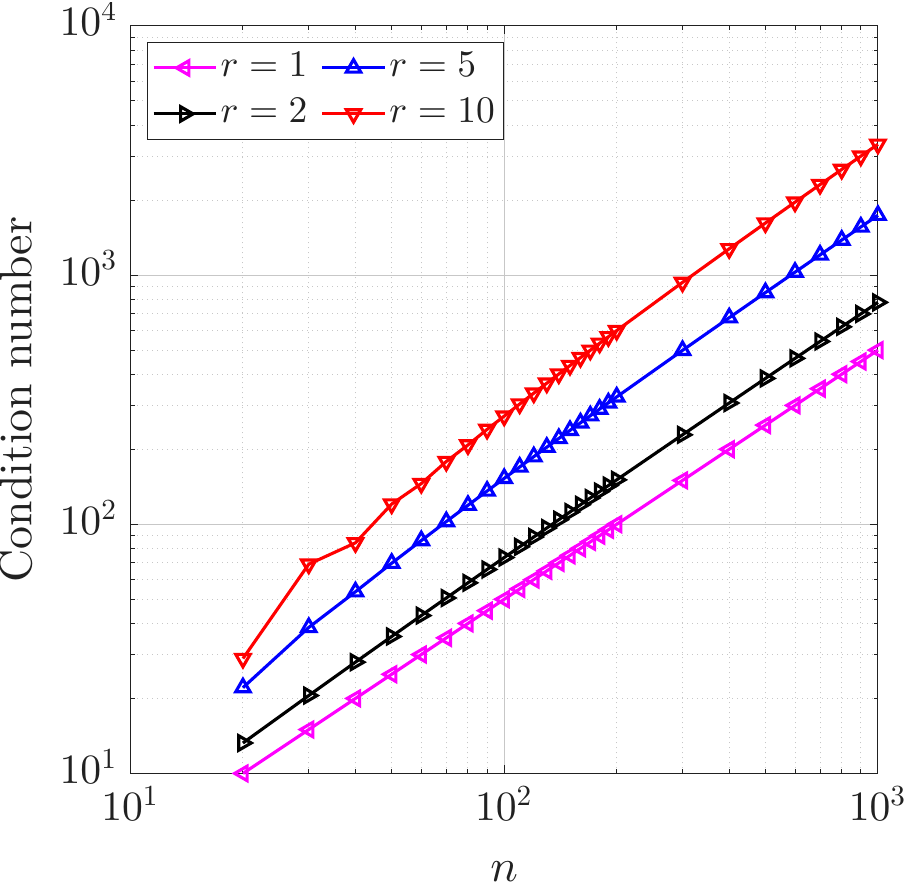}}
\hfill}		
\centerline{
\hfill
\scalebox{0.35}{\includegraphics{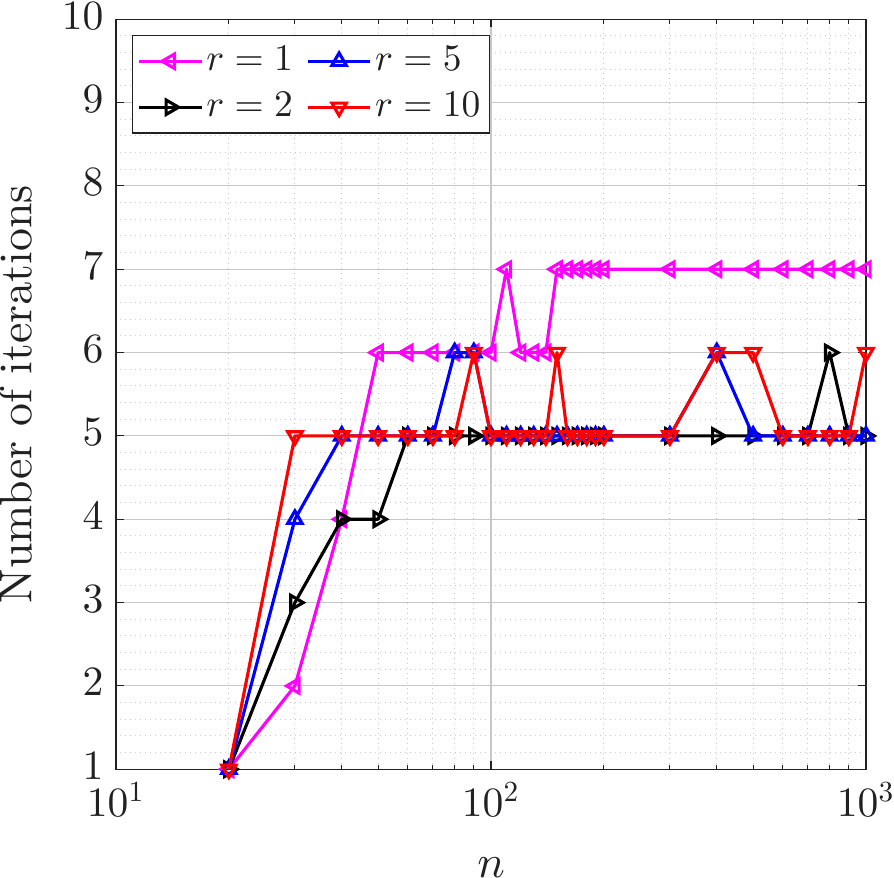}}
\hfill
\scalebox{0.35}{\includegraphics{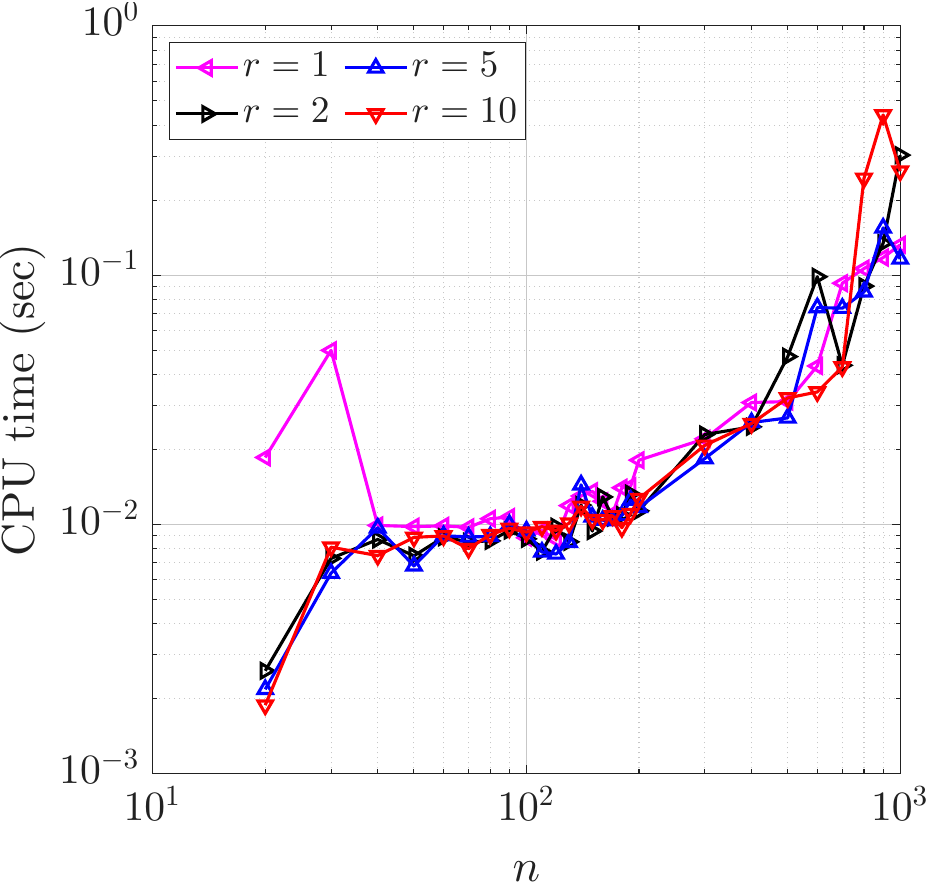}}
\hfill
\scalebox{0.35}{\includegraphics{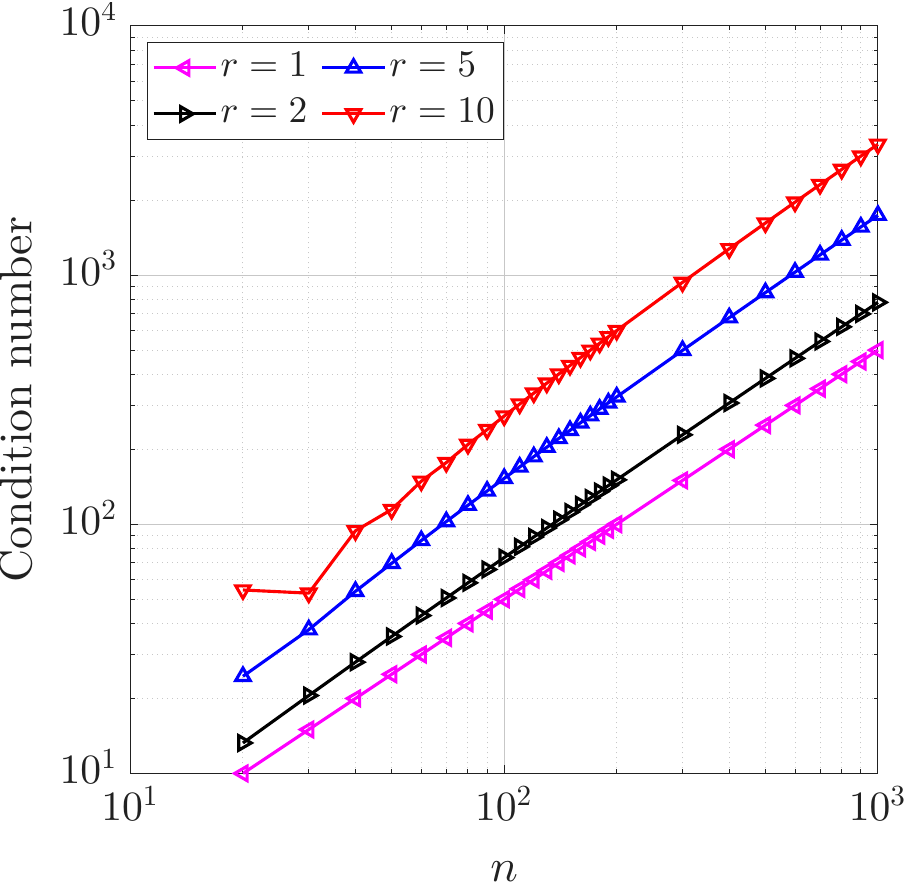}}
\hfill}
\caption{The number of iterations and CPU time (sec) required for the convergence of the MATLAB function {\tt eigs} for the first $10$ nonzero eigenvalues and the condition number of the matrix $Q+I$ for the bounded domain $G_1$ (first row) and the unbounded domain $G_2$ (second row) in Example~\ref{ex:ell}.}
\label{fig:ell_time}
\end{figure}

For the bounded domain $G_1$, Figure~\ref{fig:ex_ell_eig10} (left) shows that
none of the eigenvalues $\lambda_k(r)$ is monotone for $2\le k\le 10$. The same
plot displays a repeated pattern of near-equalities between consecutive ordered
eigenvalues. At these near-equality points, the sorted branches change local
monotonicity. It is known that $\lambda_2=\lambda_1=1$ for $r=1$ ($\Gamma$ is
the unit circle). The first $10$ nonzero eigenvalues for the bounded domain
$G_1$ are presented in Table~\ref{tab:ell} for several values of $r>1$ at which
two consecutive eigenvalues are approximately equal. These values of $r$ are
obtained using the MATLAB function {\tt fminsearch}. That is, for a given
integer $k$, $1\le k\le 9$, we use {\tt fminsearch} to find the smallest value
of $r$ such that $\lambda_k(r)\approx\lambda_{k+1}(r)$. For example,
Table~\ref{tab:ell} shows that $\lambda_2\approx\lambda_3$ for
$r=1.9838737083599$ and $\lambda_3\approx\lambda_4$ for
$r=3.11781174187898$. Numerically, Table~\ref{tab:ell} suggests that the
smallest value of $r$ for which $\lambda_k(r)\approx\lambda_{k+1}(r)$ is a
candidate extremum of the ordered branch $\lambda_k(r)$ within the ellipse
family considered here.

Because the eigenvalues are sorted independently for each value of $r$, the
label $\lambda_k(r)$ does not necessarily follow a single mode family through a
near-crossing. Accordingly, Figure~\ref{fig:ex_ell_B} should be interpreted as
displaying a reordering of eigenpairs in the sorted spectrum rather than the
creation or disappearance of modes. The computed eigenmodes for
$r=1,2,2.6,3$ suggest that when two consecutive eigenvalues become nearly equal,
the associated mode families exchange their order in the sorted list.
Near a degeneracy, the natural object to track is the two-dimensional
invariant subspace spanned by the two eigenpairs, rather than the
individual eigenvalue branches. This tracking could be based on the
boundary traces, 
for example (see Figure~\ref{fig:kite_eigf_b}).

\begin{table}[ht]
\caption{The first $10$ nonzero eigenvalues $\lambda_k$ for the bounded domain $G_1$ in Example~\ref{ex:ell} obtained with $n=2^{10}$ for the values of $r$ where two eigenvalues are approximately equal}\label{tab:ell}
\centering
\setlength{\tabcolsep}{3pt}
\resizebox{\linewidth}{!}{%
\begin{tabular}{lllll}
\hline\noalign{\smallskip}
  & $r=1.983873708359900$ & $r=3.117811741879000$ & $r=4.278336858589900$ & $r=5.442032493985020$ \\
\noalign{\smallskip}\hline\noalign{\smallskip}
$\lambda_1$    & 0.593266465889889     & 0.409978130744910     & 0.311183312003288     & 0.250165499055797 \\
$\lambda_2$    &{\bf1.679239176823338} & 1.291478177842050     & 1.025328241153468     & 0.843631750753444 \\
$\lambda_3$    &{\bf1.679239176835823} &{\bf2.397950880456617} & 2.006048996879334     & 1.699333353247237 \\
$\lambda_4$    & 2.376468343983926     &{\bf2.397950880459712} &{\bf3.120717747646763} & 2.727167519936677 \\
$\lambda_5$    & 2.823647075109909     & 3.033405492077505     &{\bf3.120717747647146} &{\bf3.846560542514552} \\
$\lambda_6$    & 3.196199653195032     & 3.552794534587696     & 3.727869417156279     &{\bf3.846560542533858} \\
$\lambda_7$    & 3.910548250630405     & 3.751799170367676     & 4.274966135900190     & 4.435413954635979 \\
$\lambda_8$    & 4.099500756310629     & 4.541254463038616     & 4.400200548856085     & 4.999147477717563 \\
$\lambda_9$    & 4.954275068888483     & 4.682238868628503     & 5.131245651835014     & 5.077862758644241 \\
$\lambda_{10}$ & 5.050493629797804     & 5.388160175669312     & 5.419827437992173     & 5.770427540965376 \\
\noalign{\smallskip}\hline\noalign{\smallskip}
  & $r=6.604561045328200$ & $r=7.765356504288500$ & $r=8.924566918140600$ & $r=10.082445969980300$ \\
\noalign{\smallskip}\hline\noalign{\smallskip}
$\lambda_1$    & 0.208930000416793     & 0.179263245830117     & 0.156922373331826     & 0.139504578198453 \\
$\lambda_2$    & 0.714184133772993     & 0.618063248878927     & 0.544179146623814     & 0.485759419484110 \\
$\lambda_3$    & 1.464076868706651     & 1.281462674354500     & 1.136998623125178     & 1.020497339361440 \\
$\lambda_4$    & 2.397284856072567     & 2.126678579478409     & 1.904666194092126     & 1.721059452223921 \\
$\lambda_5$    & 3.451787535101981     & 3.107096610614738     & 2.812052169861008     & 2.560651597525257 \\
$\lambda_6$    &{\bf4.574370154312591} & 4.178482178784736     & 3.823566202959877     & 3.510876455379496 \\
$\lambda_7$    &{\bf4.574370154348116} &{\bf5.303394572822056} & 4.906492271273063     & 4.544114910616505 \\
$\lambda_8$    & 5.150736926621857     &{\bf5.303394572840356} &{\bf6.033189909928315} & 5.635380478219814 \\
$\lambda_9$    & 5.725459115522663     & 5.870874261032724     &{\bf6.033189909942186} &{\bf6.763492313333715} \\
$\lambda_{10}$ & 5.772244964721325     & 6.453264041075331     & 6.594086280062486     &{\bf6.763492313344550} \\
\noalign{\smallskip}\hline
\end{tabular}}
\end{table}

For the unbounded domain $G_2$, Figure~\ref{fig:ex_ell_eig10}
(right) indicates that all eigenvalues $\lambda_k(r)$ are monotonic functions of $r\ge1$
for $k\ge1$, with $\lambda_k(r)$ decreasing for odd $k$ and increasing for even
$k$. Accordingly, we do not observe the same reordering pattern of eigenmodes as
in the bounded case. Figure~\ref{fig:ex_ell_U} presents the eigenmodes
corresponding to the first $8$ nonzero eigenvalues for the same sample values
$r=1,2,2.6,3$.

For both domains $G_1$ and $G_2$, Figure~\ref{fig:ex_ell_eig10} shows that
$\lambda_1(r)$ is decreasing for $1\le r\le 10$, which is consistent with the
well-known result that the first non-trivial Steklov eigenvalue is maximized on
the disk for simply connected planar domains with fixed perimeter~\cite{Alhe}.
For the bounded domain $G_1$ and for $2\le k\le 10$, Table~\ref{tab:ell}
records the values of $r$ at which consecutive ordered eigenvalues are
approximately equal, and hence provides candidate extremal locations for the
ordered branches within the ellipse family. For the unbounded simply connected
domain exterior to the ellipse with fixed perimeter $2\pi$ and major-to-minor
axis ratio $r$, Figure~\ref{fig:ex_ell_eig10} (right) indicates that
$\lambda_k(r)$ attains its maximum at $r=1$ for odd $k$, while the even
branches increase over the interval $1\le r\le 10$ considered here.

\begin{figure}[htb] %
\centering
\subfloat[$r=1$]{
\scalebox{0.33}{\includegraphics[trim=1cm 0.75cm 1cm 0.75cm,clip]{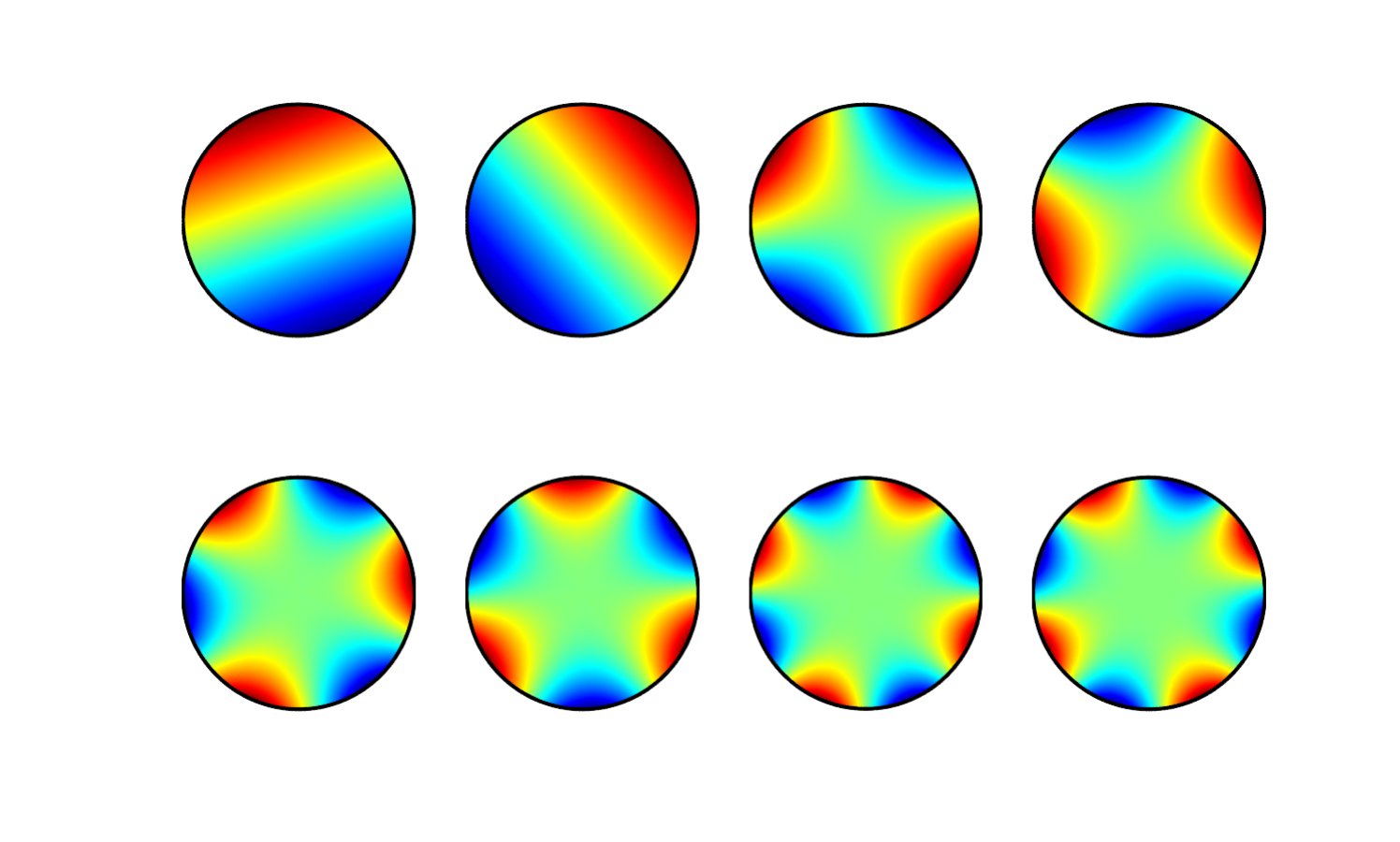}}
}
\subfloat[$r=2$]{
\scalebox{0.33}{\includegraphics[trim=1cm 0.75cm 1cm 0.75cm,clip]{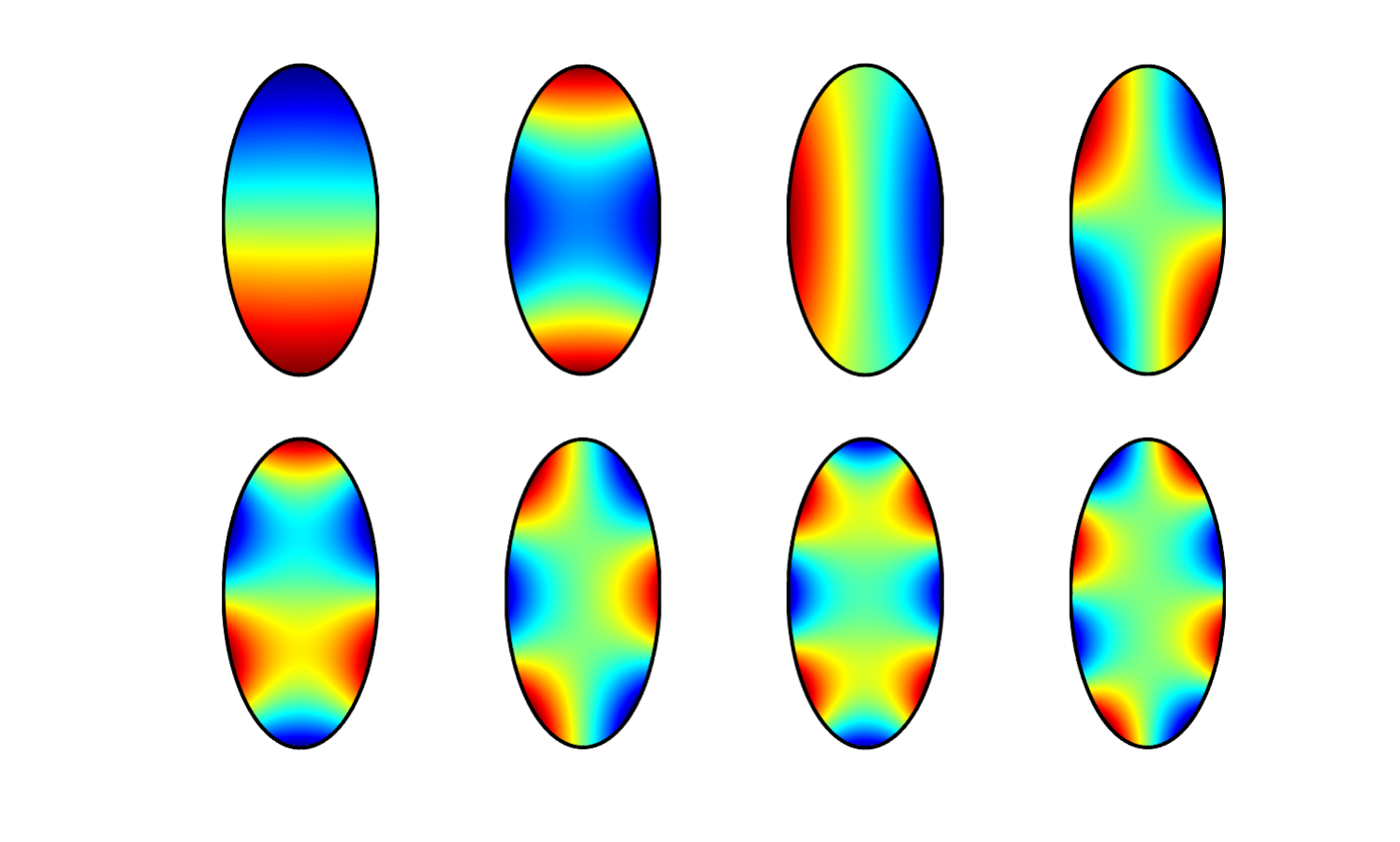}}
}\\
\subfloat[$r=2.6$]{
\scalebox{0.33}{\includegraphics[trim=1cm 0.75cm 1cm 0.75cm,clip]{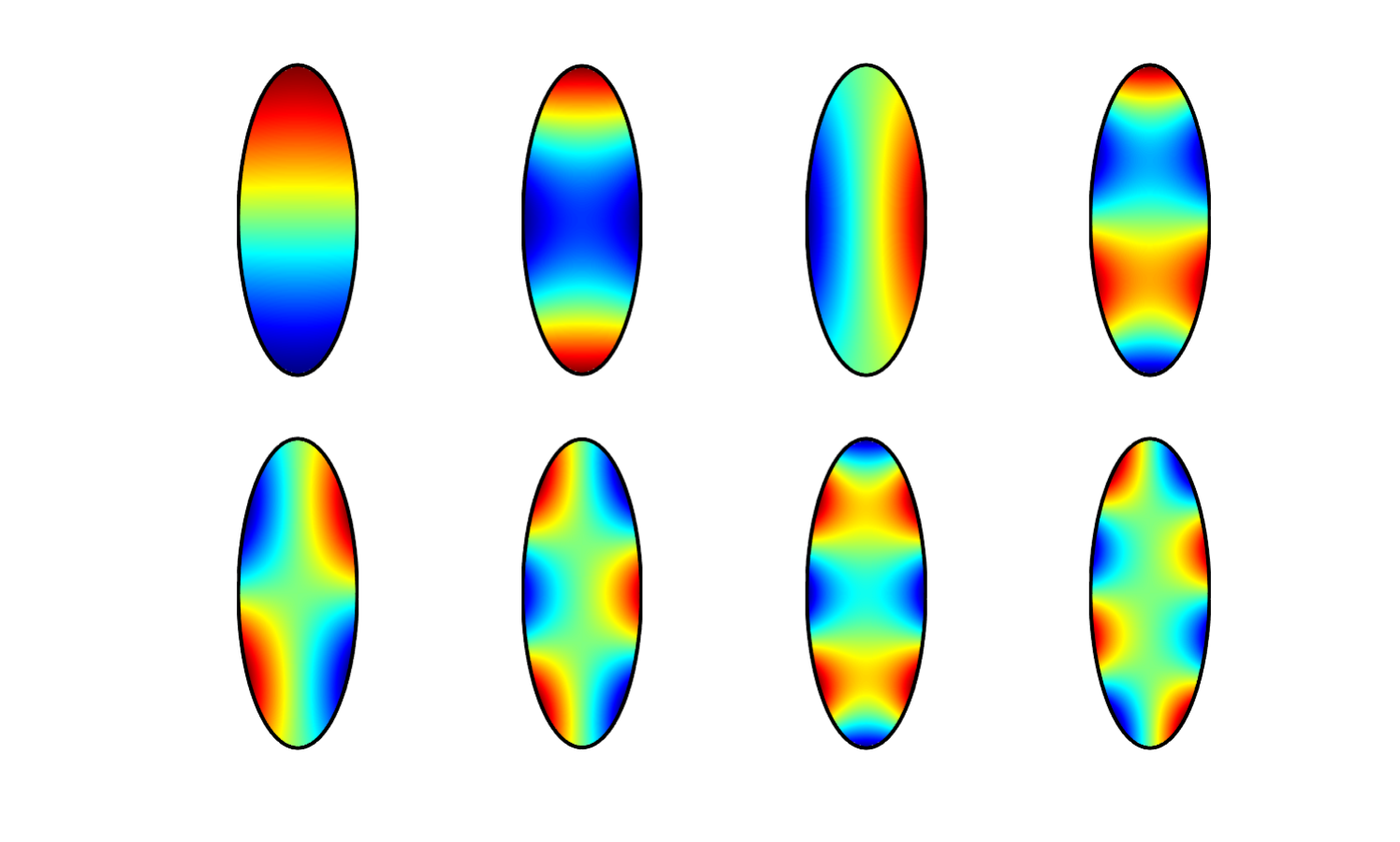}}
}
\subfloat[$r=3$]{
\scalebox{0.33}{\includegraphics[trim=1cm 0.75cm 1cm 0.75cm,clip]{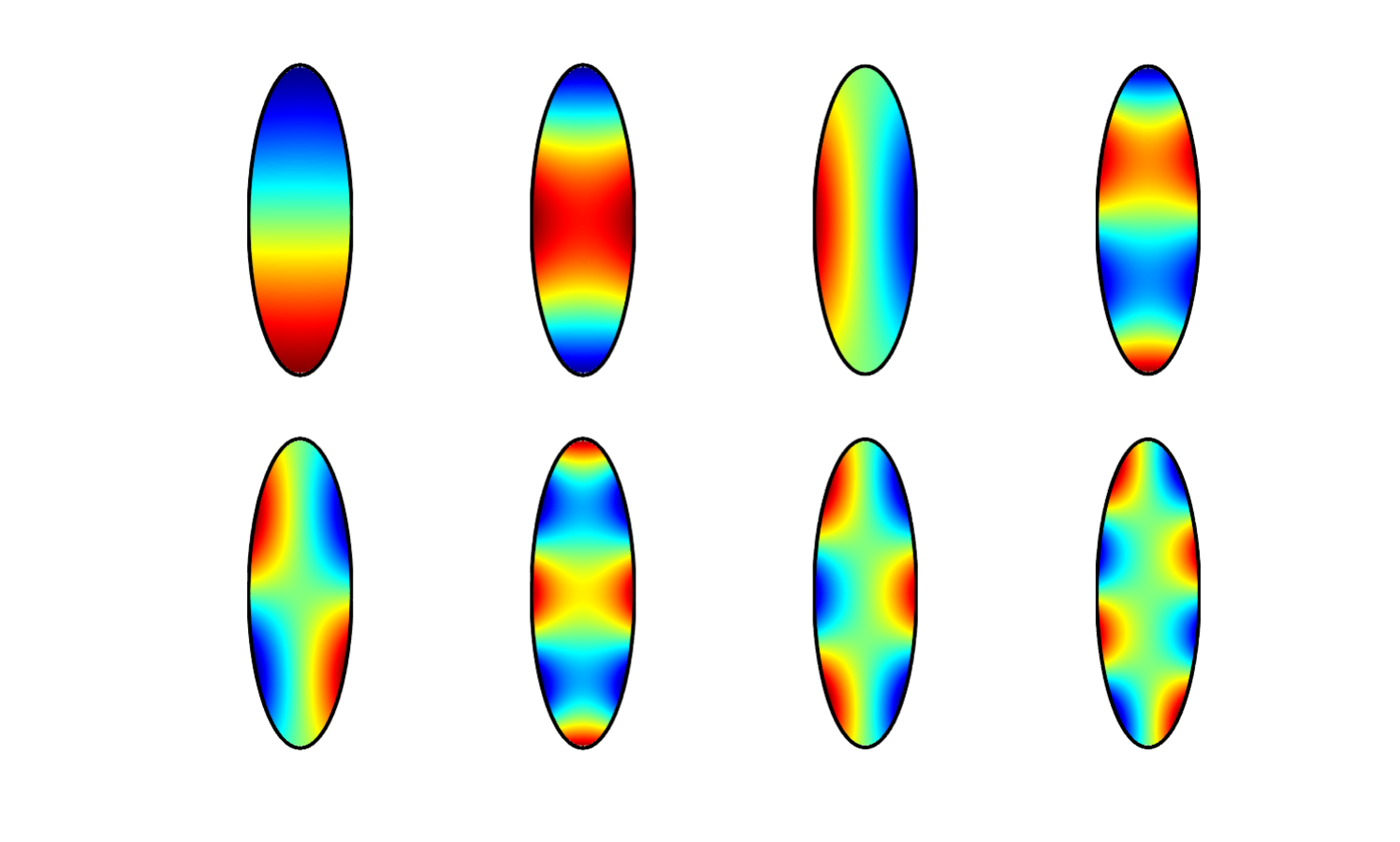}}
}
\caption{The eigenmodes of the first $8$ nonzero eigenvalues of the bounded domain $G_1$ in Example~\ref{ex:ell}.}
	\label{fig:ex_ell_B}
\end{figure}

\begin{figure}[htb] %
\centering
\subfloat[$r=1$]{
\scalebox{0.31}{\includegraphics[trim=0cm 0.0cm 0cm 0.0cm,clip]{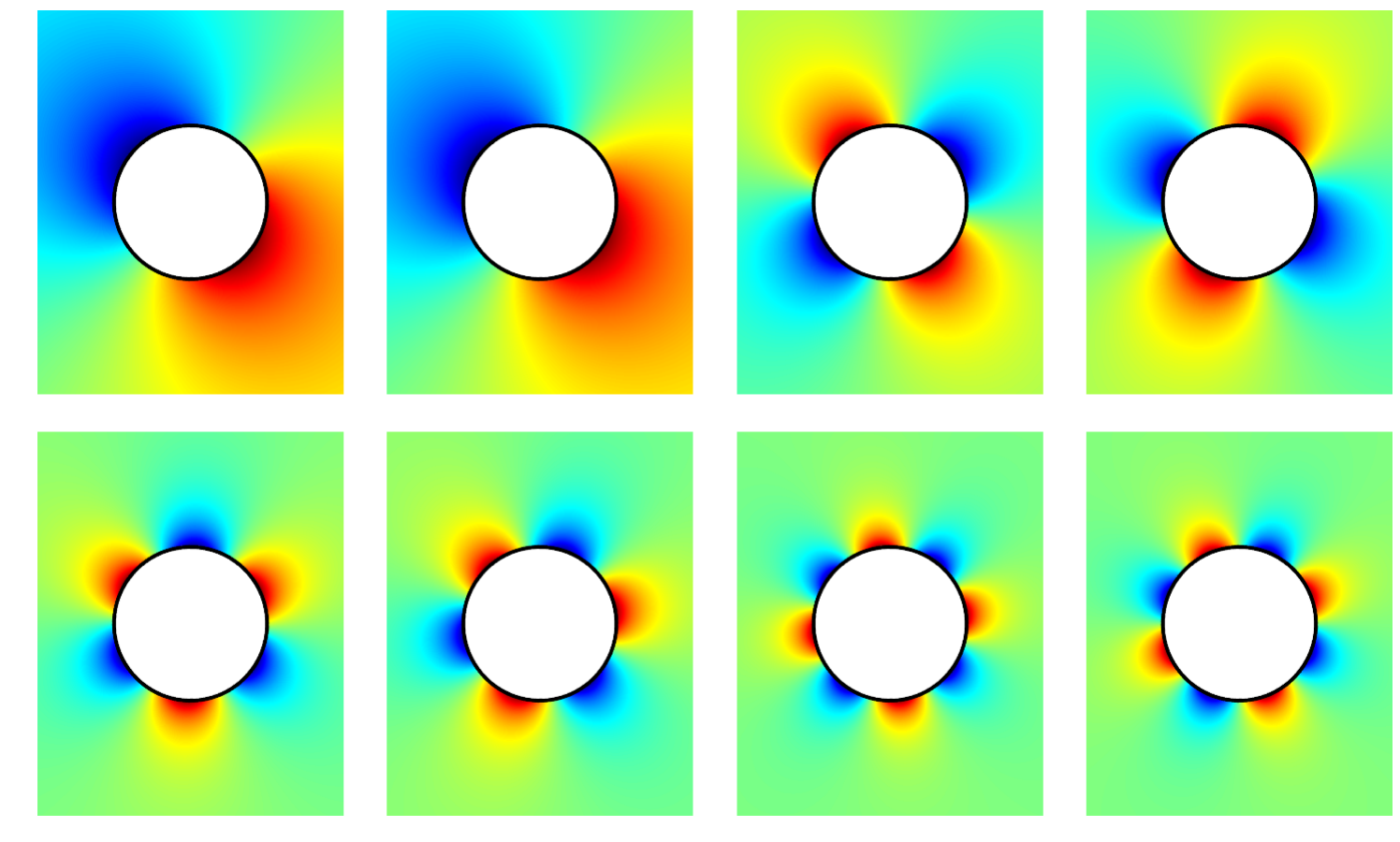}}
}
\subfloat[$r=2$]{
\scalebox{0.31}{\includegraphics[trim=0cm 0.0cm 0cm 0.0cm,clip]{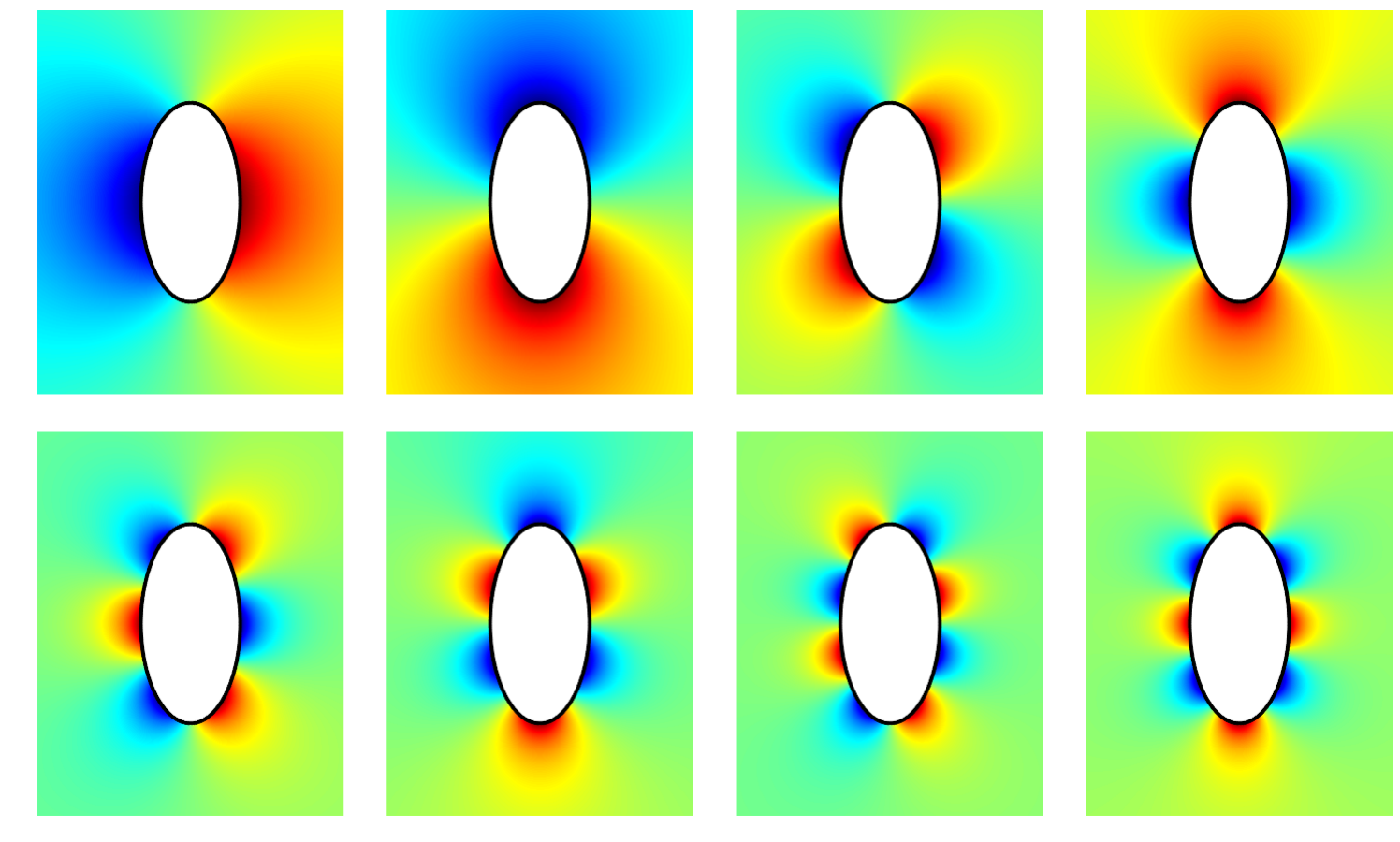}}
}\\
\subfloat[$r=2.6$]{
\scalebox{0.31}{\includegraphics[trim=0cm 0.0cm 0cm 0.0cm,clip]{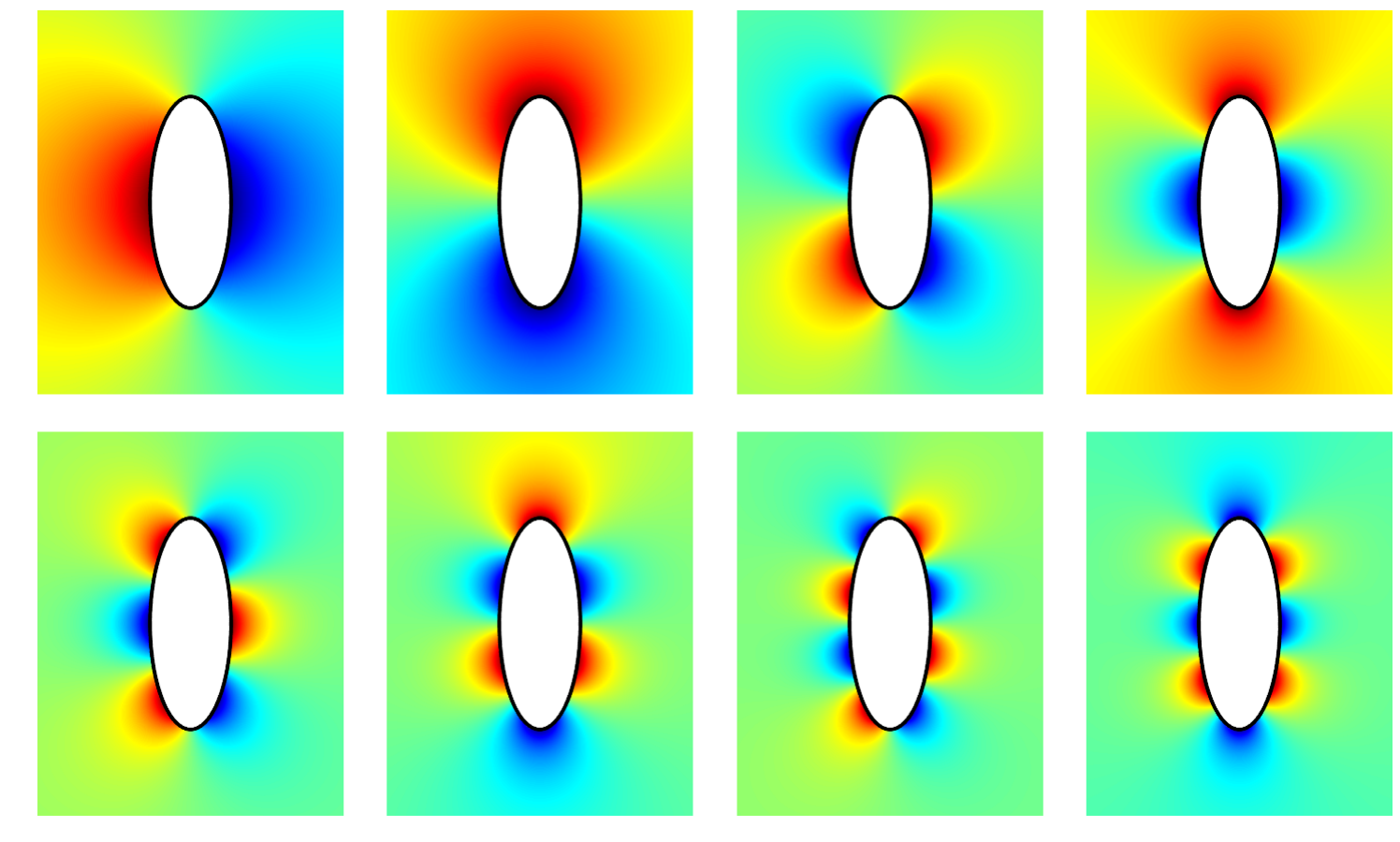}}
}
\subfloat[$r=3$]{
\scalebox{0.31}{\includegraphics[trim=0cm 0.0cm 0cm 0.0cm,clip]{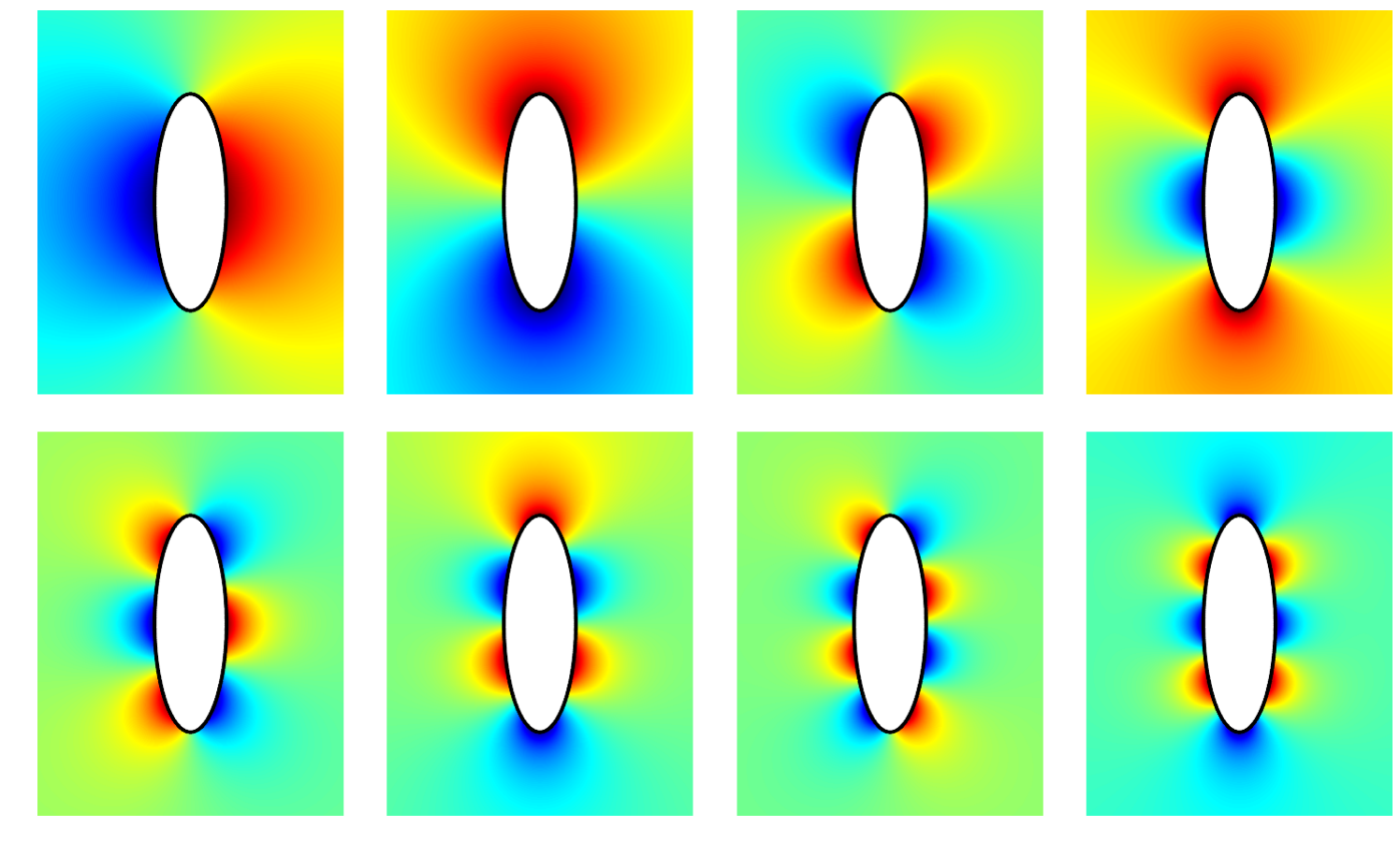}}
}
\caption{The eigenmodes of the first $8$ nonzero eigenvalues of the unbounded domain $G_2$ in Example~\ref{ex:ell}.}
	\label{fig:ex_ell_U}
\end{figure}

The presented numerical method can be used to validate some of the well-known inequalities related to Steklov eigenvalues. Recall that the boundary $\Gamma$ has a constant length $|\Gamma|=2\pi$ and the area of the bounded domain $G_1$ is $|G_1|=ra^2\pi$. For both the bounded domain $G_1$ and the unbounded domain $G_2$, the asymptotic behavior~\eqref{eq:asym_b} of eigenvalues is validated numerically for the first $100$ nonzero eigenvalues for various values of $r$ and the obtained results are presented in Figure~\ref{fig:ex_ell_eig100}.
For the bounded domain $G_1$, the first two nonzero eigenvalues $\lambda_1(r)$ and $\lambda_2(r)$ satisfy the following inequalities~\cite{Dit}
\begin{equation}\label{eq:ineq-1}
\frac{1}{\lambda_1(r)} + \frac{1}{\lambda_2(r)}  \geq 2, \quad 
\lambda_1(r)\lambda_2(r) \leq 1.
\end{equation}
These two inequalities are validated with the numerically calculated eigenvalues $\lambda_1(r)$ and $\lambda_2(r)$ for $1\le r\le 10$ in Figure~\ref{fig:ex_ell_eig1} (left).
For the unbounded domain $G_2$, the first nonzero eigenvalue $\lambda_1(r)$ satisfies the inequality~\cite{Bun}
\begin{equation}\label{eq:ineq-2}
\lambda_1(r) \le \sqrt{\frac{\pi}{|G_1|}}=\frac{1}{a\sqrt{r}}.
\end{equation}
The equality holds if and only if $G_1$ is a disk~\cite{Bun}. Figure~\ref{fig:ex_ell_eig1} (right) shows the upper bound in~\eqref{eq:ineq-2} for $1\le r \le 10$.

\begin{figure}[htb] %example_ellipse_r100.m, 
\centerline{
\hfill
\scalebox{0.35}{\includegraphics{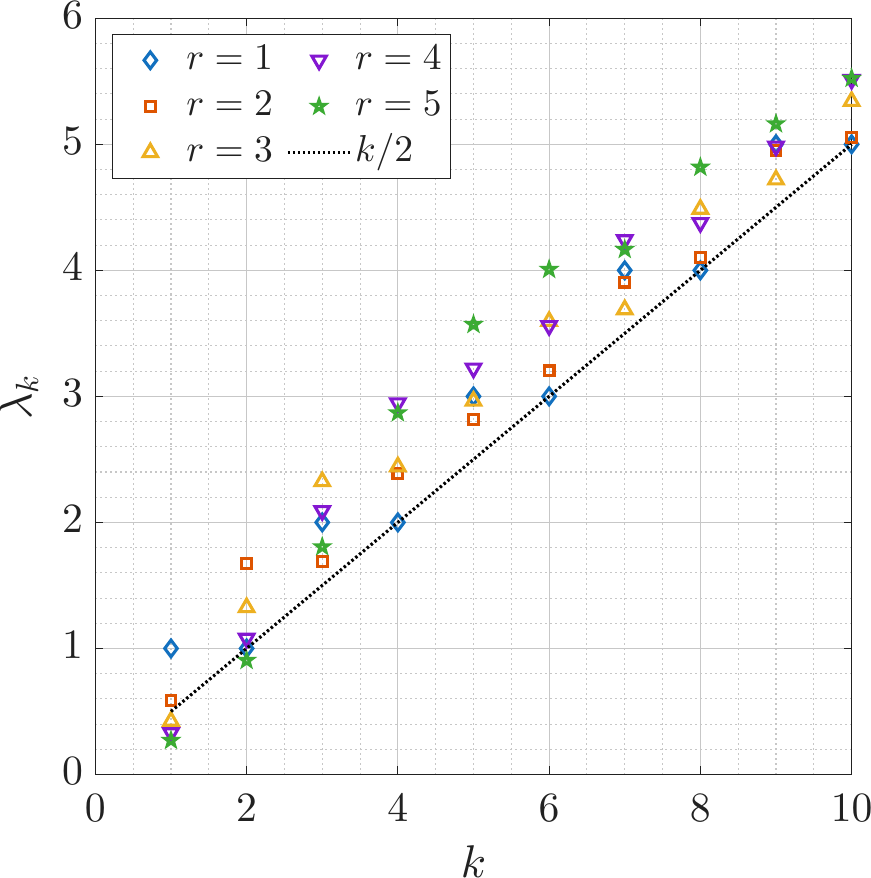}}
\hfill
\scalebox{0.35}{\includegraphics{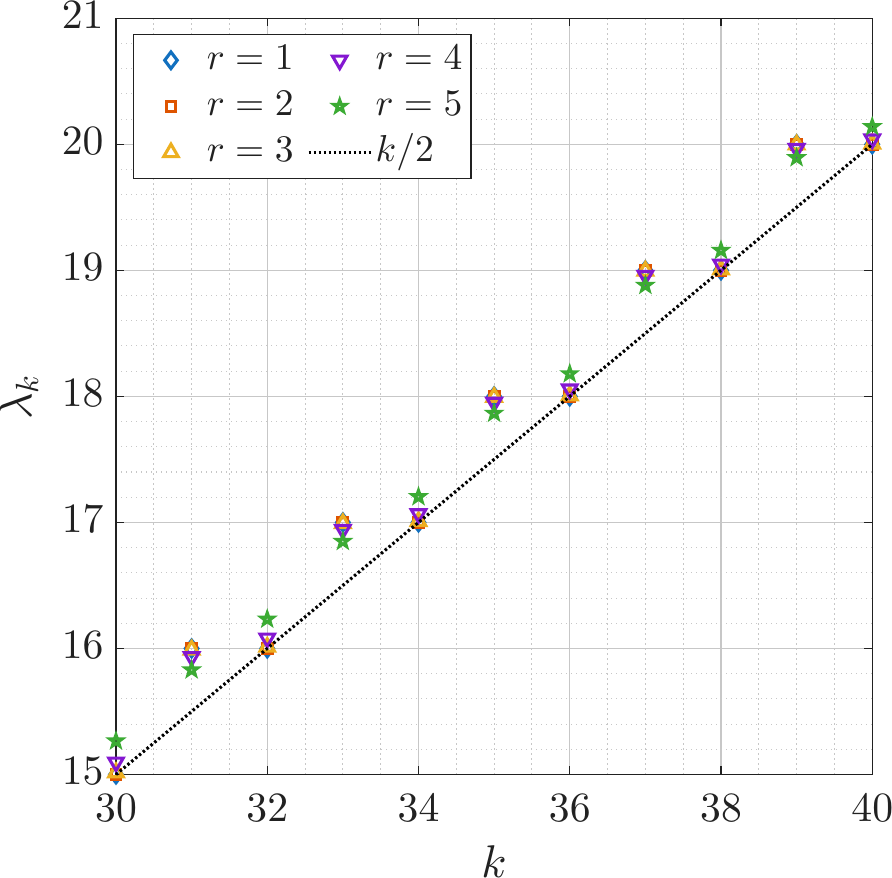}}
\hfill
\scalebox{0.35}{\includegraphics{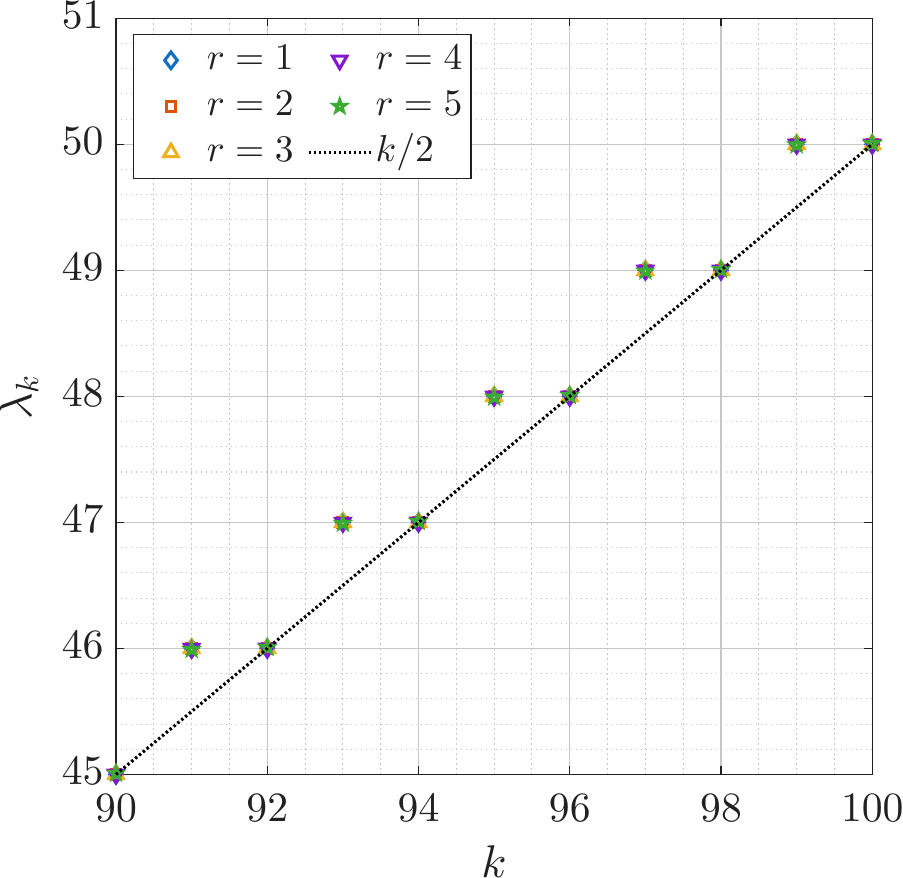}}
\hfill}
\centerline{
\hfill
\scalebox{0.35}{\includegraphics{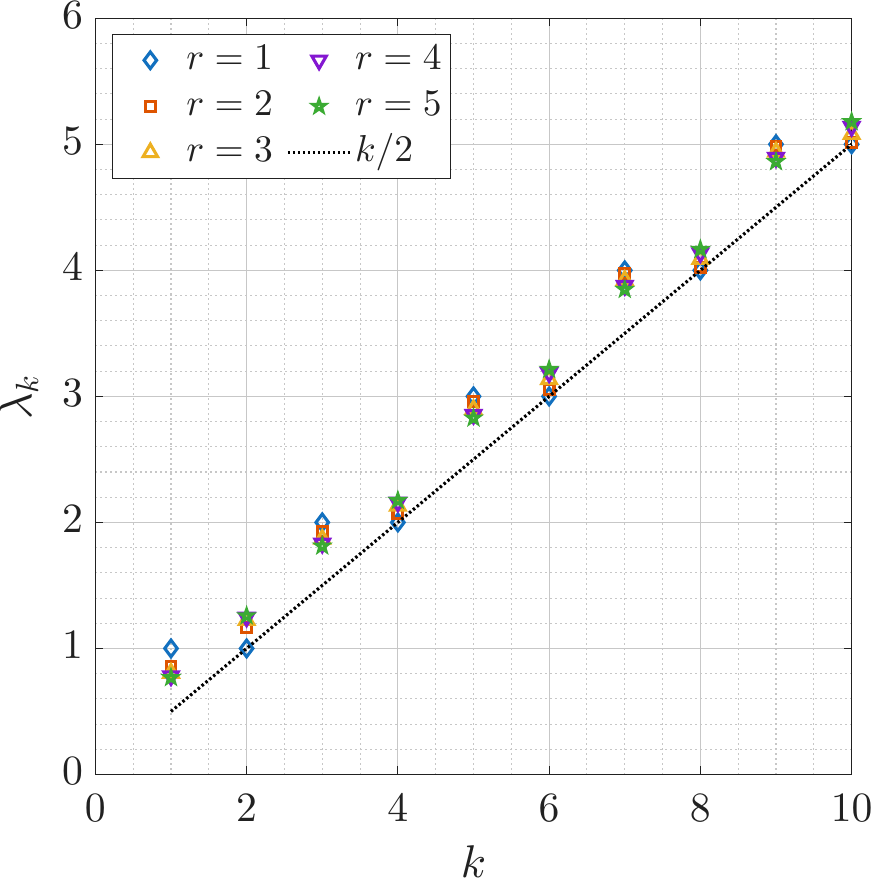}}
\hfill
\scalebox{0.35}{\includegraphics{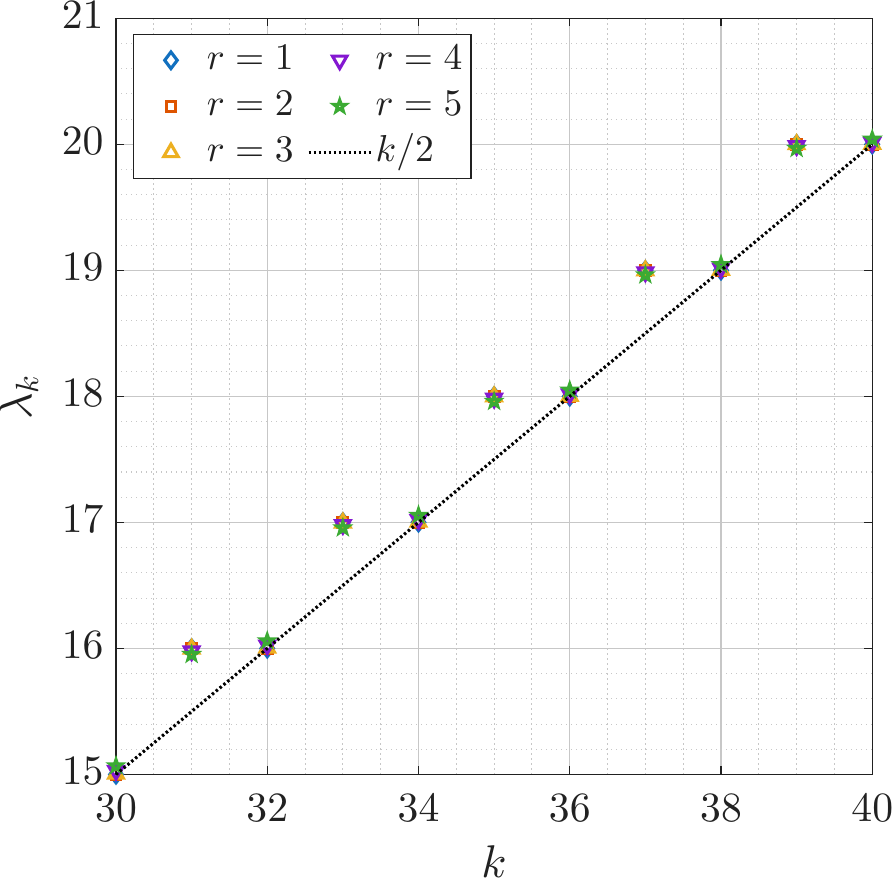}}
\hfill
\scalebox{0.35}{\includegraphics{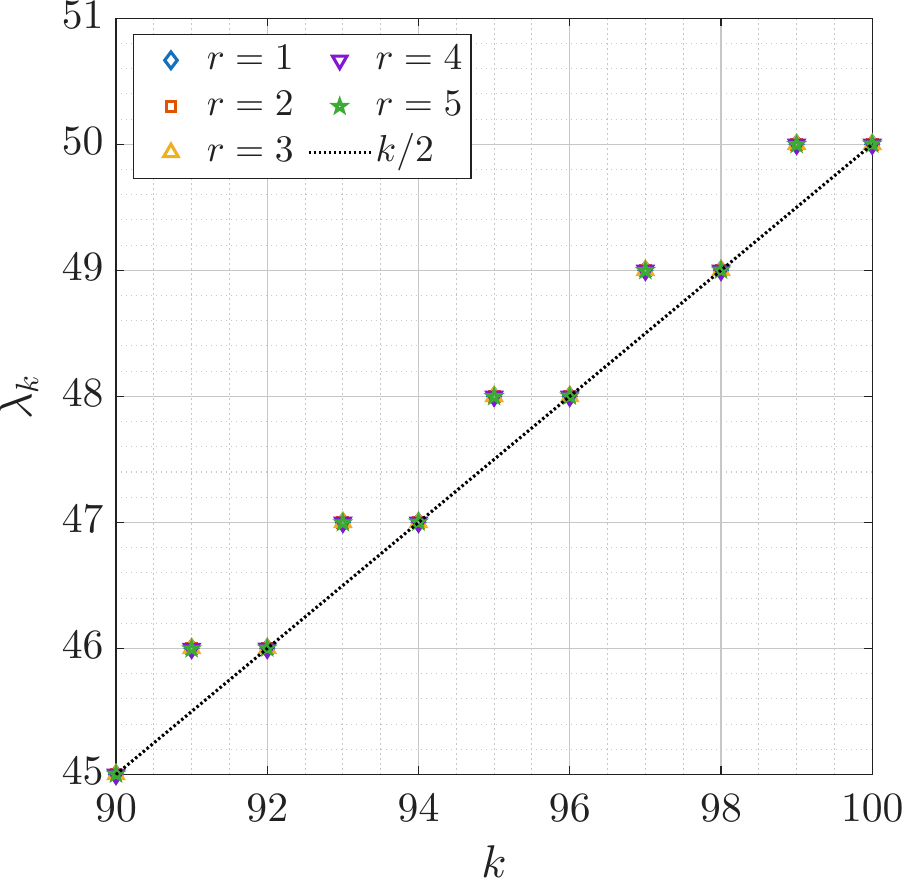}}
\hfill}
\caption{The eigenvalues $\lambda_k$ for several values of $k$ and $r$ for the bounded domain $G_1$ (first row) and the unbounded domain $G_2$ (second row) in Example~\ref{ex:ell}.}
	\label{fig:ex_ell_eig100}
\end{figure}

\begin{figure}[htb] % example_ellipse_r10.m
\centerline{
\hfill
\scalebox{0.35}{\includegraphics{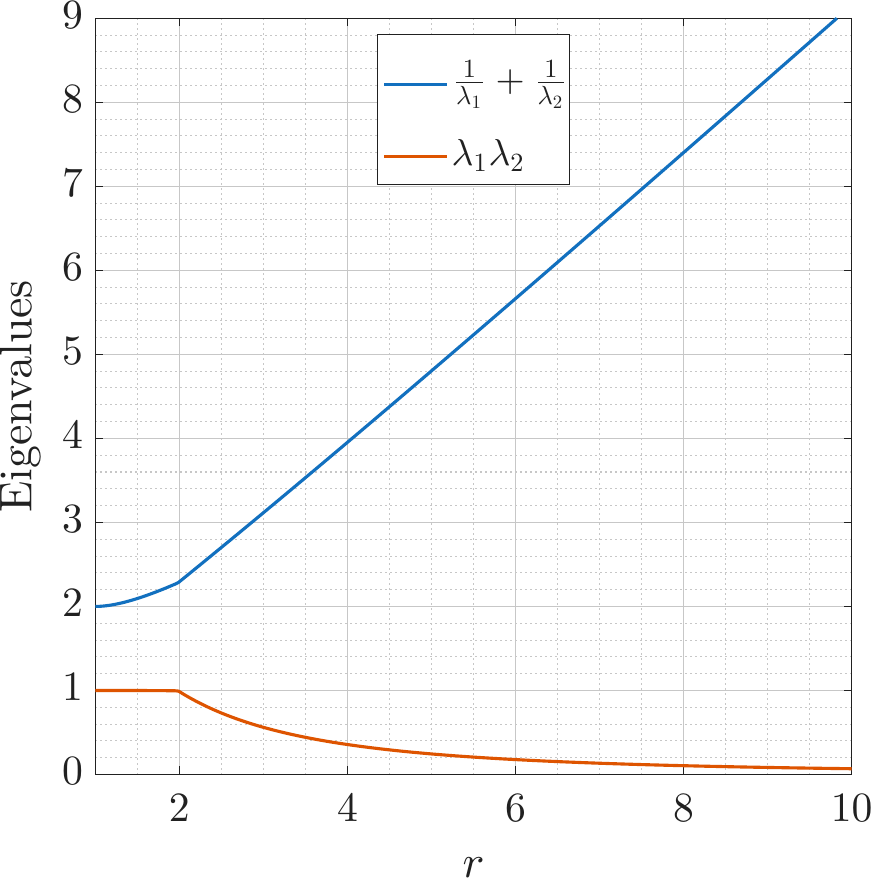}}
\hfill
\scalebox{0.35}{\includegraphics{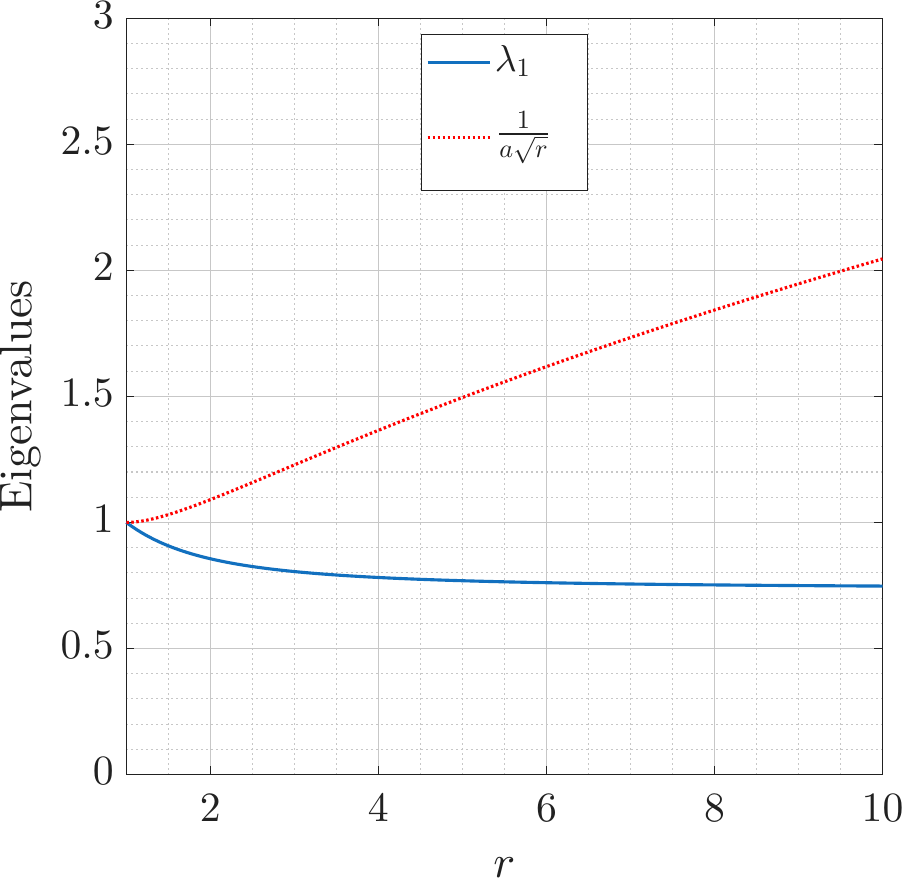}}
\hfill}
\caption{Verifying the inequalities~\eqref{eq:ineq-1} (left) and~\eqref{eq:ineq-2} (right) numerically for several values of $r$.}
	\label{fig:ex_ell_eig1}
\end{figure}

\begin{example}\label{ex:star2}
Let $\Gamma$ be the curve that is parametrized for the bounded domain $G_1$ by 
\begin{equation}
\eta(t) = a(1 + r \cos 2t)e^{\i t}, \quad 0\le t\le 2\pi, 
\end{equation}
and for the unbounded domain $G_2$ by 
\begin{equation}
\eta(t) = a(1 + r \cos 2t)e^{-\i t}, \quad 0\le t\le 2\pi, 
\end{equation}
with $a>0$ and $0\le r<1$.
As $r$ approaches $1$, the middle of the domain becomes pinched, and at $r=1$
the bounded domain $G_1$ splits into two subdomains.
\end{example}

We again choose $a$ so the length of $\Gamma$ is fixed, i.e., $2\pi=|\Gamma|=aI$ with 
\[
I= \int_0^{2\pi} \sqrt{4r^2 \sin^2 (2t) + (1+r\cos (2t))^2}dt.
\]
Therefore, for a given value of $r$, we choose $a=2\pi/I$, where $I$ is
approximated by the trapezoidal rule.
We study the effect of increasing $r$ on the Steklov eigenvalues and the
corresponding eigenfunctions for $0\le r <1$. We compute the first nonzero
eigenvalues, $0<\lambda_1 \leq \cdots \le \lambda_{10}$, with $n=2^{10}$ when
$0\leq r \leq 0.6$ and with $n=2^{11}$ when $0.6 < r < 1$ for both the bounded
and unbounded domains. These results are presented in
Figure~\ref{fig:ex_star2_eig10}. For selected values of $r$, we show the
eigenmodes of the first $9$ nonzero eigenvalues in Figure~\ref{fig:ex_star2_B}
for the bounded domain $G_1$ and in Figure~\ref{fig:ex_star2_U} for the
unbounded domain $G_2$.

\begin{figure}[htb] %example_star_r10.m, example_ellipse_U_r10.m
\centerline{
\hfill
\scalebox{0.35}{\includegraphics{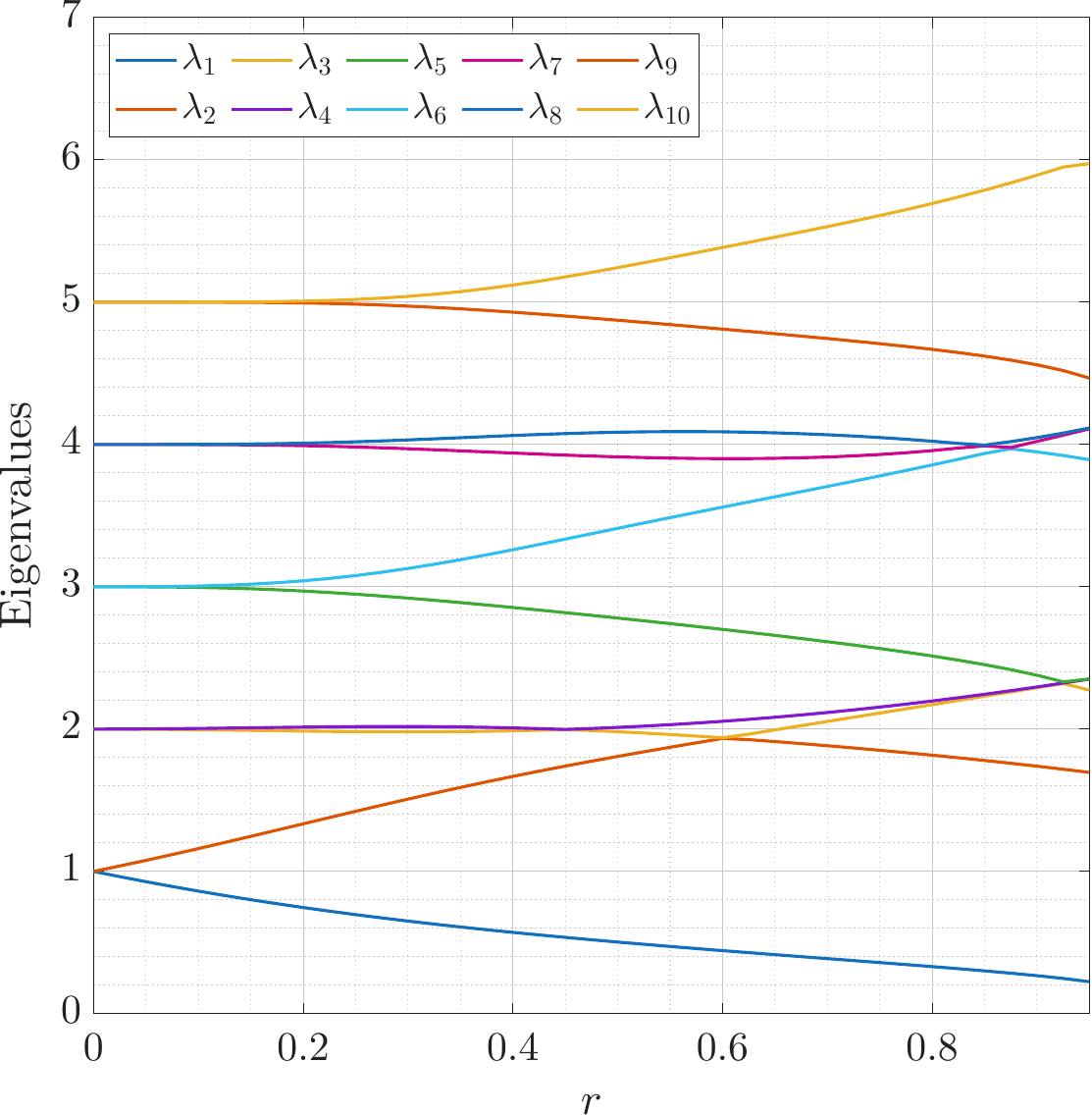}}
\hfill
\scalebox{0.35}{\includegraphics{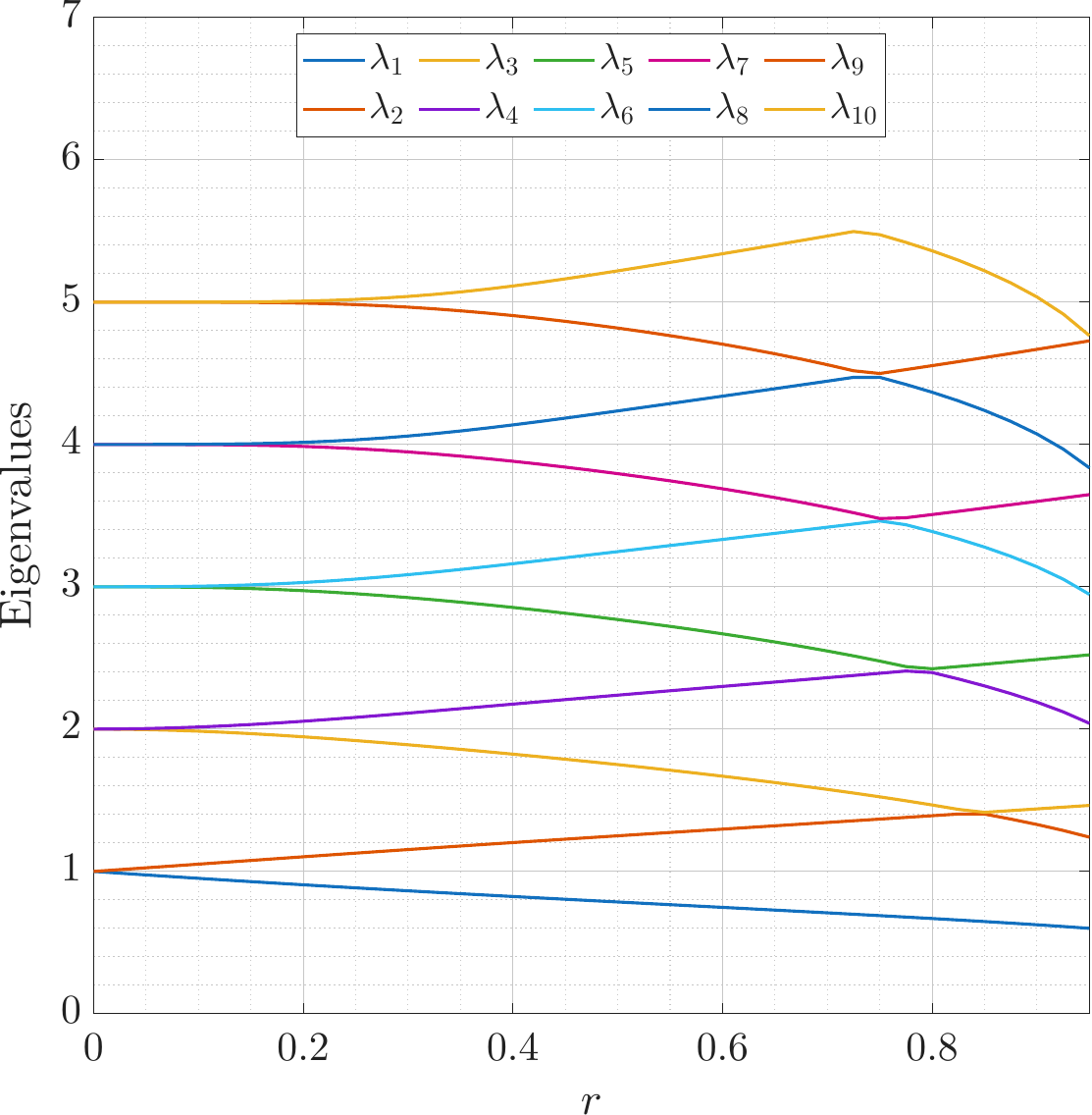}}
\hfill}
\caption{The first 10 nonzero eigenvalues as functions of $r$ for the bounded domain $G_1$ (left) and the unbounded domain $G_2$ (right) for Example~\ref{ex:star2}}
	\label{fig:ex_star2_eig10}
\end{figure}

\begin{figure}[htb] %
\centering
\subfloat[$r=0$]{
\scalebox{0.31}{\includegraphics{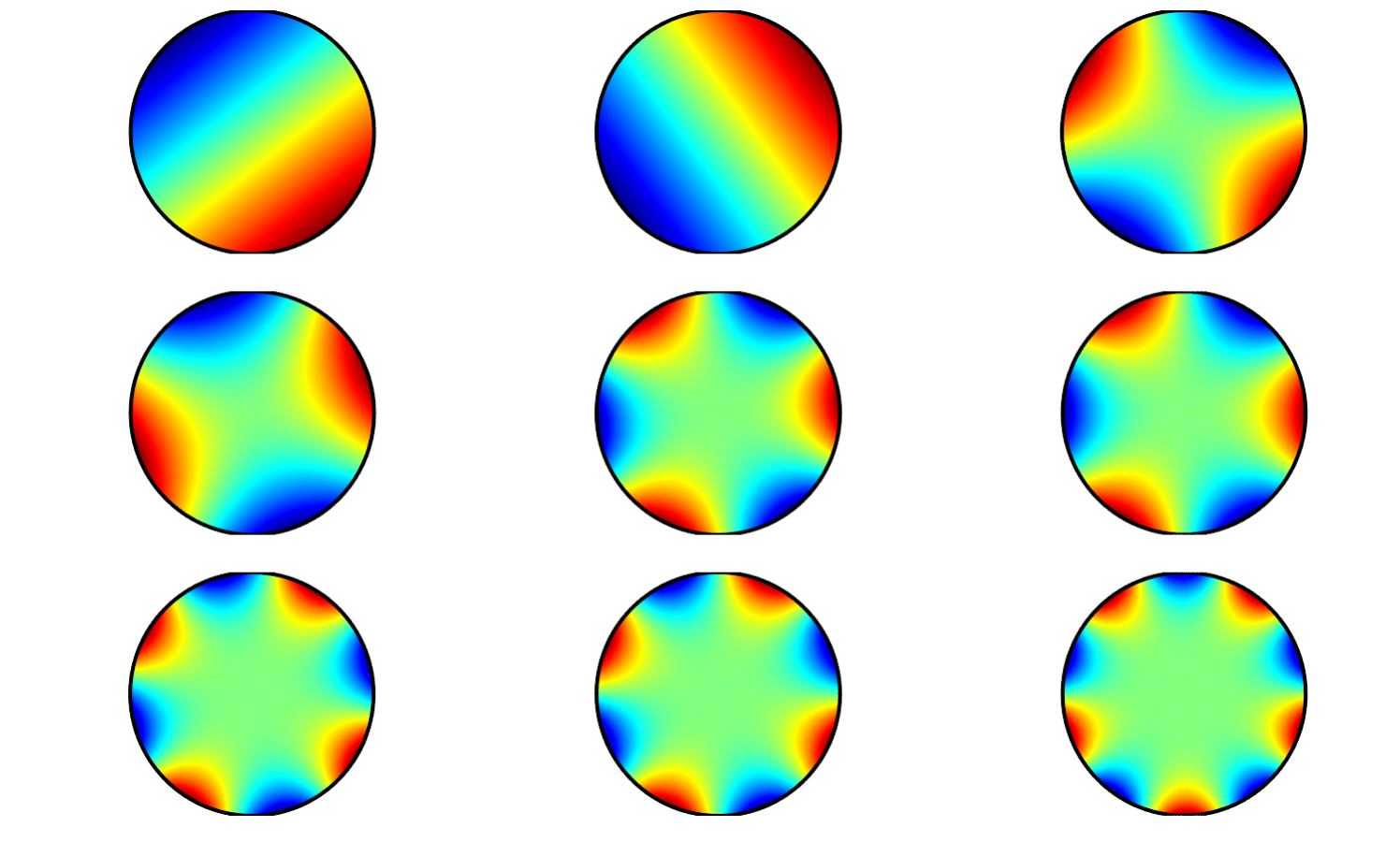}}
}
\subfloat[$r=0.5$]{
\scalebox{0.31}{\includegraphics{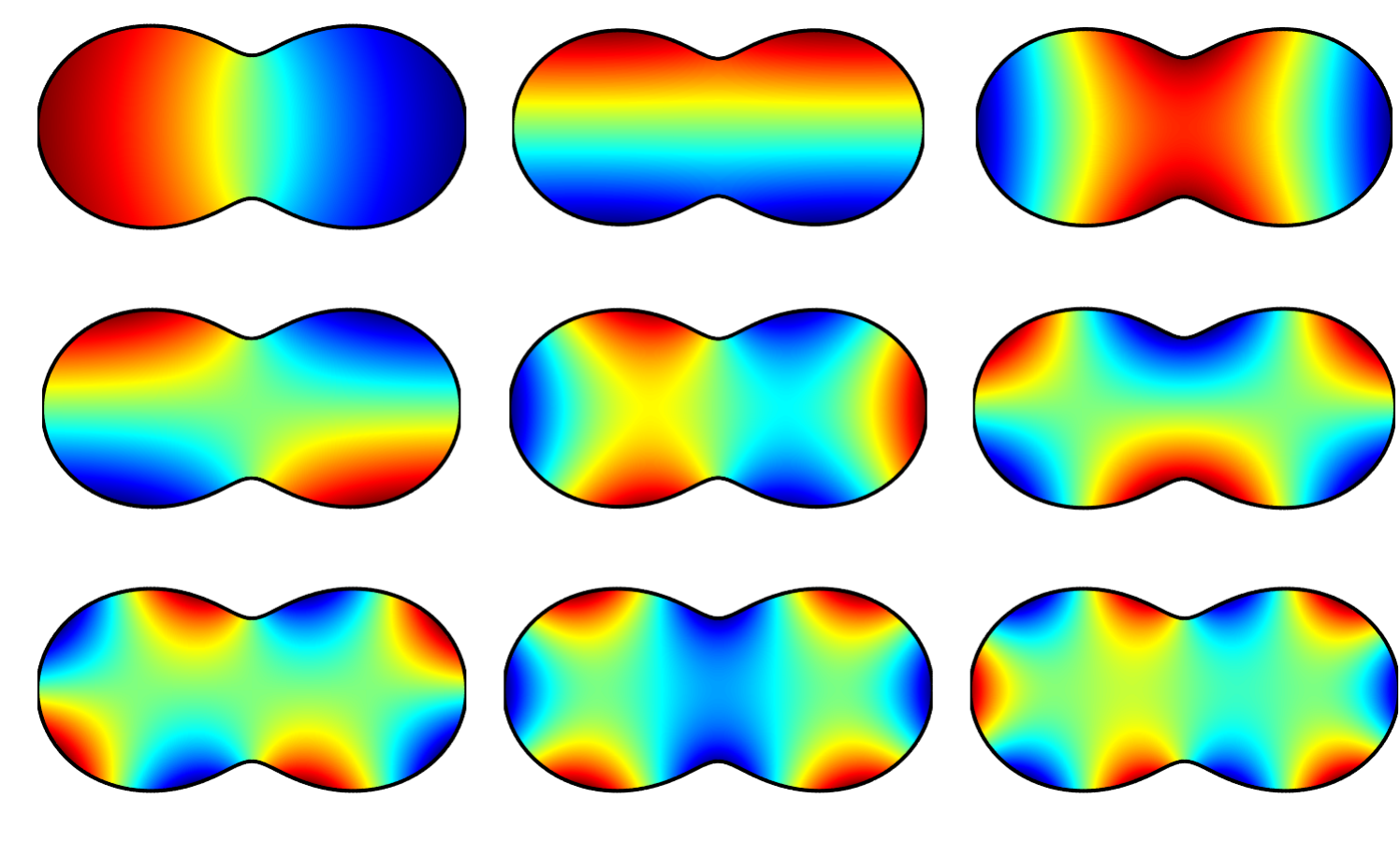}}
}\\
\subfloat[$r=0.75$]{
\scalebox{0.31}{\includegraphics{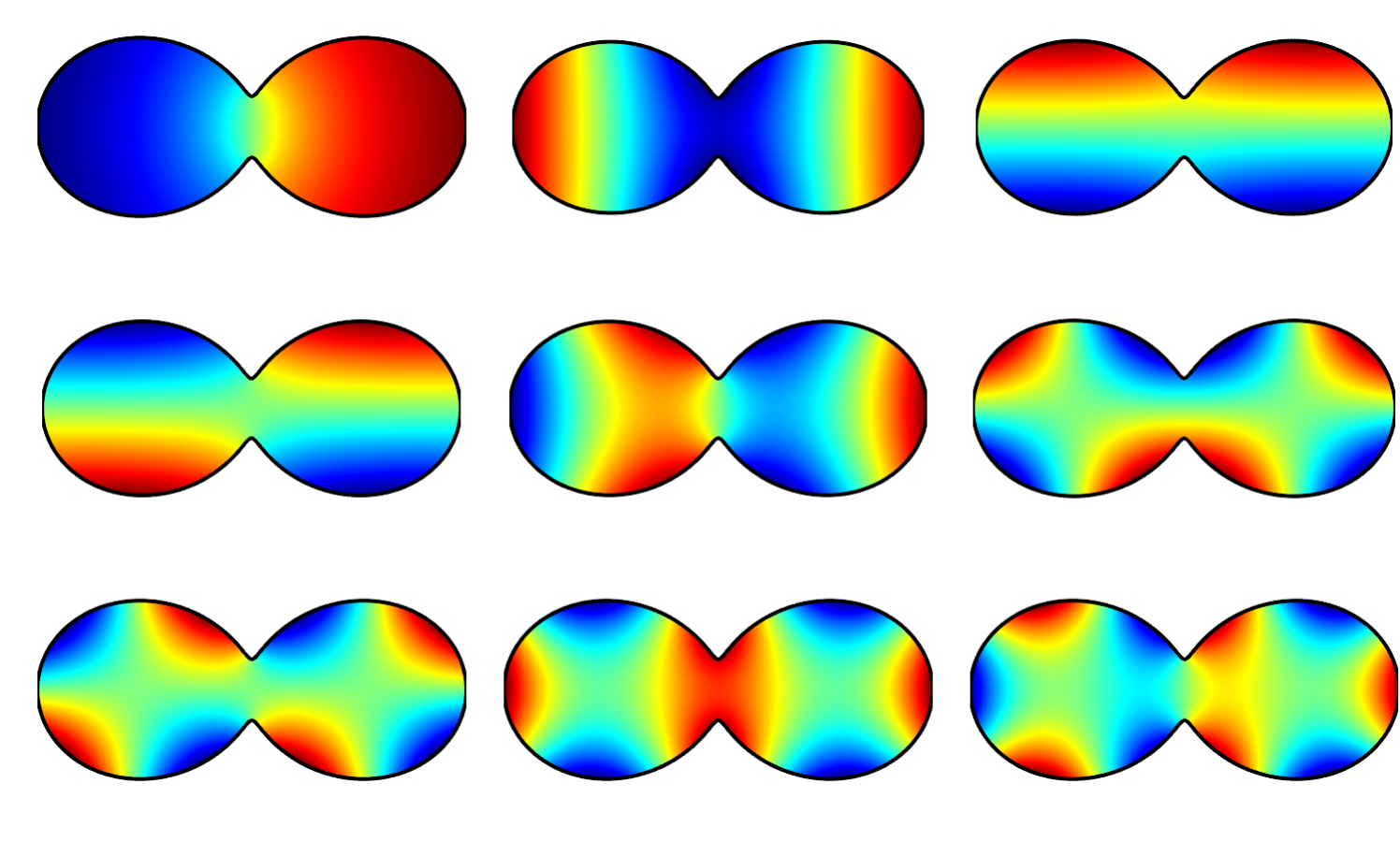}}
}
\subfloat[$r=0.95$]{
\scalebox{0.31}{\includegraphics{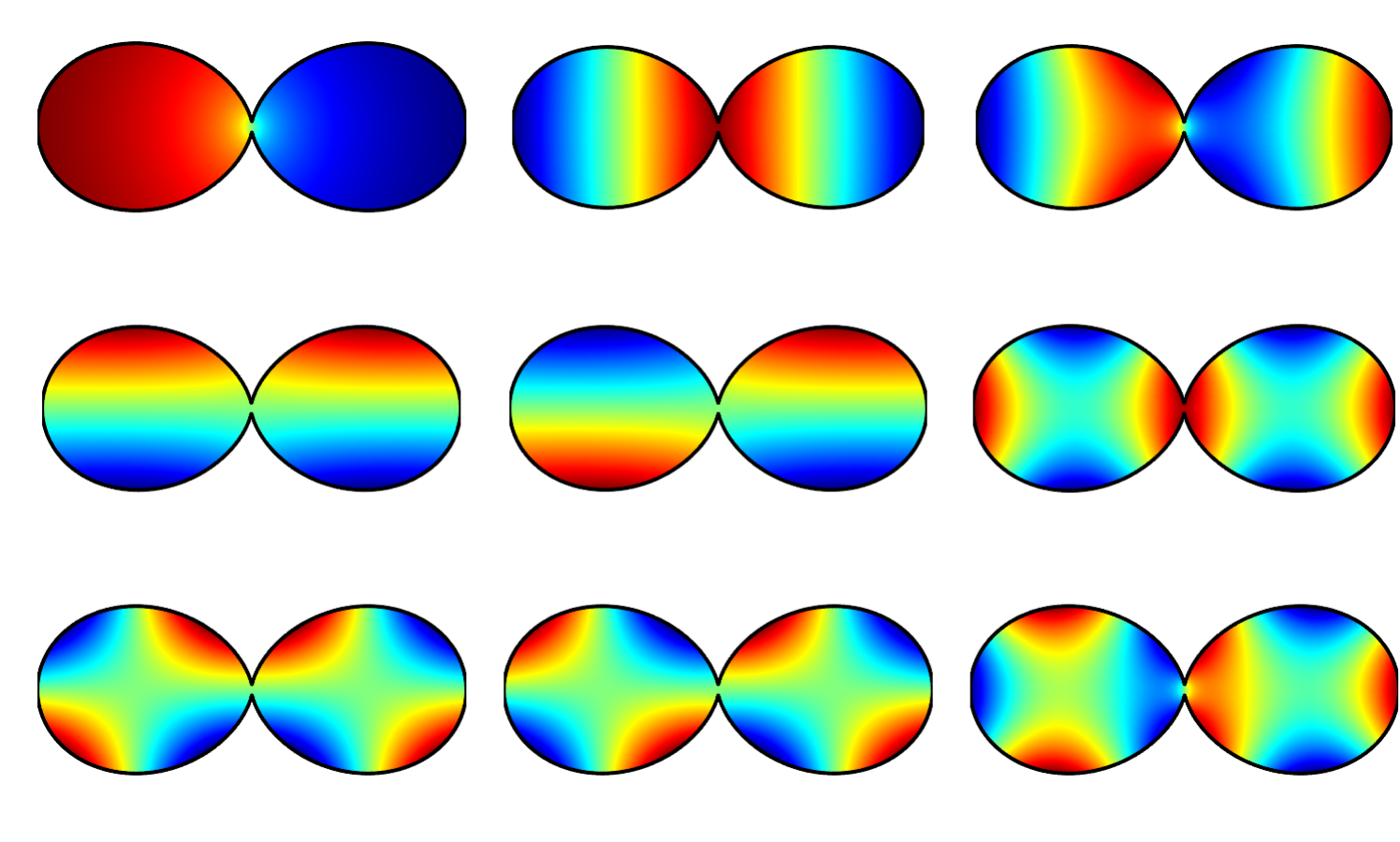}}
}
\caption{The eigenmodes of the first $9$ nonzero eigenvalues of the bounded domain $G_1$ for several values of $r$.}
	\label{fig:ex_star2_B}
\end{figure}

\begin{figure}[htb] %
\centering
\subfloat[$r=0$]{
\scalebox{0.44}{\includegraphics{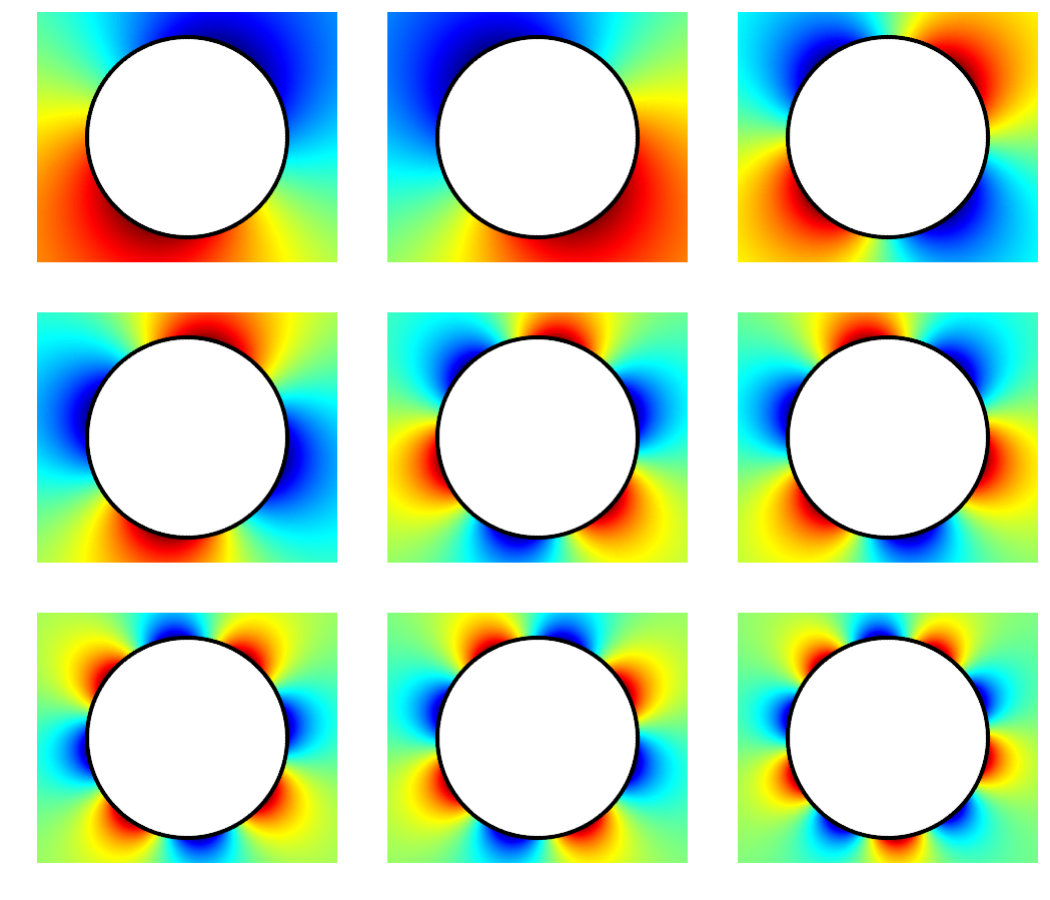}}
}
\subfloat[$r=0.5$]{
\scalebox{0.44}{\includegraphics{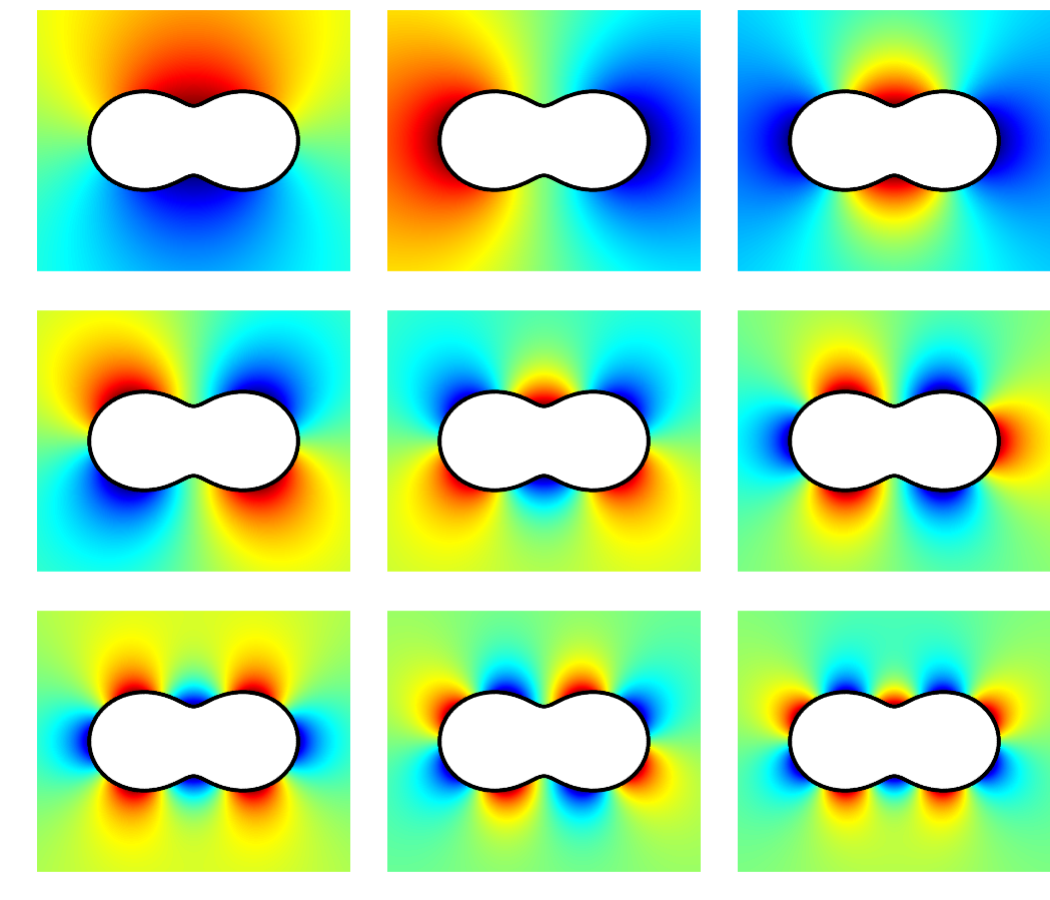}}
}\\
\subfloat[$r=0.75$]{
\scalebox{0.44}{\includegraphics{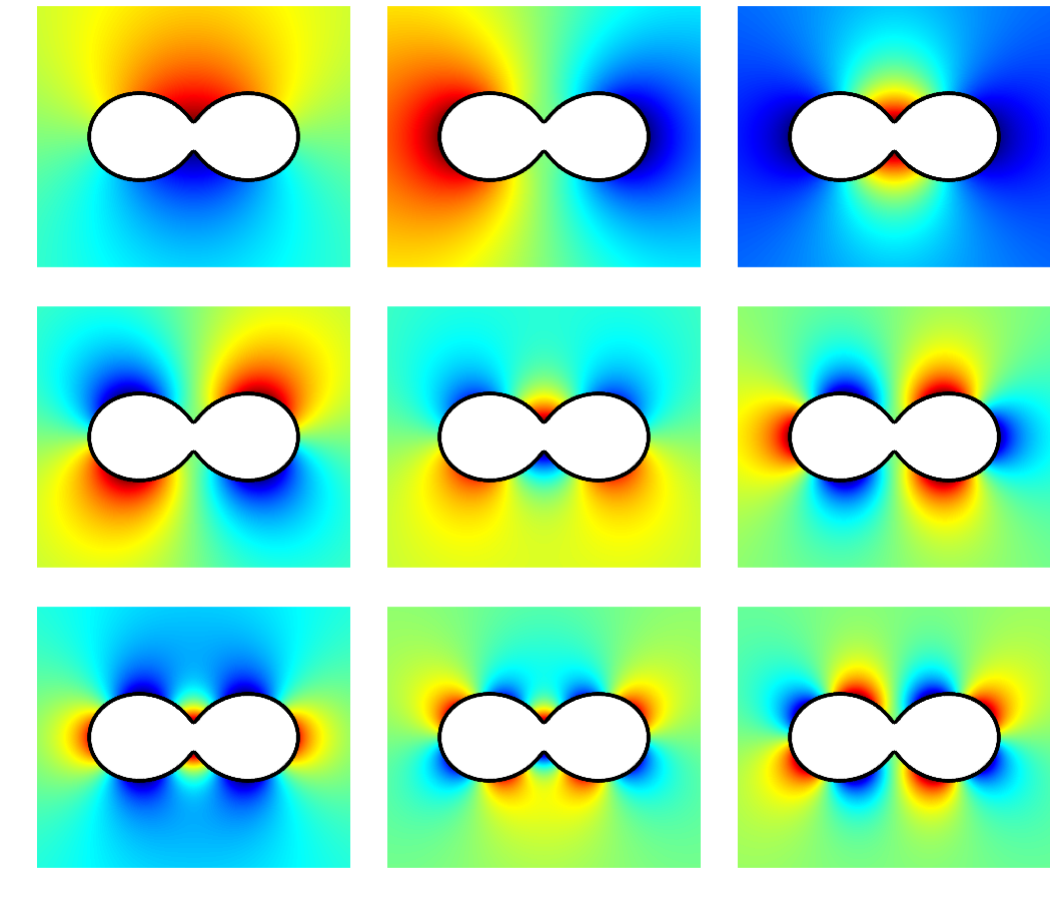}}
}
\subfloat[$r=0.95$]{
\scalebox{0.44}{\includegraphics{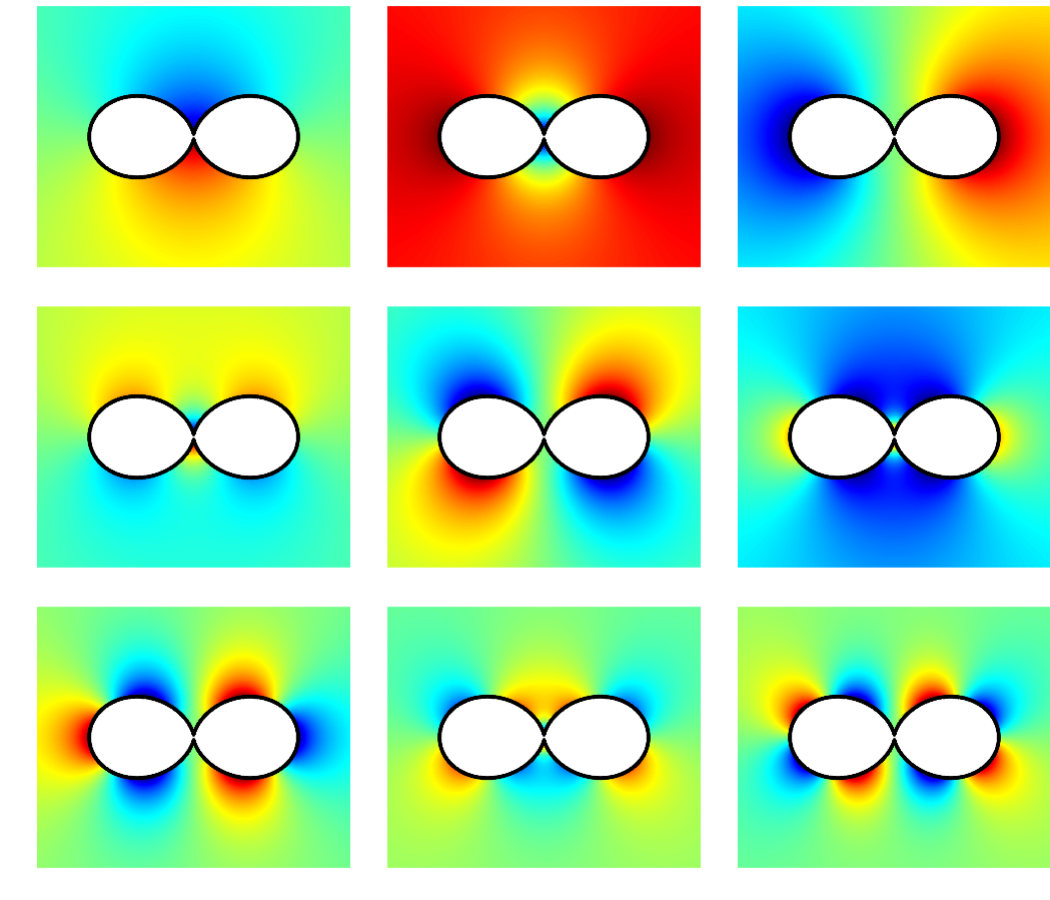}}
}
\caption{The eigenmodes of the first $9$ nonzero eigenvalues of the unbounded domain $G_2$ for several values of $r$.}
\label{fig:ex_star2_U}
\end{figure}

\subsection{On Computational Complexity}

The dominant computational costs arise from assembling the dense boundary
integral operators, solving the linear systems associated with
$(\bI-\bN)\mu=-\bM\gamma$, and repeatedly applying the resulting operator inside
the eigensolver. If the matrices are formed explicitly on an $n$-point grid,
then assembling the Nystr\"om matrices for $\bN$ and $\bM$ requires $O(n^2)$
work and storage. A direct factorization of the dense matrix $(I-B)$ has
$O(n^3)$ setup cost and $O(n^2)$ cost per solve, while the multiplication by the
Fourier differentiation operator can be implemented either through dense Fourier
matrices or, more efficiently, through FFT-based routines.

In the present MATLAB implementation, the numerical experiments suggest that the
number of {\tt eigs} iterations is relatively stable across the smooth examples
considered here; see Figures~\ref{fig:dom1_time}, \ref{fig:kite_time}, and
\ref{fig:ell_time}. Consequently, the increase in runtime is driven mainly by
the size of the dense boundary system rather than by a deterioration in the
iteration count. Geometry still enters indirectly through the required
resolution: elongated ellipses and nearly pinched star-like curves demand larger
values of $n$, and the conditioning plots show that the matrix $Q+I$ becomes
less favorable as the geometry becomes more challenging.

These observations indicate two natural directions for improvement. First, the
explicit use of dense Fourier matrices can be replaced by FFT-based
implementations. Second, the repeated solution of systems involving $(I-B)$
could be accelerated by fast direct solvers, preconditioning, or matrix
compression. It is also possible to use the fast multipole method (FMM) in solving linear systems involving the matrix $(I-B)$~\cite{Nas-E2,Nas-ETNA}. Such optimizations are not pursued here, since the goal of the
present work is to establish the boundary formulation and document its behavior
on smooth model geometries.

\subsection{Benchmark comparison with an \texorpdfstring{$hp$}{hp}-FEM solver}\label{sec:hpfemcomp}

As a final independent cross-validation of the proposed BIE method
results, we solve the Steklov eigenvalue problem on the benchmark
domain~$G_1$ from Example~\ref{ex:1} (see Figure~\ref{fig:domains} (left)) with
a high-order $hp$-finite element method (FEM). The standard weak formulation
seeks nontrivial pairs $(u,\lambda)\in H^1(G)\times\mathbb{R}$ satisfying
\begin{equation}\label{eq:hpfem}
\int_{G}\nabla u\cdot\nabla v\,dx
= \lambda\int_{\Gamma}u\,v\,ds
\qquad\forall\,v\in H^1(G).
\end{equation}
To obtain a positive-definite pencil suitable for the Arnoldi iteration we
reformulate~\eqref{eq:hpfem} as the generalized eigenvalue problem
$(K+M_{\partial})\mathbf{v}=\nu\,M_{\partial}\mathbf{v}$,
where $K$ is the stiffness matrix and $M_{\partial}$ is the boundary mass
matrix; the Steklov eigenvalue is then $\lambda=\nu-1$.

The boundary~$\partial G_1$ is discretized into $48$ curved segments that
resolve the exact geometry, and the interior mesh is generated by constrained
Delaunay triangulation. Keeping this mesh fixed, we vary the polynomial degree
$p$ of the element shape functions from~$2$ to~$12$ in order to study
$p$-convergence. The stiffness and boundary mass matrices are assembled in a
standard $hp$-FEM framework~\cite{SzaboBook}, and the resulting generalized
eigenvalue problems are solved with ARPACK.

Table~\ref{tab:hpfem} compares the first ten nonzero area-scaled
eigenvalues $\tilde\lambda_k=\lambda_k\sqrt{|G_1|}$ obtained by the
BIE method (Table~\ref{tab:ex1}, $n=2^{10}$) with those
produced by the $hp$-FEM solver at polynomial degree $p=8$ ($5\,185$
degrees of freedom). All ten eigenvalues agree to at least eight
significant digits, confirming that both discretizations resolve the same
spectral data.

\begin{table}[ht]
\caption{Area-scaled Steklov eigenvalues $\tilde\lambda_k$ for $G_1$:
BIE method ($n=2^{10}$) versus $hp$-FEM ($p=8$,
$48$~boundary segments, $5\,185$~DOFs).}\label{tab:hpfem}
\centering
\setlength{\tabcolsep}{12pt}
\begin{tabular}{cccc}
\hline\noalign{\smallskip}
  & BIE method & $hp$-FEM & Rel.\ error \\
\noalign{\smallskip}\hline\noalign{\smallskip}
 $\tilde\lambda_1$    & 1.61465185265077 & 1.61465185392749 & $7.9\times10^{-10}$ \\
 $\tilde\lambda_2$    & 1.61465185265086 & 1.61465185477199 & $1.3\times10^{-9\phantom{0}}$ \\
 $\tilde\lambda_3$    & 2.97737736702950 & 2.97737737349512 & $2.2\times10^{-9\phantom{0}}$ \\
 $\tilde\lambda_4$    & 2.97737736702974 & 2.97737737390843 & $2.3\times10^{-9\phantom{0}}$ \\
 $\tilde\lambda_5$    & 5.48337898612383 & 5.48337898917263 & $5.6\times10^{-10}$ \\
 $\tilde\lambda_6$    & 5.48337898612393 & 5.48337898942869 & $6.0\times10^{-10}$ \\
 $\tilde\lambda_7$    & 6.70773879741621 & 6.70773880665508 & $1.4\times10^{-9\phantom{0}}$ \\
 $\tilde\lambda_8$    & 6.70773879741642 & 6.70773881210991 & $2.2\times10^{-9\phantom{0}}$ \\
 $\tilde\lambda_9$    & 7.65773980917837 & 7.65773983329138 & $3.1\times10^{-9\phantom{0}}$ \\
 $\tilde\lambda_{10}$ & 9.01958292273808 & 9.01958292298226 & $2.7\times10^{-11}$ \\
\noalign{\smallskip}\hline
\end{tabular}
\end{table}

Table~\ref{tab:pconv} records the $p$-convergence of
$\tilde\lambda_1$ on the fixed 48-segment mesh as the polynomial degree
increases from~$2$ to~$12$. The relative error decreases from
$2.3\times10^{-4}$ at $p=2$ to $6.8\times10^{-13}$ at $p=12$,
exhibiting the exponential (spectral) convergence rate expected for
$hp$-FEM on smooth domains~\cite{SzaboBook}. At $p=12$ the $hp$-FEM
eigenvalue carries roughly thirteen correct digits, providing an
additional reference value against which the BIE approximation can be assessed.

The comparison in this subsection is intended as an independent
benchmark validation on the geometry $G_1$, rather than as a full efficiency
study between boundary-only and volumetric discretizations. A broader comparison
would include accuracy-versus-cost curves for both methods and, for near-multiple
eigenvalues, a comparison of the associated eigenspaces rather than only the
ordered eigenvalues.

\begin{table}[ht]
\caption{$p$-convergence of the first area-scaled Steklov eigenvalue
$\tilde\lambda_1$ for $G_1$ on a fixed mesh of $48$~boundary
segments.}\label{tab:pconv}
\centering
\setlength{\tabcolsep}{12pt}
\begin{tabular}{cccc}
\hline\noalign{\smallskip}
 $p$ & DOFs & $\tilde\lambda_1$ & Rel.\ error \\
\noalign{\smallskip}\hline\noalign{\smallskip}
  2 &    361  & 1.61502095635610 & $2.3\times10^{-4}$  \\
  4 &  1\,345 & 1.61465498408775 & $1.9\times10^{-6}$  \\
  6 &  2\,953 & 1.61465190864775 & $3.5\times10^{-8}$  \\
  8 &  5\,185 & 1.61465185392747 & $7.9\times10^{-10}$ \\
 10 &  8\,041 & 1.61465185268465 & $2.1\times10^{-11}$ \\
 12 & 11\,521 & 1.61465185265191 & $6.8\times10^{-13}$ \\
\noalign{\smallskip}\hline
\end{tabular}
\end{table}

The close agreement of two fundamentally different
discretizations---boundary-only versus volume-based---on the same test
domain provides an independent validation of the BIE results on this benchmark
geometry and supports the high accuracy of the reported eigenvalues.
In this smooth-boundary setting, the additional complexity of volumetric meshing tilts the balance in favor of the boundary-only BIE formulation.

\section{Conclusion}\label{sec:con}

This paper asked whether the Steklov spectrum of smooth simply connected planar
domains can be computed accurately by a boundary-only formulation built from
harmonic conjugation. The answer provided by the present analysis and
computations is affirmative. For the unit disk, the method recovers the
classical Steklov structure through the conjugation operator $\bK$. For general
bounded and unbounded simply connected domains, the generalized conjugation
operator $\bE$ leads to a concrete algebraic eigenvalue problem for the boundary
traces of Steklov eigenfunctions.

The numerical examples show that the method performs well on a range of smooth
interior and exterior geometries and captures both benchmark spectra and
parameter-dependent spectral behavior. In particular, the examples illustrate
how the Steklov spectrum changes under smooth deformations of the boundary, and
they show that the same framework can be used for bounded and unbounded
problems.

Several natural directions remain open. On the analytical side, it would be
useful to complement the present formulation with a sharper convergence analysis
for the eigenvalue discretization. On the computational side, faster linear
algebra and compression strategies would make the method more effective for
larger systems and more demanding geometries. 
These questions provide a natural next step for extending the present study to multiply connected domains. For multiply connected domains of high connectivity, the size of the matrices in the algebraic eigenvalue problem will be too large to be handled on a standard computer. Hence, using fast methods is unavoidable. In such a case, using the FFT and FMM will reduce the computational cost of the method. It will also reduce the memory requirement as the explicit forms of these matrices and saving them in the memory will not be required. 

%%%%%%%%%%%%%%%%%%%%%%%%%%%%%%
\clearpage
\appendix

\section{The matrix \texorpdfstring{$D$}{D}}

For a given even integer $n$, we define the equidistant points $t_j = (j-1)\frac{2\pi}{n}$ for $j=1,2,\ldots,n$. Let $\gamma(t)$ be a given real Hölder continuous $2\pi$-periodic function. We approximate $\gamma(t)$ by a trigonometric interpolating polynomial
\begin{equation}\label{eq:gamma}
    \gamma (t) \approx a_0 + \sum_{k=1}^{\frac{n}{2}} a_k e^{\i kt} + \sum_{k=-\frac{n}{2}+1}^{-1} a_ke^{\i kt}
\end{equation}
with unknown constants $a_{-\frac{n}{2}+1}, \ldots, a_{-1},a_0, a_1, .., a_{\frac{n}{2}}$.
These unknown constants are chosen such that
\begin{equation}\label{eq:gamma2}
\gamma (t_j) = a_0 + \sum_{k=1}^{\frac{n}{2}} a_k e^{\i kt_j} + \sum_{k=-\frac{n}{2}+1}^{-1} a_ke^{\i kt_j}, 
\quad j=1,\ldots,n.
\end{equation}
Since $e^{\i n t_j}=1$ for $j=1,...n$, the second sum in~\eqref{eq:gamma2} can be rewritten as
\[
    \sum^{-1}_{k=-\frac{n}{2}+1} a_ke^{\i kt_j} 
		= \sum^{n-1}_{k=\frac{n}{2}+1} a_{k-n}e^{\i (k-n)t_j} 
		= \sum^{n-1}_{k=\frac{n}{2}+1} a_{k-n}e^{\i kt_j},
\]
which implies that equation \eqref{eq:gamma2} can be rewritten as
\begin{equation}\label{eq:gamma2b}
\gamma (t_j) = a_0 + \sum_{k=1}^{\frac{n}{2}} a_k e^{\i kt_j} + \sum^{n-1}_{k=\frac{n}{2}+1} a_{k-n}e^{\i kt_j}, 
\quad j=1,\ldots,n.
\end{equation}

Define the $n$ constants $b_k$, $k=1,\ldots,n$, by $b_1=\overline{a_0}$, $b_{k} = \overline{a_{k-1}}$ for $k=2,\ldots,n/2+1$, and $b_{k} = \overline{a_{k-1-n}}$ for $k=n/2+2,\ldots,n$.
Since $\gamma(t)$ is a real function, then
\[
\gamma(t_j) = \sum_{k=1}^{n} \overline{b_k} e^{\i (k-1)t_j} 
= \sum_{k=1}^{n} b_{k} e^{-\i (k-1)t_j},
\]
which can be rewritten as
\begin{equation}\label{eq:gamma_C3}
    \gamma (t_j) =  \sum_{k=1}^{n} b_k e^{-\i (k-1)(j-1)\frac{2\pi}{n}} 
		= \sum_{k=1}^{n} b_k \omega_n^{(k-1)(j-1)}, \quad j=1,\ldots,n,
\end{equation}
where $\omega_n = e^{-\i \frac{2\pi}{n}}$.
Equation~\eqref{eq:gamma_C3} can be written in matrix form as $\gamma(\bt)=F\bvb$ where $F$ is the Fourier matrix with entries $(F)_{kj}=\omega_n^{(k-1)(j-1)}$, $k,j=1,\ldots,n$, and $\bvb=(b_1,\ldots,b_n)^T$. Hence,
\begin{equation}\label{eq: matrix_C}
\bvb = F^{-1}\gamma(\bt).
\end{equation}
By obtaining the vector $\bvb$, we obtain the unknown constants
\[
a_{-\frac{n}{2}+1}, \ldots, a_{-1},a_0, a_1, \ldots, a_{\frac{n}{2}}.
\]

For the function $\gamma(t)$ in~\eqref{eq:gamma}, the derivative $\bD\gamma(t)$ is approximated by the trigonometric interpolating polynomial
\begin{equation}\label{eq:gamma'}
    \bD\gamma(t) = \gamma' (t) \approx  \sum_{k=1}^{\frac{n}{2}-1} a_k \i ke^{\i kt} + \sum_{k=-\frac{n}{2}+1}^{-1} a_k \i ke^{\i kt}.
\end{equation}
For $t=t_j$, $j=1,\ldots,n$, we have
\begin{equation}\label{eq:gamma-p}
    \bD\gamma(t_j) = \sum_{k=1}^{\frac{n}{2}-1} a_k \i ke^{\i kt_j} 
		+ \sum_{k=-\frac{n}{2}+1}^{-1} a_k \i ke^{\i kt_j}.
\end{equation}
The second sum term of equation \eqref{eq:gamma-p} can be written as
\[
    \sum^{-1}_{k=-\frac{n}{2}+1} \i ka_ke^{\i kt_j} 
		= \sum^{n-1}_{k=\frac{n}{2}+1} \i (k-n)a_{k-n}e^{\i (k-n)t_j} 
		= \sum^{n-1}_{k=\frac{n}{2}+1} \i (k-n)a_{k-n}e^{\i kt_j}, 
\]
and then~\eqref{eq:gamma-p} become
\begin{equation}\label{eq: D_gamma1}
\bD \gamma (t_j) =  \sum_{k=1}^{\frac{n}{2}-1} \i ka_k e^{\i kt_j} + \sum^{n-1}_{k=\frac{n}{2}+1} \i(k-n)a_{k-n}e^{\i kt_j},
\quad j=1,\ldots,n.
\end{equation}

Define the $n$ constants $c_k$, $k=1,\ldots,n$, by $c_1=0$, $c_{k} = \overline{\i (k-1)a_{k-1}}$ for $k=2,\ldots,n/2$, $c_{n/2+1}=0$, and $c_{k} = \overline{\i(k-1-n)a_{k-1-n}}$ for $k=n/2+2,\ldots,n$.
Since $\bD\gamma(t)$ is a real function, then~\eqref{eq: D_gamma1} can be written as
\[
     \bD \gamma (t_j) =  \sum_{k=1}^{n} \overline{c_{k}}e^{\i (k-1)t_j} 
		=  \sum_{k=1}^{n} c_{k}e^{-\i (k-1)t_j}
		=  \sum_{k=1}^{n} c_{k}\omega_n^{(k-1)(j-1)},
		\quad j=1,\ldots,n,
\]
which can be written in matrix form as $\bD\gamma(\bt)=F\bc$ where $\bc=(c_1,\ldots,c_n)^T$.
Note that $c_1=0\cdot b_1$, $c_{k} = -\i (k-1)b_{k-1}$ for $k=2,\ldots,n/2$, $c_{n/2+1}=0\cdot b_{n/2+1}$, and $c_{k} = -\i(k-1-n)b_{k-1-n}$ for $k=n/2+2,\ldots,n$. 
Hence $\bc=\hat{W}\bvb$ with the $n\times n$ matrix 
\[
\hat W= -\i \left[\begin{array}{cccc}
	~~0~~   & ~~0~~    &  ~~0~~   & 0  \\
	 0  & R  & 0  & 0  \\
	0    &0 & 0 & 0 \\
    0  & 0 & 0 &  -JRJ    \\
\end{array}\right]
\]
with the $(n/2-1)\times(n/2-1)$ matrices 
\[
R=\left[\begin{array}{cccc}
	~~1~~   & ~~0~~    &\cdots &0     \\
	0   & 2    &\cdots &0     \\
	\vdots     &\vdots      &\ddots &\vdots   \\
	0          & 0   &\cdots &n/2-1   \\
\end{array}\right], \quad
J=\left[\begin{array}{cccc}
~0~      &\cdots &~0~  &~1~   \\
0      &\cdots &1  &0     \\
\vdots        &\iddots  &\vdots    &\vdots   \\
1             &\cdots &0 &0   \\
\end{array}\right].
\]
Thus, $\bD\gamma(\bt)=F\bc=F\hat W\bvb=F\hat WF^{-1}\gamma(\bt)$. 
Since $F^{-1} = (1/n)F^\ast$ where $F^\ast$ is the Hermitian transpose of $F$, we have $\bD\gamma(\bt)=FWF^\ast\gamma(\bt)$ where the matrix $W$ is defined by
\[
W = \frac{1}{n}\hat W.
\] 
Thus the differentiation operator $\bD$ is discretized by the matrix
\begin{equation}
D = F WF^\ast.
\end{equation}

\bibliographystyle{plain}

\end{document}